\documentclass[12pt,reqno]{amsproc}

\usepackage{amsmath}
\usepackage{amssymb}
\usepackage{graphicx}
\usepackage{mathrsfs}
\usepackage{color}
\usepackage{comment}
\usepackage{hyperref}
\usepackage{enumitem}
\usepackage{bm}
\usepackage[margin=1.25in]{geometry}

\numberwithin{equation}{section}

\newtheorem{prop}{Proposition}[section]
\newtheorem{theo}[prop]{Theorem}
\newtheorem*{theo*}{Theorem}
\newtheorem{lemm}[prop]{Lemma}
\newtheorem{coro}[prop]{Corollary}

\newtheorem{claim}{Claim}
\newtheorem{defi}[prop]{Definition}
\newtheorem{rema}[prop]{Remark}

\theoremstyle{definition}

\newtheorem{prob}[prop]{Problem}

\newtheorem{quest}[prop]{Question}

\newcommand{\DD}{\mathbf{D}}

\newcommand{\KK}{\mathbf{K}}

\newcommand{\NN}{\mathbf{N}}

\newcommand{\PP}{\mathbf{P}}

\newcommand{\RR}{\mathbf{R}}
\renewcommand{\SS}{\mathbf{S}}
\newcommand{\TT}{\mathbf{T}}

\newcommand{\cE}{\mathcal E}

\newcommand{\cK}{\mathcal K}

\newcommand{\cM}{\mathcal M}

\newcommand{\cR}{\mathcal R}

\newcommand{\cV}{\mathcal V}

\newcommand{\sM}{\mathscr{M}}

\DeclareMathOperator{\tr}{tr}
\DeclareMathOperator{\Tr}{Tr}

\DeclareMathOperator{\graph}{graph}

\DeclareMathOperator{\spt}{spt}

\DeclareMathOperator{\loc}{loc}

\DeclareMathOperator{\vol}{vol}

\DeclareMathOperator{\area}{area}

\DeclareMathOperator{\Ric}{Ric}

\DeclareMathOperator{\Div}{div}

\DeclareMathOperator{\sff}{\mathrm{I\!I}}

\newcommand{\eps}{\varepsilon}
\newcommand{\restr}{\mathbin{\vrule height 1.6ex depth 0pt width
0.13ex\vrule height 0.13ex depth 0pt width 1.3ex}}

\numberwithin{equation}{section}

\begin{document}

\title[Metrics with $\lambda_1 \geq 0$]{Metrics with $\lambda_1(-\Delta + k R) \geq 0$ and flexibility in the Riemannian Penrose Inequality}

\author{Chao Li}
\address{Courant Institute, New York University, New York, NY 10012, USA}
\email{chaoli@nyu.edu}

\author{Christos Mantoulidis}
\address{Department of Mathematics, Rice University, Houston, TX 77005, USA}
\email{christos.mantoulidis@rice.edu}

\begin{abstract}
	On a closed manifold, consider the space of all Riemannian metrics for which $-\Delta + kR$ is positive (nonnegative) definite, where $k > 0$ and $R$ is the scalar curvature. This spectral generalization of positive (nonnegative) scalar curvature arises naturally for different values of $k$ in the study of scalar curvature via minimal hypersurfaces, the Yamabe problem, and Perelman's Ricci flow with surgery. When $k=1/2$, the space models apparent horizons in time-symmetric initial data to the Einstein equations. We study these spaces in unison and generalize Cod\'a Marques's path-connectedness theorem. Applying this with $k=1/2$, we compute the Bartnik mass of 3-dimensional apparent horizons and the Bartnik--Bray mass of their outer-minimizing generalizations in all dimensions. Our methods also yield efficient constructions for the scalar-nonnegative fill-in problem.
\end{abstract}

\maketitle


\section{Introduction} \label{sec:introduction}

\subsection{The spaces}

In all that follows, $M$ denotes a closed $n$-manifold and $\operatorname{Met}(M)$ denotes the space of smooth Riemannian metrics on $M$. We emphasize that, unless stated otherwise, our manifolds are not assumed to be connected.

\begin{defi} \label{defi:m.k.plus}
	For $k \in (0, \infty)$, we define 
\begin{equation} \label{eq:mbar.k.plus}
	\sM^{\geq 0}_{k}(M) := \{ g \in \operatorname{Met}(M) : \lambda_1(-\Delta_g + k R_g) \geq 0 \},
\end{equation}
where $\lambda_1(-\Delta_g + k R_g)$ is the first eigenvalue of the operator $-\Delta_g + kR_g$ on $M$, and $R_g$ is the scalar curvature of $g$. We also define
\begin{equation} \label{eq:mbar.infty.plus}
	\sM^{\geq 0}_{\infty}(M) := \{ g \in \operatorname{Met}(M) : R_g \geq 0 \}.
\end{equation}
Finally, we define $\sM^{>0}_{k}(M)$, $k \in (0, \infty]$, as above with all ``$\geq$'' replaced by ``$>$.''
\end{defi}

These spaces are not generally encountered in the literature in this level of generality, so some remarks are in order about their actual geometric significance. First, and crucially, these spaces of metrics are closed under scaling and diffeomorphisms. Second, $\sM^{>0}_{k}(M)$ and $\sM^{\geq 0}_{k}(M)$ are descending filtrations in the space of metrics on $M$, i.e., for $0 < k < k' \leq \infty$, 
\begin{equation} \label{eq:m.k.plus.inclusion}
	\begin{array}{ccc}
		\sM^{>0}_{k'}(M) & \subset & \sM^{\geq 0}_{k'}(M) \\
		\cap & & \cap \\
		\sM^{>0}_{k}(M) & \subset & \sM^{\geq 0}_{k}(M).
	\end{array}
\end{equation}
Their geometric interest is due to:
\begin{itemize}
	\item When $k=\infty$, $\sM^{>0}_{k}(M)$ and $\sM^{\geq 0}_k(M)$ denote the sets of metrics on $M$ with, respectively, everywhere positive or everywhere nonnegative scalar curvature. The study of these spaces (even determining necessary and sufficient conditions for their non-emptiness) goes back several decades and is still active. 
	\item When $k=\tfrac12$, $\sM^{>0}_{k}(M)$ and $\sM^{\geq 0}_{k}(M)$ denote the sets of metrics induced on $M$ if $M$ occurs as a two-sided stable minimal hypersurface in an ambient manifold with, respectively, everywhere positive or  everywhere nonnegative scalar curvature. See Lemma \ref{lemm:m.12.plus.equivalences}. 
	\item When $k=\tfrac14$, $\sM^{>0}_{k}(M)$ appears crucially in Perelman's work on 3-dimensional Ricci flow with surgery \cite{Perelman1, Perelman2}. 
	\item When $k=\tfrac{n-2}{4(n-1)}$ and $n = \dim M \geq 3$, $\sM^{>0}_{k}(M)$ and $\sM^{\geq 0}_{k}(M)$ denote sets of metrics that arise in the Yamabe problem, specifically, the sets of metrics on $M$ that are, respectively, conformal to some metric with everywhere positive or everywhere nonnegative scalar curvature. See Lemma \ref{lemm:m.k.plus.equivalences}. 
\end{itemize}
In this paper, we put these spaces on common footing by incorporating them into a single filtration that interpolates between them.

\subsection{Tools}

Our starting point is a generalization of the theorem of Cod\'a Marques \cite{Marques:deforming.psc} regarding closed manifolds $M$ that carry metrics with positive scalar curvature, i.e., with $\sM^{>0}_\infty(M) \neq \emptyset$. Such manifolds will be referred to as ``topologically PSC.''\footnote{In the literature, they are sometimes referred to as having positive Yamabe invariant.} For closed orientable topologically PSC 3-manifolds $M$, he proved the path-connectedness of the \textit{moduli space} $\sM^{>0}_{\infty}(M) \big/ \operatorname{Diff}_+(M)$ of positive scalar curvature metrics up to orientation-preserving diffeomorphisms. For the proof, he exploited Hamilton--Perelman's Ricci flow with surgery \cite{Perelman2}, which he combined with the Gromov--Lawson \cite{GromovLawson:classification} construction of the positive scalar curvature connected-sum procedure (cf. Schoen--Yau \cite{SchoenYau:structure.psc}). We extend the result to:

\begin{theo} \label{theo:m.k.plus.path.moduli}
	Let $M$ be a closed orientable topologically PSC 3-manifold.\footnote{When $M$ is not topologically PSC, the $\sM^{>0}_{k}(M)$ in Theorem \ref{theo:m.k.plus.path.moduli} are all empty (see Corollary \ref{coro:m.k.plus.nonempty}) and the $\sM^{\geq 0}_{k}(M)$ consist of Ricci-flat metrics (see Lemma \ref{lemm:li.monotonicity.consequence}), so they are flat and isometrically covered by flat tori, and their space is understood (see \cite{Kang:flat.metrics}).} Then,
	\[ \sM^{> 0}_{k}(M) \big/ \operatorname{Diff}_+(M) \text{ and } \sM^{\geq 0}_{k}(M) \big/ \operatorname{Diff}_+(M) \]
	are both path-connected for all $k \in [\tfrac14, \infty]$. For the latter space, all path interiors can be taken in $\sM^{> 0}_{k}(M) \big/ \operatorname{Diff}_+(M)$.
\end{theo}

Once we have Theorem \ref{theo:m.k.plus.path.moduli}, the recent breakthrough theorem of Bamler--Kleiner \cite{BamlerKleiner:contractibility} proving that $\sM^{>0}_{\infty}(M)$ is contractible (and thus path-connected) when $M$ is a closed orientable topologically PSC 3-manifold combines with Theorem \ref{theo:m.k.plus.path.moduli} to give:

\begin{theo} \label{theo:m.k.plus.path}
	Let $M$ be a closed orientable topologically PSC 3-manifold. Then,
	\[ \sM^{> 0}_{k}(M) \text{ and } \sM^{\geq 0}_{k}(M) \]
	are both path-connected for all $k \in [\tfrac14, \infty]$. For the latter space, all path interiors can be taken in $\sM^{> 0}_{k}(M)$.
\end{theo}

Naturally, it would be interesting to know what happens in Theorems \ref{theo:m.k.plus.path.moduli}, \ref{theo:m.k.plus.path} in the regime $k \in (0, \tfrac14)$, particularly given that the 3-dimensional Yamabe problem has associated $k = \tfrac18 < \tfrac14$. To that end, we note the following special companion result:

\begin{theo} \label{theo:m.k.plus.path.yamabe}
	Let $M$ be a closed orientable topologically PSC 3-manifold. Then,
	\[ \cM^{>0}_{1/8}(M) \text{ and } \cM^{\geq 0}_{1/8}(M) \]
	are contractible and weakly contractible, respectively.
\end{theo}

It would also be interesting to understand topological properties of the inclusion in \eqref{eq:m.k.plus.inclusion} besides path-connectedness, along the lines of the Bamler--Kleiner result. We do not pursue this. See Appendix \ref{app:m.k.plus.more} for more on these spaces.

\subsection{Bartnik mass of apparent horizons}

When $k = \tfrac12$, Definition \ref{defi:m.k.plus} relates to the space of apparent horizons (quasilocal black hole boundaries), diffeomorphic to $M$, of time-symmetric $(n+1)$-dimensional initial data to the Einstein equations satisfying the dominant energy condition. To see why, let us recall the setting.

For closed $n$-dimensional $(M, g)$, the apparent horizon Bartnik mass is defined as\footnote{We adopt the convention that the infimum of an empty set is $\infty$.}
\begin{equation} \label{eq:bartnik}
	\mathfrak{m}_{B}(M, g, H = 0) = \inf \{ \mathfrak{m}_{ADM}(\bm{M}, \bm{g}) : (\bm{M}, \bm{g}) \in \cE_{B}(M, g, H=0) \},
\end{equation}
where $\cE_{B}(M, g, H=0)$ is the set of complete, connected, asymptotically flat $(\bm{M}, \bm{g})$ with nonnegative scalar curvature, no closed interior minimal hypersurfaces, and minimal ($H = 0$) boundary isometric to $(M, g)$. The quantity $\mathfrak{m}_{ADM}(\bm{M}, \bm{g})$ is the ADM energy/mass of the time-symmetric initial data set \cite{ADM:1, ADM:2}. Such $(\bm{M}, \bm{g})$ are time-symmetric initial data sets for solutions of Einstein's equations with the dominant energy condition. 

Note that all $(\bm{M}, \bm{g}) \in \cE_{B}(M, g, H=0)$ are orientable when $n+1 \leq 7$ (otherwise geometric measure theory yields a closed interior minimizing hypersurface) and thus $M = \partial \bm{M}$ is orientable whenever $\cE_{B}(M, g, H=0) \neq \emptyset$. Thus, $M$ is always assumed orientable when discussing $\mathfrak{m}_{B}(M, g, H = 0)$, at least in dimension $n+1 \leq 7$.

The apparent horizon Bartnik mass is difficult to compute because:
\begin{itemize}
	\item the set $\cE_{B}(M, g, H=0)$ is difficult to understand and is even unclear when it is nonempty (minimal surfaces tend to abound, \cite{IrieMarquesNeves:generic, MarquesNevesSong:equidistribution, ChodoshMantoulidis:generic, GasparGuaraco:weyl, Zhou:mult.one, Song:yau.conjecture});
	\item the quantity $\mathfrak{m}_{ADM}(\bm{M}, \bm{g})$ is difficult to compute.
\end{itemize}
There exists a very nontrivial lower bound for \eqref{eq:bartnik} due to Bray \cite{Bray:penrose} and Bray--Lee's \cite{BrayLee:penrose} Riemannian Penrose Inequality, a refinement of the Schoen--Yau Positive Energy Theorem \cite{SchoenYau:positive.energy, Schoen:variational.theory} (cf. Witten \cite{Witten:positive.energy}). The Riemannian Penrose Inequality says that, for $2 \leq n \leq 6$, one has
\begin{equation} \label{eq:rpi} 
	(\bm{M}, \bm{g}) \in \cE_B(M, g, H=0) \implies \mathfrak{m}_{ADM}(\bm{M}, \bm{g}) \geq \tfrac12 (\sigma_{n}^{-1} \vol_g(M))^{(n-1)/n},
\end{equation}
and thus obviously
\begin{equation} \label{eq:rpi.bartnik}
	\mathfrak{m}_{B}(M, g, H = 0) \geq \tfrac12 (\sigma_{n}^{-1} \vol_g(M))^{(n-1)/n},
\end{equation}
where $\sigma_n$ is the volume of the standard round $\SS^n$ in both equations above; we note that the first proof of a special case of the theorem was provided by Huisken--Illmanen \cite{HuiskenIlmanen:penrose} where they treated the case $n = 2$ and $M$ connected. The Riemannian Penrose Inequality also carries a rigidity statement. Namely, equality is attained on the right of \eqref{eq:rpi} if and only if $(\bm{M}, \bm{g}) \in \cE_B(M, g, H=0)$ is the Riemannian mass $m$ exterior Schwarzschild manifold 
\[ (\bm{M}, \bm{g}) \cong \left( \RR^{n+1} \setminus B_{(m/2)^{1/(n-1)}}(\bm{0}), \Big( 1 + \tfrac{m}{2} |\bm{x}|^{1-n} \Big)^{\frac{4}{n-1}} \bm{\delta} \right), \]
with $\bm{\delta}$ the flat metric on $\RR^{n+1}$ and the parameter choice
\[ m = \tfrac12 (\sigma_{n}^{-1} \vol_g(M))^{(n-1)/n}. \]
These manifolds will feature in some of our subsequent constructions.

An insight originating from \cite{MantoulidisSchoen:bartnik} is that path-connectedness of the spaces in Definition \ref{defi:m.k.plus}, with $k=\tfrac12$, yields constructions that saturate the Riemannian Penrose Inequality \eqref{eq:rpi}, forcing \eqref{eq:rpi.bartnik} to be an equality in cases of interest.

Let us first explain the relevance of Definition \ref{defi:m.k.plus}, since it is not obvious. Consider any metric $g$ on $M$ for which $\cE_B(M, g, H=0) \neq \emptyset$. Let $(\bm{M}, \bm{g}) \in \cE_B(M, g, H=0)$. The boundary $(M, g) = \partial (\bm{M}, \bm{g})$ must be a strictly area-minimizing minimal hypersurface in $(\bm{M}, \bm{g})$ when $2 \leq n \leq 6$, otherwise there would exist an interior minimal hypersurface (with lesser area) contradicting the no-interior-minimal hypersurface requirement for $(\bm{M}, \bm{g})$. Using Lemma \ref{lemm:m.12.plus.galloway}, which revisits the proof of a subtle splitting theorem of Galloway \cite[Theorem 3.1]{Galloway:outermost.psc} who proved that $M$ must be topologically PSC, one gets the following refined conclusion:\footnote{Conclusion \eqref{eq:galloway.inclusion}  implies that $M$ is topologically PSC and $g \in \sM^{\geq 0}_{1/2}(M)$. See also Remark \ref{rema:bartnik.bray.closure.condition}.}
\begin{equation}
	\cE_B(M, g, H=0) \neq \emptyset \implies g \in \overline{\sM^{> 0}_{1/2}(M)}. \label{eq:galloway.inclusion}
\end{equation} 
In other words \eqref{eq:bartnik} is only nontrivial when $g \in \overline{\sM^{>0}_{1/2}(M)}$. 

Consider first the 2-dimensional case. Here, our orientable and topologically PSC $M$ is a finite collection of $\SS^2$'s. Only the connected case, $M = \SS^2$, is well-understood to date:
\begin{multline} \label{eq:bartnik.2.dim}
	g \in \overline{\sM^{> 0}_{1/2}(\SS^2)} \implies \cE_B(\SS^2, g, H=0) \neq \emptyset \text{ and }  \\
	\mathfrak{m}_{B}(\SS^2, g, H = 0) = \tfrac12 (\sigma_2^{-1} \area_g(\SS^2))^{1/2}.
\end{multline}
This is due to the second named author and Rick Schoen \cite{MantoulidisSchoen:bartnik} (the case $\sM^{>0}_{1/2}(\SS^2)$) and Chau--Martens \cite{ChauMartens:degenerate} (the borderline case). The proof requires:
\begin{itemize}
	\item the path-connectedness of $\overline{\sM^{>0}_{1/2}(\SS^2)}$, with path interiors in $\sM^{>0}_{1/2}(\SS^2)$, and 
	\item a sharp extension construction involving warping functions. 
\end{itemize}

In Section \ref{sec:psc.cobordances} we give a proof of the second construction in all dimensions, which applies whenever one has the appropriate path-connectedness result. Given Theorem \ref{theo:m.k.plus.path} (which guarantees path-connectedness of the relevant space), we can compute the Bartnik mass of connected 3-dimensional apparent horizons:

\begin{theo} \label{theo:bartnik.3.dim}
	Let $M$ be a closed connected orientable topologically PSC 3-manifold. Then, either:
	\begin{itemize}
		\item $\cE_{B}(M, g, H=0) = \emptyset$ for all $g \in \overline{\sM^{> 0}_{1/2}(M)}$, or
		\item $\cE_{B}(M, g, H=0) \neq \emptyset$ for all $g \in \overline{\sM^{> 0}_{1/2}(M)}$.
	\end{itemize}
	In the latter case, 
	\begin{equation} \label{eq:bartnik.3.dim.general}
		\mathfrak{m}_{B}(M, g, H=0) = \mathfrak{c}_B(M) \vol_g(M)^{2/3} \text{ for all } g \in \overline{\sM^{> 0}_{1/2}(M)},
	\end{equation}
	for a topological constant $\mathfrak{c}_B(M)$ that is $\geq \tfrac12 \sigma_3^{-2/3}$ by \eqref{eq:rpi}. Additionally:
	\begin{equation} \label{eq:bartnik.3.dim}
		\mathfrak{c}_B(\SS^3) = \tfrac12 \sigma_3^{-2/3}.
	\end{equation}
\end{theo}

\begin{rema} \label{rema:sormani.lee.conjecture}
	The proof of \eqref{eq:bartnik.3.dim} is explicit enough to see that, whenever $M = \SS^3$, the minimizing sequence $(\bm{M}, \bm{g}_i)$ contains isometric copies of mass $m_i$ Schwarzschild manifolds outside a compact set, with $\lim_{i \to \infty} m_i = \tfrac12 (\sigma_3^{-1} \vol_g(M))^{2/3}$ and $\lim_{i \to \infty} \bm{g}_i = \bm{g}$ in $C^0$ for a piecewise-smooth $\bm{g}$ with the property that $(\bm{M}, \bm{g})$ contains an isometric copy of a mass $\tfrac12 (\sigma_3^{-1} \vol_g(M))^{2/3}$ Schwarzschild manifold up to its horizon attached to a compact PSC cylinder from a round $\SS^3$ to $(M, g)$ via a foliation of minimal spheres. The same is true in 2 dimensions. It was conjectured in \cite[Conjecture 5.9]{LeeSormani:near.equality} that minimizing sequences for non-round $(M, g)$ would converge, in a suitable weak sense, to Schwarzschild with an infinitely long neck. Our construction disproves this. Armando Cabrera Pacheco and Carla Cederbaum also observed a similar failure of the conjecture in January 2019 and their work is forthcoming.
\end{rema}

We emphasize that $M$ need not be $\SS^3$ in \eqref{eq:bartnik.3.dim.general}, and all apparent horizon metrics are exhausted in view of \eqref{eq:galloway.inclusion}. In the special case $M = \SS^3$, $g \in \sM^{>0}_{\infty}(M)$, the Bartnik mass of apparent horizons was previously computed by Cabrera Pacheco--Miao \cite{CabreraMiao:higher.dim} using Cod\'a Marques's connectedness result, which we generalized. It is of physical interest to determine whether minimizing extensions that yield \eqref{eq:bartnik.3.dim} can be arranged to be scalar-flat, i.e., vacuum initial data sets, as in Miao--Xie \cite{MiaoXie:vacuum.extensions}. It is also of interest to derive upper bounds for the Bartnik mass in case $M$ is disconnected. Our method here provides a non-explicit upper bound on the right hand side of \eqref{eq:bartnik.3.dim.general} that depends on the topology of $M$ and the individual areas of its components. One can likely derive explicit non-sharp upper bounds from the work of Carlotto--Schoen \cite{CarlottoSchoen:gluing}. We do not pursue this.

\subsection{A Bartnik--Bray mass}

Unfortunately, the precise value of the apparent horizon Bartnik mass remains unknown for:
\begin{itemize}
	\item disconnected 2- or 3-dimensional $M$;
	\item 3-dimensional $M$ other than $M = \SS^3$;
	\item all higher dimensional $M$, except for certain special metrics on $M = \SS^n$; see Cabrera Pacheco--Miao \cite{CabreraMiao:higher.dim}, Cabrera Pacheco--Cederbaum--Gehring--Pe\~{n}uela Diaz \cite{CabreraCederbaumGehringPenuela:preprint}.
\end{itemize}
While we do not have satisfactory answers for the Bartnik mass for these bullet points at this time, we know how to compute a relaxation of Bartnik's mass due to Bray \cite{Bray:penrose} in near-complete generality (see Remark \ref{rema:bartnik.bray.closure.condition}). We defer the precise definition of the Bartnik--Bray mass until Section \ref{sec:applications.relativity.bartnik.bray}, but state the main result here:

\begin{theo} \label{theo:bartnik.bray.n.dim}
	Let $M$ be a closed topologically PSC $n$-manifold. Consider the subset of $\overline{\sM^{> 0}_{1/2}(M)}$ given by:\footnote{See Remark \ref{rema:differentiating.lambda1} for more on differentiating $t \mapsto \lambda_1(-\Delta_{g_t} + \tfrac12 R_{g_t})$.}
	\begin{align} \label{eq:m.k.plus.bar}
		\operatorname{LinClos}[\sM^{>0}_{1/2}(M)] 
			& := \{ g \in \overline{\sM^{> 0}_{1/2}(M)} : \text{ there exists a smooth path} \\
			& \qquad \qquad \qquad [0, 1) \ni t \mapsto g_t \text{ with } g_0 = g \text{ and } \nonumber \\
			& \qquad \qquad \qquad \big[ \tfrac{d}{dt} \lambda_1(-\Delta_{g_t} + \tfrac12 R_{g_t}) \big]_{t=0} > 0 \}. \nonumber 
	\end{align}
	If $g \in \operatorname{LinClos}[\sM^{>0}_{1/2}(M)]$ and $\cE_{BB}(M, g, H=0) \neq \emptyset$, then
	\begin{equation} \label{eq:bartnik.bray.n.dim.general}
		\mathfrak{m}_{BB}(M, g, H=0) = \mathfrak{c}_{BB}(n) \vol_g(M)^{(n-1)/n}
	\end{equation}
	for some universal constant $\mathfrak{c}_{BB}(n)$ satisfying
	\begin{equation} \label{eq:bartnik.bray.n.dim}
		\mathfrak{c}_{BB}(n) \leq \tfrac12 \sigma_n^{-(n-1)/n}.
	\end{equation}
	If the outer-minimizing Riemannian Penrose Inequality holds\footnote{This is known to be true when $2 \leq n \leq 6$ in view of Bray--Lee \cite{Bray:penrose} but is expected to hold for all $n$ in view of the flexibility of the Bray--Lee argument and recently announced advances in minimal hypersurface theory (\cite{SchoenYau:higher.dimensional, Lohkamp:higher.dimensional}).} for $(M, g)$, then \eqref{eq:bartnik.bray.n.dim} gets upgraded to an equality. 
\end{theo}

As in Remark \ref{rema:sormani.lee.conjecture}, the minimizing sequence consists of metrics that are mass $m_i$ Schwarzschild outside a compact set, $\lim_{i \to \infty} m_i = \tfrac12 (\sigma_n^{-1} \vol_g(M))^{(n-1)/n}$. Unlike the previous remark, however, our minimizing sequence produces minimal hypersurfaces outside $(M, g)$ with much larger area. This effectively shields the prescribed horizon within the apparent horizon of a more massive black hole. Therefore, this construction is not a candidate for the much more restrictive Bartnik mass setting of Theorem \ref{theo:bartnik.3.dim}. To the authors, this indicates that the two masses are not physically interchangeable.

We emphasize that $M$ need not be connected and that our computation is valid as long as a single Bartnik--Bray extension exists. See Lemma \ref{lemm:bartnik.bray.n.dim.extensions} for many examples of Bartnik--Bray extendible manifolds, i.e., those satisfying $\cE_{BB}(M, g, H=0) \neq \emptyset$. 

\begin{rema} \label{rema:bartnik.bray.closure.condition}
	What more can be said about the inclusions
	\begin{equation} \label{eq:bartnik.bray.closure.condition.three}
		\operatorname{LinClos}[\sM^{>0}_{1/2}(M)] \subset \overline{\sM^{> 0}_{1/2}(M)} \subset \sM^{\geq 0}_{1/2}(M)
	\end{equation}
	for closed topologically PSC $M$? 
	
	To begin with, note that running Ricci flow for a short amount of time and invoking Lemma \ref{lemm:li.monotonicity.consequence} implies that
	\begin{equation} \label{eq:bartnik.bray.closure.condition.3.minus.1}
		\sM^{\geq 0}_{1/2}(M) \setminus \operatorname{LinClos}[\sM^{>0}_{1/2}(M)] \subset \{ g \in \operatorname{Met}(M) : \Ric_{g} \equiv 0 \}.
	\end{equation}

	In particular, when $\dim M = 2$ or $3$, the spaces in \eqref{eq:bartnik.bray.closure.condition.three} coincide in view of the fact that topologically PSC $2$- or $3$-manifolds don't carry (Ricci-)flat metrics, forcing the right hand side of \eqref{eq:bartnik.bray.closure.condition.3.minus.1} to be empty.
	
	The situation is more subtle in higher dimensions.\footnote{We are grateful to the anonymous referee for providing us with an abundance of references for this discussion and answering questions of ours.} It is beneficial to understand the set differences 
	\[ \sM^{\geq 0}_{1/2}(M) \setminus \overline{\sM^{>0}_{1/2}(M)} \text{ and } \overline{\sM^{>0}_{1/2}(M)} \setminus \operatorname{LinClos}[\sM^{>0}_{1/2}(M)]  \]
	separately. First, we have:
	\begin{equation}\label{eq:bartnik.bray.closure.condition.3.minus.2}
	\begin{gathered} 
		\{ g \in \operatorname{Met}(M) : (\tilde M, \tilde g) \text{ carries a parallel spinor} \} \\
		\bigcap \\
		\sM^{\geq 0}_{1/2}(M) \setminus \overline{\sM^{>0}_{1/2}(M)} \\
		\bigcap \\
		\{ g \in \operatorname{Met}(M) : \Ric_{g} \equiv 0 \}, 
	\end{gathered}
	\end{equation}
	where, in the topmost set, $(\tilde M, \tilde g)$ denotes the universal covering of $(M, g)$. The second inclusion in \eqref{eq:bartnik.bray.closure.condition.3.minus.2} follows from \eqref{eq:bartnik.bray.closure.condition.3.minus.1}, while the first follows from \cite{AmmannKronckeWeissWitt} (see also \cite{Wang:preserving.parallel.spinors, DaiWangWei}) and specifically from the fact that metrics in the topmost set of \eqref{eq:bartnik.bray.closure.condition.3.minus.2} are Ricci-flat, and thus in $\sM^{\geq 0}_{1/2}(M)$, but by \cite[Corollary 3]{AmmannKronckeWeissWitt} are not limits of PSC metrics and by a simple continuity argument involving first eigenfunctions of the conformal Laplacian and \eqref{eq:m.k.plus.inclusion} they are not in $\overline{\sM^{>0}_{1/2}(M)}$.
	
	Note that, by Stolz's (\cite{Stolz:simply.connected}) deep characterization of the topologically PSC condition in simply connected manifolds, there are several examples in the topmost set in  \eqref{eq:bartnik.bray.closure.condition.3.minus.2} with $M$ being topologically PSC: all closed simply-connected manifolds with holonomy in $G_2$ or $\operatorname{SU}(4k+3)$ (see \cite[Proof of Proposition 5.2]{DaiWangWei}) and products of such manifolds with Ricci-flat manifolds with holonomy in $\{ 1\}$, $\operatorname{SU}$, $\operatorname{Sp}$, or $G_2$. 
	
	There are no known examples of metrics in the bottommost set of \eqref{eq:bartnik.bray.closure.condition.3.minus.2} that are not in the topmost set. This is a known open question.

	A similar argument yields 
	\begin{equation} \label{eq:bartnik.bray.closure.condition.2.minus.1}
	\begin{gathered}
		\overline{\sM^{> 0}_{1/2}(M)} \setminus \operatorname{LinClos}[\sM^{>0}_{1/2}(M)] \\
		\bigcap \\
		\{ g \in \operatorname{Met}(M) : \Ric_{g} \equiv 0 \} \setminus \{ g \in \operatorname{Met}(M) : (\tilde M, \tilde g) \text{ carries a parallel spinor} \}. 
	\end{gathered}
	\end{equation}
	As explained, it is an interesting open question whether the right hand side of \eqref{eq:bartnik.bray.closure.condition.2.minus.1} (and thus also the left, by inclusion) is ever nonempty when $\dim M \geq 4$.
\end{rema}

\subsection{Fill-in problem}

The tools we have developed make progress toward:

\begin{quest}[Gromov, \cite{Gromov:scalar.curvature.boundary.questions}] \label{ques:gromov}
	Which closed Riemannian 3-manifolds $(M, g)$ carry compact 4-dimensional fill-ins $(\bm{M}, \bm{g})$ with nonnegative scalar curvature and the mean curvature vectors on the boundary pointing to the interior?
\end{quest}

For orientable topologically PSC $M$ we prove that, as long as there exists at least one fill-in whose induced boundary metric belongs to $\sM^{\geq 0}_{1/2}(M)$\footnote{Recall that this class is larger than the class of positive scalar curvature metrics on the boundary, which is the most commonly studied one.}, then \textit{every} metric in $\sM^{\geq 0}_{1/2}(M)$ can be filled in with the \textit{same} fill-in topology:

\begin{theo} \label{theo:m.k.plus.fillin}
	Assume that $(\bm{M}, \bm{g}_0)$ is a compact Riemannian 4-manifold-with-boundary, with $\partial \bm{M} = M$ a topologically PSC orientable 3-manifold, satisfying:
	\begin{enumerate}
		\item $g_0 := \bm{g}_0 \restr M \in \sM^{\geq 0}_{1/2}(M)$;
		\item $R_{\bm{g}_0} \geq 0$ everywhere on $\bm{M}$; and,
		\item $M \subset (\bm{M}, \bm{g}_0)$ has strictly inward pointing mean curvature vector.
	\end{enumerate}
	Write $M = M_1 \sqcup \cdots \sqcup M_k$ for the connected components of $M$. Then, for every $g \in \sM^{\geq 0}_{1/2}(M)$ with $\vol_g(M_i) > \vol_{g_0}(M_i)$ for all $i = 1, \ldots, k$, and every $\bm{U} \subset \subset \bm{M} \setminus \partial \bm{M}$ open, there exists a metric $\bm{g}$ on $\bm{M}$ so that
	\begin{itemize}
		\item $\bm{g} \restr \bm{U} \equiv \bm{g}_0 \restr \bm{U}$; and 
		\item (1)-(3) hold with $g$, $\bm{g}$ in place of $g_0$, $\bm{g}_0$.
	\end{itemize}
	Moreover, if $(\bm{M}, \bm{g}_0)$ is known to also satisfy:
	\begin{enumerate}
		\item[(4)] $(\bm{M}, \bm{g}_0)$ contains no closed minimal hypersurfaces,
	\end{enumerate}
	then $(\bm{M}, \bm{g})$ can be taken so that, additionally,
	\begin{itemize}
		\item (4) holds too with $\bm{g}$ in place of $\bm{g}_0$.
	\end{itemize}
\end{theo}

The work of Carr \cite{Carr:psc.construction} guarantees that, for every $k \in \NN$, one can find standard $\#^k (\DD^3 \times \SS^1)'s$ embedded in $(\RR^4, \sum_{i=1}^4 dx_i^2)$ with boundary having a metric $g_0$ of positive induced scalar curvature. These are mean-convex by the maximum principle (see \cite{CaffarelliNirenbergSpruck, HuiskenSinestrari}) and thus do not contain minimal hypersurfaces. So, by Theorem \ref{theo:m.k.plus.fillin} above one can construct fill-ins for all $g \in \sM^{\geq 0}_{1/2}(\#^k (\SS^2 \times \SS^1))$ on the standard handlebody topology and with no interior minimal hypersurfaces.

While the restriction of only producing boundary metrics $g \in \sM^{\geq 0}_{1/2}(M)$ is undesirable for the purposes of Question \ref{ques:gromov}, it is key in allowing us to obtain our extra conclusion (4) above, which has specific geometric significance: such manifolds-with-boundary are indecomposable from the point of view of minimal surface theory; cf. the Meeks--Simon--Yau \cite{MeeksSimonYau} characterization of 3-dimensional handlebodies.

Recent work of Kazaras \cite[Theorem B]{Kazaras:desingularizing} on PSC bordisms guarantees that every PSC $(M, g_0)$ can be filled in with (1)-(3), so by Theorem \ref{theo:m.k.plus.fillin} every $g \in \sM^{\geq 0}_{1/2}(M)$ can also be filled in with (1)-(3) on the same background topology. It is not clear how many fill-ins in \cite{Kazaras:desingularizing} satisfy (4), but those of lens spaces are believed to (\cite{Kazaras:personal}).

Theorem \ref{theo:m.k.plus.fillin} above is a special case of our ``Monotone PSC almost-cobordance'' construction, which relies crucially on the flexibility of $\sM^{\geq 0}_{1/2}(M)$ that Theorem \ref{theo:m.k.plus.path} affords us. The 2-dimensional case of such a result was a crucial component of our study of singularities in PSC metrics in 3 dimensions  \cite{LiMantoulidis:skeleton} (see also \cite[Theorem A]{Kazaras:desingularizing} for a 4-dimensional analog that relates to the discussion above), and in previous work of the second author in mathematical general relativity and quasi-local mass \cite{MantoulidisMiao:total.mean.curvature, MantoulidisMiaoTam:capacity}. See also Jauregui \cite{Jauregui:fillins} and Jauregui--Miao--Tam \cite{JaureguiMiaoTam:fillins} for relativistic applications in case where the boundary metric is PSC.

\section{Proof of Theorems \ref{theo:m.k.plus.path.moduli}, \ref{theo:m.k.plus.path}, \ref{theo:m.k.plus.path.yamabe}} \label{sec:proof.of.connectedness}

The path-connectedness of $\sM^{>0}_k(M)$ will follow from Ricci flow with surgery. We show that the condition of being in $\sM^{>0}_k(M)$ is preserved under smooth Ricci flow as well as by the standard 3D Ricci flow surgery process (taking care to keep track of the topology when surgeries occur), and that we eventually arrive at an element of $\sM^{>0}_\infty(M)$. We then conclude by invoking \cite{Marques:deforming.psc} and \eqref{eq:m.k.plus.inclusion}. 

The case of $\sM^{\geq 0}_k(M)$ follows from this as well, as Ricci flow will immediately flow elements of $\sM^{\geq 0}_k(M)$ into $\sM^{>0}_k(M)$ in the relevant setting.

We start by observing the monotonicity of $\lambda_1(-\Delta_g + kR_g)$ under smooth Ricci flow. This is the only place in the argument where $k \geq \tfrac14$ is invoked.

\begin{lemm} \label{lemm:li.monotonicity.consequence}
	Suppose that $(M, g_t)_{t \in [0,T)}$ is a smooth Ricci flow on a closed $n$-manifold, $n \geq 2$. If $k \in [\tfrac14, \infty)$ and $g_0$ isn't Ricci-flat, then
	\begin{equation} \label{eq:li.monotonicity.consequence}
		\tfrac{d}{dt} \lambda_1(-\Delta_{g_t} + k R_{g_t}) > 0 \text{ for all } t \in [0, T).
	\end{equation}
\end{lemm}

\begin{rema}[Smooth dependence of first eigenvalues and eigenfunctions] \label{rema:differentiating.lambda1}
	When $M$ is connected, the space of first eigenfunctions of $-\Delta_g + kR_g$ is one-dimensional. Given a smooth path of metrics $(g_t)_{t \in I}$ on $M$, with $I \subset \RR$ an interval, it follows from \cite[Lemma A.1]{MantoulidisSchoen:bartnik} that $t \mapsto \lambda_1(-\Delta_{g_t} + kR_{g_t})$ is smooth and that we can also choose a smooth path of functions $(u_t)_{t \in [0,1]}$ so that each $u_t$ is a first eigenfunction of $-\Delta_{g_t} + kR_{g_t}$ that is, e.g., positive.
	
	When $M$ is disconnected, $\lambda_1(-\Delta_g + k R_g)$ is the minimum of $\lambda_g((-\Delta_g + kR_g) \restr M_i)$ among all connected components $M_i \subset M$. Given a smooth path $t \mapsto g_t$, $\tfrac{d}{dt} \lambda_1(-\Delta_{g_t} + kR_{g_t})$ will be interpreted in the sense of forward difference quotients.
\end{rema}

\begin{proof}[Proof of Lemma \ref{lemm:li.monotonicity.consequence}]
	This is a consequence of \cite{Li:monotonicity.RF} (where ``$k$'' is to be replaced with ``$4k$'' in our setting), specifically \cite[Theorem 4.2]{Li:monotonicity.RF} and a modification of the argument in \cite[Theorem 5.2]{Li:monotonicity.RF}. We explain how below. Note that it suffices to give a proof for $t=0$, since non-Ricci-flat closed manifolds cannot flow into Ricci-flat ones under smooth Ricci flow (\cite[Theorem 1.1]{Kotschwar:ricci}).
	
	By Remark \ref{rema:differentiating.lambda1}, it suffices to assume that $M$ is connected. We choose, using the same remark, $(u_t)_{t \in [0,T)}$ to be a smooth path of positive eigenfunctions of $(-\Delta_{g_t} + k R_{g_t})_{t \in [0,T)}$. Fix $t \in (0, T)$. Solve the conjugate heat equation \cite[(2)]{Li:monotonicity.RF} backwards for $s \in [0,t]$:
	\begin{align*}
		\tfrac{\partial}{\partial s} f^{(t)}_s 
			& = -\Delta_{g_s} f^{(t)}_s - R_{g_s} + |\nabla_{g_s} f^{(t)}_s|^2 \text{ for } s \in [0,t], \\
		f^{(t)}_t & := -2 \log u_t \text{ (terminal condition).}
	\end{align*}
	This is doable because the equation is a backwards linear heat equation when written out in terms of $e^f$. Using the chain of inequalities in the proof of \cite[Theorem 5.2]{Li:monotonicity.RF}, and also \cite[Theorem 4.2]{Li:monotonicity.RF} together with the mean value theorem, it follows that there exists a constant $\theta^{(t)} \in (0, 1)$ such that
	\begin{align*}
		& \frac{\lambda_1(-\Delta_{g_t} + k R_{g_t}) - \lambda_1(-\Delta_{g_0} + k R_{g_0})}{t} \\
		&\qquad \geq 2(4k-1) \int_M |\Ric_{g_{\theta^{(t)} t}}|^2 \, d\mu_{g_{\theta^{(t)} t}} + 2 \int_M |\Ric_{g_{\theta^{(t)} t}} + \nabla^2_{g_{\theta^{(t)} t}} f^{(t)}_{\theta^{(t)} t}|^2 \, d\mu_{g_{\theta^{(t)} t}}.
	\end{align*}
	Using the smoothness of the Ricci flow up to $t=0$, we can send $t \to 0$ above to get
	\[ [ \tfrac{d}{dt} \lambda_1(-\Delta_{g_t} + kR_{g_t}) ]_{t=0} \geq 2(4k-1) \int_M |\Ric_{g_0}|^2 \, d\mu_{g_0} + 2 \int_M |\Ric_{g_0} + \nabla^2_{g_0} f^{(0)}_0|^2 \, d\mu_{g_{0}}. \]
	If $k > \tfrac14$, the result follows immediately from the strict positivity of the first integral above. If $k = \tfrac14$, one needs to investigate the second integral. This integral is positive too: its vanishing would imply that $(M, g_0)$ is a steady gradient Ricci soliton, which would imply that $g_0$ is Ricci-flat since $M$ is closed \cite[Theorem 1.5]{PetersenWylie}.
\end{proof}

In all that follows, we assume the setting of Theorem \ref{theo:m.k.plus.path.moduli}. Namely, $M$ is a closed, orientable, topologically PSC 3-manifold, and $k \in [\tfrac14, \infty)$. We have:

\begin{lemm} \label{lemm:m.k.plus.ricci.flow.positive}
	If $g \in \sM^{\geq 0}_{k}(M)$ and $(M, g_t)_{t \in [0,T)}$ is a smooth Ricci flow with $g_0 = g$, then
	\[ \tfrac{d}{dt} \lambda_1(-\Delta_{g_t} + k R_{g_t}) > 0 \text{ for all } t \in [0, T). \]
\end{lemm}
\begin{proof}
	This is a direct consequence of Lemma \ref{lemm:li.monotonicity.consequence} and the fact that closed topologically PSC 3-manifolds do not carry (Ricci-)flat metrics.
\end{proof}

\begin{lemm} \label{lemm:m.k.plus.finite.extinction}
	If $g \in \sM^{\geq 0}_{k}(M)$, then the Ricci flow with surgery starting at  $(M, g)$, with any choice of sufficiently small surgery parameters, becomes extinct in finite time and surgery times do not accumulate.
\end{lemm}
\begin{proof}
	This follows from Perelman's breakthroughs on Ricci flow with surgery. For finite-time extinction, see \cite[Theorem 1.1]{Perelman3} (see also \cite[Theorem 1.4]{ColdingMinicozzi:finite.time.extinction}, and \cite[Claim 3.7]{KleinerLott:notes} and the subsequent paragraph), recalling that closed orientable topologically PSC 3-manifolds have no aspherical factors their prime decomposition (\cite[Theorem E]{GromovLawson:summary}). For non-accumulation of singularities, see \cite[Claim 3.6]{KleinerLott:notes}.
\end{proof}

\begin{lemm} \label{lemm:m.k.plus.ricci.flow.surgery}
	If $g \in \sM^{\geq 0}_{k}(M)$ and $(M_t, g_t)_{t \in [0, T^*)}$ is a Ricci flow with surgery starting at $(M_0, g_0) = (M, g)$ with sufficiently small surgery parameters, then 
	\[ g_t \in \sM^{>0}_{k}(M_t) \text{ for all } t \in (0, T^*). \]
\end{lemm}
\begin{proof}
	Given Lemma \ref{lemm:m.k.plus.ricci.flow.positive}, we may suppose, without loss of generality, that
	\begin{equation} \label{eq:m.k.plus.ricci.flow.surgery.initial}
		\lambda_1(-\Delta_g + k R_g) \geq 2 \lambda > 0.
	\end{equation}
	It will suffice, then, to prove that 
	\begin{equation} \label{eq:m.k.plus.ricci.flow.surgery.goal}
		\lambda_1(-\Delta_{g_t} + k R_{g_t}) \geq \lambda \text{ for all } t \in [0, T^*),
	\end{equation}
	provided we choose sufficiently small surgery parameters. This follows essentially as in \cite[Section 93.12]{KleinerLott:notes}, who did the case $k = \tfrac14$. The case $k \in (\tfrac14, \infty)$ continues to enjoy the same necessary proof ingredients: eigenvalue monotonicity and coercivity (``Agmon-type,'' in the language of \cite{KleinerLott:notes}) estimates in the surgery region. We sketch a proof of the argument  for clarity.
	
	First, note that \eqref{eq:m.k.plus.ricci.flow.surgery.goal} is trivially true in the absence of surgeries, in view of Lemma \ref{lemm:m.k.plus.ricci.flow.positive}. So the point is to study drops of $\lambda_1(t) := \lambda_1(-\Delta_{g_t} + k R_{g_t})$ at surgery times $t = t_1, \ldots, t_Q \in [0, T^*)$. Arguing verbatim as in \cite[Section 93.12]{KleinerLott:notes} and replacing \cite[Lemmas 93.16, 93.21]{KleinerLott:notes} with Lemmas \ref{lemm:kleiner.lott.lambda}, \ref{lemm:kleiner.lott.agmon}, we find:
	\begin{equation} \label{eq:m.k.plus.ricci.flow.surgery.lambdagap.1}
		|\lambda_1^-(t_i) - \lambda_1^+(t_i)| \leq C h(t_i)^4 \text{ at every } i = 1, \ldots, Q,
	\end{equation}
	where
	\[ \lambda_1^+(t_i) = \lim_{t \to t_i^+} \lambda_1(t), \; \lambda_1^-(t_i) = \lim_{t \to t_i^-} \lambda_1(t) \]
	and $h : [0, T^*) \to \RR_+$ denotes the surgery scale parameter. Recall that Ricci flow with surgery discards regions with volume $\geq C h(t_i)^3$, so
	\begin{equation} \label{eq:m.k.plus.ricci.flow.surgery.lambdagap.2}
		|\lambda_1^-(t_i) - \lambda_1^+(t_i)| \leq C h(t_i) (V^-(t_i) - V^+(t_i))
	\end{equation}	
	where
	\[ V^+(t_i) = \lim_{t \to t_i^+} \vol_{g_t}(M_t), \; V^-(t_i) = \lim_{t \to t_i^-} \vol_{g_t}(M_t). \]
	Combining \eqref{eq:m.k.plus.ricci.flow.surgery.lambdagap.1} and \eqref{eq:m.k.plus.ricci.flow.surgery.lambdagap.2} yields, for all $\tau \in (0, T^*)$:
	\begin{equation} \label{eq:m.k.plus.ricci.flow.surgery.lambdagap.3}
		\sum_{i : t_i \leq \tau} |\lambda_1^-(t_i) - \lambda_1^+(t_i)| \leq C (\sup_{[0, \tau]} h) \sum_{i : t_i \leq \tau} (V^-(t_i) - V^+(t_i)).
	\end{equation}
	It is crucial to point out that $C$ does not depend on $\tau$. We only introduce $\tau$ to facilitate our estimation of the volume-drop series on the right hand side of \eqref{eq:m.k.plus.ricci.flow.surgery.lambdagap.3}. We compute for all $t \in [0, T^*) \setminus \{ t_1, \ldots, t_Q \}$:
	\begin{align}
		- \tfrac{d}{dt} \vol_{g_t}(M_t) 
			& = \int_{M_t} R_{g_t} \, d\mu_{g_t} \nonumber \\
			& = k^{-1} \int_{M_t} (|\nabla_{g_t} 1|^2 + k R_{g_t} 1^2) \, d\mu_{g_t} \nonumber \\
			& \geq k^{-1} \lambda_1(t) \vol_{g_t}(M_t); \label{eq:m.k.plus.ricci.flow.surgery.lambdagap.4}
	\end{align}
	Thus, the volume of $(M_t, g_t)$ decreases monotonically (even across surgeries, where it decreases by $\geq C h(t_i)^3$) as long as $\lambda_1(t) \geq 0$. Combining \eqref{eq:m.k.plus.ricci.flow.surgery.lambdagap.3} with \eqref{eq:m.k.plus.ricci.flow.surgery.lambdagap.4}, the sum on the right side of \eqref{eq:m.k.plus.ricci.flow.surgery.lambdagap.3} telescopes, so
	\begin{equation} \label{eq:m.k.plus.ricci.flow.surgery.lambdagap.5}
		\sum_{i : t_i \leq \tau} |\lambda_1^-(t_i) - \lambda_1^+(t_i)| \leq C (\sup_{[0, \tau]} h) \vol_{g_0}(M_0)
	\end{equation}
	provided $\lambda_1(t) \geq 0$ on $[0, \tau]$. In particular, by choosing $h$ globally sufficiently small (as \cite[Section 93.12]{KleinerLott:notes} is free to do) we can arrange for the right hand side of \eqref{eq:m.k.plus.ricci.flow.surgery.lambdagap.5} to be $\leq \lambda$. Thus, this allows us to take $\tau \to T^*$ and conclude that  \eqref{eq:m.k.plus.ricci.flow.surgery.goal} holds, as claimed.
\end{proof}

\begin{proof}[Proof of Theorem \ref{theo:m.k.plus.path.moduli}]
	It will benefit us to first recall the strategy of \cite{Marques:deforming.psc}, where this result was proven (with different notation) for $k = \infty$. We start with an arbitrary $g \in \sM^{>0}_\infty(M)$ and isotope it through $\sM^{>0}_\infty(M)$ to a metric belonging to a smaller class of ``model'' metrics\footnote{Elements of $\mathfrak{M}(M)$ are called ``canonical'' metrics in \cite{Marques:deforming.psc}.} $\mathfrak{M}(M)$ on $M$ (its definition is not important for us), whose quotient $\mathfrak{M}(M)/\operatorname{Diff}_+(M)$ is path-connected (see \cite[p. 841-2]{Marques:deforming.psc}). 
	
	We use the same strategy, except we can now rely on the path-connectedness of $\sM^{>0}_{\infty}(M)/\operatorname{Diff}_+(M)$ and \eqref{eq:m.k.plus.inclusion}: it suffices to isotope an arbitrary $g \in \sM^{\geq 0}_k(M)$ to a positive scalar curvature metric; the result then follows from \cite{Marques:deforming.psc} and \eqref{eq:m.k.plus.inclusion}.
	
	The isotopy is provided by 3D Ricci flow \textit{with surgery} starting with $(M, g)$. Short-time existence for the smooth Ricci flow with initial data $(M, g)$ was established by Hamilton \cite{Hamilton:ricci}, and  long-time existence for the Ricci flow with surgery, with sufficiently small surgery parameters, was established by Perelman \cite{Perelman1, Perelman2, Perelman3}; see also \cite{KleinerLott:notes} for a detailed exposition. We know that the Ricci flow with surgery starting off $(M, g)$ becomes extinct in finite time and only undergoes finitely many surgeries (Lemma \ref{lemm:m.k.plus.finite.extinction}), and that at any given positive time our flow consists of $\sM^{>0}_{k}$ pieces (Lemma  \ref{lemm:m.k.plus.ricci.flow.surgery}). 
	
	We prove the existence of the isotopy through $\sM^{>0}_k(M)$ from $g$ to a metric in $\sM^{>0}_\infty(M)$ by induction on the number of surgeries. Since we will need to work on several different manifolds (due to surgery), we will  emphasize the manifolds on which our isotopies occur; for example, instead of saying $g$ is isotopic to $g'$, we will say $(M, g)$ is isotopic to $(M, g')$. 
	
	\textbf{Base case} (no surgeries). In this case, our Ricci flow $(M, g_t)_{t \in [0,T)}$ with $g_0 = g$ is smooth, becomes extinct as $t \to T$, and $(M, g_t)$ is covered by canonical neighborhoods (which have positive scalar curvature) as $t \to T$; we direct the reader to the top of \cite[p. 822]{Marques:deforming.psc} for the precise definition of these neighborhoods (see also \cite[Definition 69.1]{KleinerLott:notes}). Let $\varsigma > 0$ be sufficiently small so that $(M, g_{T-\varsigma})$ is covered by canonical neighborhoods. Then the Ricci flow provides the isotopy through $\sM^{>0}_k(M)$ from $(M, g) = (M, g_0)$ to $(M, g_{T-\varsigma})$, and the result follows since $g_{T-\varsigma} \in \sM^{>0}_\infty(M)$.
	
	\textbf{Inductive step} ($Q \geq 1$ surgeries). In this case, we have a Ricci flow with surgery $(M_t, g_t)_{t \in [0,T)}$ starting at $(M_0, g_0) = (M, g)$, with surgery times $0 < t_1 < \ldots < t_Q < T$, and which becomes extinct as $t \to T$. Now denote by $(M_-, g_-)$, $(M_+, g_+)$ the pre- and post-surgery manifolds at the surgery time $t=t_1$. (Thus, $M_- = M$.) By virtue of \eqref{eq:m.k.plus.ricci.flow.surgery.goal}, we may assume that
	\begin{equation} \label{eq:m.k.plus.path.lambda.lower.bound}
		\lambda_1(-\Delta + k R) \geq 2 \lambda > 0 \text{ for both } (M_-, g_-), (M_+, g_+).
	\end{equation}
	Viewing $(M_+, g_+)$ as an initial manifold that also evolves by Ricci flow with $Q-1$ surgeries, by the inductive hypothesis, $(M_+, g_+)$ is isotopic through $\sM^{>0}_k(M_+)$ to some $(M_+, h_+)$ with $h_+ \in \sM^{>0}_{\infty}(M_+)$. Therefore, by the connected-sum result for continuous families (Corollary \ref{coro:m.k.plus.surgery.families}), the Gromov--Lawson connected-sum of the components of $(M_+, g_+)$ and any discarded components at the surgery (which all are in $\sM^{>0}_\infty \subset \sM^{>0}_k$, the inclusion by \eqref{eq:m.k.plus.inclusion}) is isotopic through $\sM^{>0}_k(M_-)$ to some $(M_-, h_-)$ with $h_- \in \sM^{>0}_{\infty}(M_-)$. Moreover, this Gromov--Lawson connected sum of the components of $(M_+, g_+)$ and any discarded components is also isotopic through $\sM^{>0}_{k}(M_-)$ to $(M_-, g_-)$ by virtue of the reconstruction lemma below (Lemma \ref{lemm:m.k.plus.reconstruction}). Finally, $(M, g_-)$ is isotopic through $\sM^{>0}_{k}(M)$ to $(M, g_0) = (M, g)$ via the smooth Ricci flow itself. This completes the proof.
\end{proof}

\begin{lemm}[Reconstruction lemma] \label{lemm:m.k.plus.reconstruction}
	Assume the setup above. Then, the pre-surgery manifold is isotopic to the Gromov--Lawson connected sum of the components of the post-surgery manifold and any discarded components, through metrics with $\lambda_1(-\Delta + kR) \geq \lambda$,  with $\lambda$ as in \eqref{eq:m.k.plus.path.lambda.lower.bound}.
\end{lemm}
\begin{proof}
	Recall that the pre-surgery and post-surgery manifolds are denoted $(M_-, g_-)$, $(M_+, g_+)$, and they satisfy \eqref{eq:m.k.plus.path.lambda.lower.bound}. Without real loss of generality, we will assume that no components get discarded by surgery. The general case is a simple modification.  
	
	We introduce the following notation following \cite[Section 57.2]{KleinerLott:notes} and \cite[Section 93.12]{KleinerLott:notes} adapted to our setting:
	\begin{itemize}
		\item $\eps$, $r(t_1)$, $\delta(t_1)$, and $h(t_1)$ are the Ricci flow surgery parameters at the surgery time $t_1$;
		\item $(M_{\operatorname{cap}}, g_{\operatorname{cap}}) := (M_+, g_+) \setminus (M_-, g_-)$ is the cap region in $M_+$;
		\item $(X, g_X) := \overline{(M_+, g_+) \cap (M_-, g_-)}$ is the region unaffected by the surgery;
		\item $\Omega \subset M_-$ is the $\eps$-horn inside which the surgery is performed and which is extended until $R_{g_-} \sim r(t_1)^{-2}$.
	\end{itemize}
	The construction of Ricci flow with surgery guarantees that
	\begin{equation} \label{eq:m.k.plus.reconstruction.scal.cap}
		R \geq 2 c_1 \delta(t_1)^{-2} r(t_1)^{-2} \text{ on } (M_{\operatorname{cap}}, g_{\operatorname{cap}}),
	\end{equation}
	for a universal $c_1 \in \RR$. By Proposition \ref{prop:m.k.plus.surgery} we can perform a connected-sum operation at the tips of $(M_{\operatorname{cap}}, g_{\operatorname{cap}})$ to obtain $(M_{\operatorname{cap}}^{\operatorname{GL}}, g_{\operatorname{cap}}^{\operatorname{GL}})$ satisfying
	\begin{equation} \label{eq:m.k.plus.reconstruction.scal.gl}
		R \geq c_1 \delta(t_1)^{-2} r(t_1)^{-2} \text{ on } (M_{\operatorname{cap}}^{\operatorname{GL}}, g_{\operatorname{cap}}^{\operatorname{GL}}).
	\end{equation}
	Note that $M_- \approx X \cup_\sim M_{\operatorname{cap}}^{\operatorname{GL}}$ where $\sim$ identifies the common boundary points along $X$, $M_{\operatorname{cap}}^{\operatorname{GL}}$, with suitable orientation.
	
	\textbf{Step 1} (isotopy with scalar curvature control). Following \cite[p. 2838]{KleinerLott:notes}, let $U \subset \Omega$ be an $\eps$-tube within the $\eps$-horn $\Omega$ ($\subset M_-$), whose center has scalar curvature
	\begin{equation} \label{eq:m.k.plus.reconstruction.scal.tube}
		R \sim 3 c_2 r(t_1)^{-2},
	\end{equation}
	where $c_2 \in \RR$ is a large constant to be determined. Such an $\eps$-tube exists by the surgery construction, as long as the surgery parameter $\delta(t_1)$ (see \eqref{eq:m.k.plus.reconstruction.scal.gl}) is sufficiently small depending on $c_2$. Note that there are infinitely many such $\eps$-tubes $U$. Among them, we choose one that is closest to the surgery region. Thus, if we call $V$ the connected component of $\Omega \setminus U$ containing the surgery region, it follows by definition, \eqref{eq:m.k.plus.reconstruction.scal.tube}, and the surgery construction that
	\begin{equation} \label{eq:m.k.plus.reconstruction.scal.inside.tube}
		R \geq 2 c_2 r(t_1)^{-2} \text{ on } (V, (g_X \sqcup g_{\operatorname{cap}}^{\operatorname{GL}}) \restr V).
	\end{equation}
	At this point, we follow \cite[Section 6]{Marques:deforming.psc} to construct an isotopy
	\[ [0,1] \ni \mu \mapsto g_\mu \in \operatorname{Met}(M_-) = \operatorname{Met}(X \cup_\sim M_{\operatorname{cap}}^{\operatorname{GL}}) \] 
	satisfying $g_0 \equiv g_-$ and $g_1 \equiv g_X \sqcup g_{\operatorname{cap}}^{\operatorname{GL}}$ and
	\begin{enumerate}
		\item $\mu \mapsto g_\mu$ is constant everywhere on $M_- \setminus V$;
		\item $R \geq c_2 r(t_1)^{-2}$ on $(V, g_\mu \restr V)$ for all $\mu \in [0,1]$.
	\end{enumerate}
	The reference applies verbatim as long as we replace the positive scalar curvature isotopy in \cite[Proposition 3.3]{Marques:deforming.psc} with our strongly positive scalar curvature isotopy from Lemma \ref{lemm:psc.round.foliations.isotopy}, and the regular Gromov--Lawson connected-sum construction for families in \cite[Proposition 6.1]{Marques:deforming.psc} with Corollary \ref{coro:m.k.plus.surgery.families}. Note that we have used Corollary \ref{coro:m.k.plus.surgery.families}'s (2) to derive conclusion (2) above, and nowhere do we use Corollary \ref{coro:m.k.plus.surgery.families}'s (3).
	
	\textbf{Step 2} ($\lambda_1(-\Delta+kR)$ along the isotopy). First, by the Rayleigh quotient characterization of the first eigenvalue, applied with a test function supported on another $\eps$-tube $U'$ near $\partial \Omega$ that is fully contained in $X$ (and thus does not intersect $V$), we have:
	\begin{equation} \label{eq:m.k.plus.reconstruction.lambda.upper}
		\lambda_1(\mu) \leq c_3 r(t_1)^{-2} \text{ on } (M_-, g_{\mu}) \text{ for every } \mu \in [0,1],
	\end{equation}
	where $\lambda_1(\mu)$ denotes $\lambda_1(-\Delta+kR)$ of $(M_-, g_{\mu})$ and the constant $c_3 \in \RR$ is universal. In particular, if we choose $c_2$ sufficiently large depending on $c_3$ and on $k$, we will have
	\begin{equation} \label{eq:m.k.plus.reconstruction.scal.lambda}
		kR \geq \lambda_1(\mu) + 1 \text{ on } (V, g_\mu \restr V) \text{ for all } \mu \in [0,1],
	\end{equation}
	by conclusion (2) of Step 1. The control of $\lambda_1(\mu)$ now follows as in \cite[p. 2838-2839]{KleinerLott:notes} with Lemma \ref{lemm:kleiner.lott.agmon} in place of \cite[Lemma 93.21]{KleinerLott:notes}. The reference's proof carries through verbatim to show that every $\lambda_1(\mu)$ is close to the first Dirichlet eigenvalue of $-\Delta + kR$ on $(X, g_X)$, which is independent of $\mu$ and close to $\lambda_1(-\Delta + kR)$ of both $(M_{\pm}, g_{\pm})$, so $\lambda_1(\mu) \geq \lambda$ for all $\mu \in [0,1]$ by \eqref{eq:m.k.plus.path.lambda.lower.bound}.
\end{proof}

We proceed to Theorems \ref{theo:m.k.plus.path} and \ref{theo:m.k.plus.path.yamabe}, the proofs of both of which invoke  \cite{BamlerKleiner:contractibility}. 

\begin{proof}[Proof of Theorem \ref{theo:m.k.plus.path}]
	Let $g_\infty \in \sM^{>0}_{\infty}(M)$ be a fixed auxiliary metric on $M$, and let $g \in \sM^{\geq 0}_{k}(M)$ be arbitrary. It suffices to show the existence of a continuous path through $\sM^{>0}_k(M)$ from $g$ to $g_\infty$. This is done in two steps. First, there exists a continuous path through $\sM^{>0}_k(M)$ from $g$ to $\psi^* g_\infty$, for some $\psi \in \operatorname{Diff}_+(M)$. This is due to Theorem \ref{theo:m.k.plus.path.moduli}. Second, there exists a continuous path through $\sM^{>0}_\infty(M) \subset \sM^{>0}_k(M)$ (the inclusion by \eqref{eq:m.k.plus.inclusion}) from $\psi^* g_\infty$ to $g_\infty$. This is the content of \cite{BamlerKleiner:contractibility}.
\end{proof}

For Theorem \ref{theo:m.k.plus.path.yamabe}, in an earlier version of the paper we relied on the now well-understood theory of the long-time behavior of the Yamabe flow following the deep work of Brendle \cite{Brendle:yamabe.low.dim} to prove that $\sM_{1/8}^{>0}(M)$ is contractible. After a talk of the second author at Oberwolfach, Bernd Ammann indicated that the contractibility of the space of conformal factors alone suffices for proving Theorem \ref{theo:m.k.plus.path.yamabe}. We present this simpler argument here:

\begin{proof}[Proof of Theorem \ref{theo:m.k.plus.path.yamabe}]
	Without loss of generality, $M$ is connected. In the disconnected case, we concatenate our null-homotopies across components of $M$.
	
	We show that $\sM^{>0}_{1/8}(M)$, $\sM^{\geq 0}_{1/8}(M)$ are weakly contractible, i.e., that their homotopy groups vanish. So, fix a homotopy dimension $k \in \NN$ and consider any continuous map
	\[ \SS^k \ni \theta \mapsto g_{\theta,0} \in \sM^{>0}_{1/8}(M) \; (\sM^{\geq 0}_{1/8}(M)). \]
	For each $\theta \in \SS^k$, let $u_\theta \in C^\infty(M)$ be the unique positive first eigenfunction of $-\Delta_{g_\theta} + \tfrac18 R_{g_\theta}$ with unit $L^2(d\mu_{g_\theta})$ norm. Then, consider the continuous family\footnote{$\theta \mapsto u_\theta \in C^\infty(M)$ is continuous when $\theta \mapsto g_\theta \in \operatorname{Met}(M)$ is because, e.g., $\theta \mapsto \lambda_1(-\Delta_{g_\theta} + \tfrac18 R_{g_\theta})$ is continuous (see \cite[Chapter IV, \S 3.5]{Kato:perturbation.theory}, or the proof of \cite[Lemma A.1]{MantoulidisSchoen:bartnik})}
	\[ g_{\theta,t} := \big[ (1-t) + t u_\theta \big]^4 g_\theta \text{ for all } (\theta, t) \in \SS^k \times [0,1]. \]
	For each $t \in [0,1]$, $g_{\theta,t}$ and $g_\theta$ are in the same conformal class, so by Lemma \ref{lemm:m.k.plus.equivalences} 
	\[ g_{\theta,t} \in \sM^{>0}_{1/8}(M) \; (\sM^{\geq 0}_{1/8}(M)) \text{ for all } (\theta, t) \in \SS^k \times [0,1]. \]
	Moreover, at $t=1$ we have by \eqref{eq:scalar.curvature.conformal} and our choice of $u_\theta$ that
	\[ g_{\theta,1} \in \sM^{>0}_{\infty}(M) \; (\sM^{\geq 0}_{\infty}(M)) \text{ for all } \theta \in \SS^k. \]
	Since $M$ is topologically PSC, we can run Ricci flow starting at $g_{\theta, 1}$, $\theta \in \SS^k$, for a small uniform time $\eps > 0$ to get a continuous family
	\[ g_{\theta,t} \in \sM^{>0}_{\infty}(M) \text{ for all } (\theta, t) \in \SS^k \times [1, 1+\eps]. \]
	In particular, $\theta \mapsto g_{\theta,1+\eps}$ is an element of $\pi_k(\sM^{>0}_{\infty}(M))$. By Bamler--Kleiner \cite{BamlerKleiner:contractibility}, this element is null-homotopic in $\sM^{>0}_{\infty}(M)$. Concatenating homotopies, $\theta \mapsto g_{\theta,0}$ is null-homotopic in $\sM^{>0}_{1/8}(M)$ ($\sM^{\geq 0}_{1/8}(M)$). This completes the proof that $\sM^{>0}_{1/8}(M)$ and $\sM^{\geq 0}_{1/8}(M)$ are weakly contractible. This also implies the contractibility of $\sM^{>0}_{1/8}(M)$ by Whitehead's theorem, since this weakly contractible space has the homotopy type of a CW complex by \cite[Corollary 1]{Milnor:cw.complex}, being an open subset of the separable and metrizable manifold $\operatorname{Met}(M)$.
\end{proof}

\section{PSC almost-cobordance tools} \label{sec:psc.cobordances}

We will need a refinement of the existence result of Theorem \ref{theo:m.k.plus.path}. 

\begin{prop}[Fundamental path, 3D] \label{prop:m.k.plus.pairs}
	Suppose that $M$ is as in Theorem \ref{theo:m.k.plus.path} and that $g_L, g_R \in \sM^{\geq 0}_k(M)$ with $k \in [\tfrac14, \infty)$. 
	
	Then, there exists a smooth path $(g^o_t)_{t \in [0,1]}$ of metrics on $M$ with:
	\begin{enumerate}
		\item $g^o_0 = g_L$;
		\item $g^o_1 = g_R$;
		\item $g^o_t \in \sM^{>0}_{k}(M)$ for all $t \in (0, 1)$.
		\item[(4.a)] $\tfrac{d}{dt} \lambda_1(-\Delta_{g^o_t} + k R_{g^o_t}) > 0$ at $t=0$;
		\item[(4.b)] $\tfrac{d}{dt} \lambda_1(-\Delta_{g^o_t} + k R_{g^o_t}) < 0$ at $t=1$.
	\end{enumerate}
\end{prop}

\begin{rema}[Fundamental path, 2D] \label{rema:m.k.plus.pairs.2.dim}
	Proposition \ref{prop:m.k.plus.pairs} remains valid for closed orientable topologically PSC 2-manifolds $M$ and all $k \in (0, \infty)$. 
	
	Conclusions (1)-(3) follow from Proposition \ref{prop:m.k.2.dim.connected}. Conclusions (4.a), (4.b) follow as in \cite[Lemma 2.1]{ChauMartens:degenerate}. 
\end{rema}

\begin{rema}[Fundamental path, all dimensions] \label{rema:m.k.plus.pairs.nearly.constant}
	Proposition \ref{prop:m.k.plus.pairs} remains valid for closed $n$-manifolds $M$ ($n \geq 3$) and all $k \in (0, \infty)$ under sufficiently restrictive hypotheses, such as:
	\begin{enumerate}
		\item[(i)] $g_L, g_R \in \sM^{>0}_{k}(M)$ are sufficiently close to rescalings of one another, or
		\item[(ii)] $g_L$, $g_R \in \operatorname{LinClos}[\sM^{>0}_{k}(M)]$ coincide up to scaling.
	\end{enumerate}
	See Remark \ref{rema:bartnik.bray.closure.condition} for ways in which the proposition can otherwise fail.
\end{rema}

\begin{proof}[Proof of Proposition \ref{prop:m.k.plus.pairs}]
	The construction takes three steps.
	
	\textbf{Step 1} (constructing $(g^o_t)_{t \in [0,1/3]}$). First run smooth Ricci flow starting at $g^o_0 = g_L$ for some short amount of time $t \in [0, \varepsilon]$, $\varepsilon < \tfrac13$. By Lemma \ref{lemm:m.k.plus.ricci.flow.positive}, this smooth path already satisfies conclusions (1), (3), (4.a) for $t \in [0,\varepsilon]$. We then invoke Theorem \ref{theo:m.k.plus.path.moduli} to extend the path continuously to the time interval $t \in [\varepsilon,\tfrac13]$ so that $g^o_{1/3} \in \sM^{>0}_{\infty}(M)$ and $g^o_t \in \sM^{>0}_k(M)$ for all $t \in [\varepsilon,\tfrac13]$. Thus, conclusions (1), (3), (4.a) have been arranged to hold for $t \in [0, \tfrac13]$, and $g^o_{1/3} \in \sM^{>0}_\infty(M)$.
	
	\textbf{Step 2} (constructing $(g^o_t)_{t \in [2/3,1]}$). Run the same process as Step 1, except starting at $g_R$ and ultimately reparametrizing time by $t \mapsto 1-t$ to obtain $(g^o_t)_{t \in [2/3,1]}$ that satisfies conclusions (2), (3), (4.b) for $t \in [\tfrac23, 1]$, and with $g^o_{2/3} \in \sM^{>0}_{\infty}(M)$.

	\textbf{Step 3} (constructing $(g^o_t)_{t \in [1/3,2/3]}$). Note that $g^o_{1/3}, g^o_{2/3} \in \sM^{>0}_{\infty}(M)$ and that this space is path-connected by the recent breakthrough of Bamler--Kleiner \cite{BamlerKleiner:contractibility}, so we can extend our construction to a continuous $(g^o_t)_{t \in [0,1]}$ with $g^o_t \in \sM^{>0}_{\infty}(M)$ for all $t \in [\tfrac13, \tfrac23]$, obtaining (3) in full.

	The result follows by smoothing for $t \in [\tfrac12 \varepsilon, 1-\tfrac12 \varepsilon]$ as in \cite[Section 2]{CabreraMiao:higher.dim} and recalling the inclusion in \eqref{eq:m.k.plus.inclusion}.
\end{proof}

\begin{prop}[Volume normalization and Moser twist] \label{prop:m.k.plus.pairs.twist}
	Suppose that $M$ is a closed connected $n$-manifold and that $(g^o_t)_{t \in [0,1]}$ is a smooth path of metrics on $M$. 
	
	Then, there exists a smooth path $(\sigma_t)_{t \in [0,1]}$ of positive scalars and a smooth path $(\psi_t)_{t \in [0,1]}$ of diffeomorphisms of $M$ isotopic to the identity ($\psi_t \in \operatorname{Diff}_0(M)$), so that:
	\begin{enumerate}
		\item $\psi_0 \equiv \operatorname{Id}$ and $\sigma_0 = 1$;
		\item the Riemannian volume form induced on every slice 
			\[ M \times \{t\} \subset (M \times [0,1], \sigma_t \psi_t^* g^o_t + dt^2) \]
			is $t$-independent.
	\end{enumerate}
	If $g_t := \sigma_t \psi_t^* g_t^o$, we say that $(g_t)_{t \in [0,1]}$ is obtained from $(g_t^o)_{t \in [0,1]}$ by normalizing volume and performing a Moser twist.
\end{prop}

\begin{rema} \label{rema:m.k.plus.pairs}
	Item (2) above is equivalent to the statement:
	\begin{equation*}
		\text{every } M \times \{t\} \text{ is a stable minimal hypersurface in } (M \times [0,1], \sigma_t \psi_t^* g_t + dt^2).
	\end{equation*}
	Indeed, if all $M \times \{t\}$ are (stable) minimal, then their volume form is constant by the first variation formula. Conversely, if their volume form is constant, then all $M \times \{t\}$ are minimal by the first variation formula and (degenerate-)stable by the second variation formula.
\end{rema}

\begin{proof}[Proof of Proposition \ref{prop:m.k.plus.pairs.twist}]
	We fix $\sigma_t$ once and for all by requiring that
	\[ \tfrac{d}{dt} \vol_{\sigma_t g^o_t}(M) \equiv 0 \text{ for } t \in [0,1]. \]
	We determine $\psi_t$ using Moser's trick as in \cite[Lemma 1.2]{MantoulidisSchoen:bartnik}. We recall the argument here for the reader's convenience. For all $t \in [0,1]$,
	\begin{equation} \label{eq:m.k.plus.pairs.twist.rhs}
		\int_M \left( \tfrac12 \Tr_{\sigma_t g^o_t} \tfrac{d}{dt} (\sigma_t g^o_t) \right) \, d\mu_{\sigma_t g^o_t} = \tfrac{d}{dt} \vol_{\sigma_t g^o_t}(M) = 0,
	\end{equation}
	so the divergence equation 
	\begin{equation} \label{eq:m.k.plus.pairs.twist}
		\Div_{\sigma_t g^o_t} X_t = - \tfrac12 \Tr_{\sigma_t g^o_t} \tfrac{d}{dt} (\sigma_t g^o_t)
	\end{equation}
	is smoothly solvable in time.\footnote{After normalizing by a constant, we can smoothly in $t$ solve $\Delta_{\sigma_t g^o_t} f_t = - \tfrac12 \Tr_{\sigma_t g^o_t} \tfrac{d}{dt} (\sigma_t g^o_t)$ in view of \eqref{eq:m.k.plus.pairs.twist.rhs} and elliptic regularity. Then, take $X_t := \nabla_{\sigma_t g^o_t} f_t$.} Then, setting $\psi_t$ as the integral flow of the vector field $X_t$ we get:
	\begin{align*}
		\tfrac{d}{dt} d\mu_{\sigma_t \psi_t^* g^o_t} 
			& = \tfrac{d}{dt} (\psi_t^* d\mu_{\sigma_t g^o_t}) = \psi_t^* \left( \tfrac{d}{dt} d\mu_{\sigma_t g^o_t} + \mathcal{L}_{X_t} d\mu_{\sigma_t g^o_t} \right) \\
			& = \psi_t^* \left( \tfrac12 \Tr_{\sigma_t g^o_t} \tfrac{d}{dt} (\sigma_t g^o_t) + \Div_{\sigma_t g^o_t} X_t \right) \, d\mu_{\sigma_t g^o_t} = 0,
	\end{align*}
	where the last equation follows from \eqref{eq:m.k.plus.pairs.twist}. 
\end{proof}

We can now construct our monotone PSC almost-cobordances:

\begin{prop}[Monotone PSC almost-cobordance, I] \label{prop:m.k.plus.pairs.psc.minimal}
	Suppose that $M$ is a closed connected $n$-manifold and that $g_L, g_R \in \sM^{>0}_{1/2}(M)$ have $\vol_{g_L}(M) = \vol_{g_R}(M)$. 
	
	Let $(g_t)_{t \in [0,1]}$ be obtained by applying Proposition \ref{prop:m.k.plus.pairs.twist} to a smooth path of metrics $(g^o_t)_{t \in [0,1]}$ on $M$ that satisfies Proposition \ref{prop:m.k.plus.pairs}'s (1)-(3) with $g_L$, $g_R$, $k = \tfrac12$.

	Then, there exists a metric $\bm{h}$ on $\bm{N} := M \times [0,1]$ such that:
	\begin{enumerate}
		\item $\bm{h} \restr (M \times \{0\}) = g_L$;
		\item $\bm{h} \restr (M \times \{1\}) = \psi^* g_R$ for some $\psi \in \operatorname{Diff}_0(M)$;
		\item $R_{\bm{h}} > 0$ on $\bm{N}$;
		\item $M \times \{t\} \subset (\bm{N}, \bm{h})$ is stable minimal for all $t \in [0,1]$.
	\end{enumerate}
	One may, in fact, set $\bm{h} = g_t + A^2 u_t^2 dt^2$ where:
	\begin{itemize}
		\item $(u_t)_{t \in [0,1]}$ is a smooth path of first eigenfunctions of $(-\Delta_{g_t} + \tfrac12 R_{g_t})_{t \in [0,1]}$,
		\item $A \in (0, \infty)$ is sufficiently large depending on $(g_t)_{t \in [0,1]}$, $(u_t)_{t \in [0,1]}$.
	\end{itemize}
\end{prop}
\begin{proof} 
	This was carried out in \cite{MantoulidisSchoen:bartnik} in a narrower setting and with more tedious computations, so we redo a simpler proof. 
	
	See Remark \ref{rema:differentiating.lambda1} for the existence of $(u_t)_{t \in [0,1]}$. Note that our ansatz for $\bm{h}$ already satisfies conclusions (1)-(2) by Proposition \ref{prop:m.k.plus.pairs}'s (1)-(2) and Proposition \ref{prop:m.k.plus.pairs.twist}'s (1)-(2). It remains to check conclusions (3)-(4). Lemma \ref{lemm:slicing.formula.curv} gives the relationship 
	\begin{equation} \label{eq:m.k.plus.pairs.psc.minimal.sff}
		\bar{\sff}_t = \frac{1}{A u_t} \sff_t
	\end{equation}
	between the second fundamental form $\bar{\sff}_t$ of $M \times \{t\} \subset (M \times [0,1], \bm{h})$ and $\sff_t$ of $M \times \{t\} \subset (M \times [0,1], g_t + dt^2)$. Likewise, we have the relationship
	\begin{equation} \label{eq:m.k.plus.pairs.psc.minimal.meancurv}
		\bar{H}_t = \frac{1}{A u_t} H_t \equiv 0
	\end{equation}
	between the mean curvature $\bar{H}_t$ of $M \times \{t\} \subset (M \times [0,1], \bm{h})$ and $H_t$ of $M \times \{t\} \subset (M \times [0,1], g_t + dt^2)$, where the last equality uses Proposition \ref{prop:m.k.plus.pairs.twist}'s (2), forcing $H_t \equiv 0$. This implies conclusion (4); see Remark \ref{rema:m.k.plus.pairs}.
	
	Proposition \ref{prop:m.k.plus.pairs}'s (1)-(3) imply $\lambda_1(-\Delta_{g^o_t} + \tfrac12 R_{g^o_t}) > 0$ for all $t \in [0, 1]$ and thus
	\begin{equation} \label{eq:m.k.plus.pairs.psc.minimal.lambda}
		\lambda_1(-\Delta_{g_t} + \tfrac12 R_{g_t}) > 0 \text{ for all } t \in [0, 1],
	\end{equation}
	since the sign of $\lambda_1(-\Delta + \tfrac12 R)$ is invariant under rescaling and diffeomorphisms of the underlying metric. Writing $\lambda_1(t) := \lambda_1(-\Delta_{g_t} + \tfrac12 R_{g_t})$ for brevity, Lemma \ref{lemm:slicing.formula.curv} gives:
	\begin{equation} \label{eq:m.k.plus.pairs.minimal.scal}
		R_{\bm{h}}|_{M \times \{t\}} = 2 \lambda_1(t) - 2 A^{-1} u_t^{-1} \tfrac{\partial}{\partial t} \bar{H}_t - \bar{H}_t^2 - |\bar{\sff}_t|^2,
	\end{equation}
	For large $A$, \eqref{eq:m.k.plus.pairs.psc.minimal.sff}, \eqref{eq:m.k.plus.pairs.psc.minimal.meancurv}, \eqref{eq:m.k.plus.pairs.psc.minimal.lambda} tell us that the first term of the right hand side of \eqref{eq:m.k.plus.pairs.minimal.scal} dominates the rest, thus implying conclusion (3). 
\end{proof}

\begin{rema}
	The reason we call the construction above an ``almost'' cobordance is because we cannot prescribe both boundary metrics: we have to allow for a pullback by an element of $\operatorname{Diff}_0(M)$ on one side.
\end{rema}

We will also be interested in the following variant of the previous result, where we produce mean-convex foliations of the interior rather than minimal foliations.

\begin{prop}[Monotone PSC almost-cobordance, II] \label{prop:m.k.plus.pairs.psc.mean.convex}
	Suppose that $M$ is a closed connected $n$-manifold and that $g_L, g_R \in \sM^{\geq 0}_{1/2}(M)$ have $\vol_{g_L}(M) = \vol_{g_R}(M)$. 
	
	Let $(g_t)_{t \in [0,1]}$ be obtained by applying Proposition \ref{prop:m.k.plus.pairs.twist} to a path of metrics $(g^o_t)_{t \in [0,1]}$ on $M$ satisfying:
	\begin{itemize}
		\item Proposition \ref{prop:m.k.plus.pairs}'s (1)-(3) with $g_L$, $g_R$, $k = \tfrac12$,
		\item Proposition \ref{prop:m.k.plus.pairs}'s (4.a) if $g_L \not \in \sM^{>0}_{1/2}(M)$, and
		\item Proposition \ref{prop:m.k.plus.pairs}'s (4.b) if $g_R \not \in \sM^{>0}_{1/2}(M)$.
	\end{itemize}
	Then, for every $\eps > 0$, there exists a metric $\bm{h}$ on $\bm{N} := M \times [0,1]$ such that:
	\begin{enumerate}
		\item $\bm{h} \restr (M \times \{0\}) = g_L$;
		\item $\bm{h} \restr (M \times \{1\}) = (1+\eps) \psi^* g_R$, for some $\psi \in \operatorname{Diff}_0(M)$;
		\item $R_{\bm{h}} > 0$ on $\bm{N} \setminus (M \times \{0\})$;
		\item $R_{\bm{h}}|_{M \times \{0\}} > 0$ unless $g_L \not \in \sM^{>0}_{1/2}(M)$;
		\item $M \times \{t\} \subset (\bm{N}, \bm{h})$ has mean curvature vectors pointing strictly toward $M \times \{0\}$ for all $t \in (0, 1]$;
		\item $M \times \{0\} \subset (\bm{N}, \bm{h})$ is stable minimal.
	\end{enumerate}
	One may, in fact, set $\bm{h} = (1 + \eps \rho(t)) g_{\tau(t)} + A^2 u_{\tau(t)}^2 dt^2$, where:
	\begin{itemize}
		\item $(u_t)_{t \in [0,1]}$ is a smooth path of first eigenfunctions of $(-\Delta_{g_t} + \tfrac12 R_{g_t})_{t \in [0,1]}$,
		\item $A \in (0, \infty)$ is sufficiently large depending on $(g_t)_{t \in [0,1]}$, $(u_t)_{t \in [0,1]}$, $\eps$.
		\item $\rho, \tau : [0,1] \to [0,1]$ are smooth increasing bijections of $[0,1]$ to itself depending on $(g_t)_{t \in [0,1]}$, $(u_t)_{t \in [0,1]}$, $\eps$.
	\end{itemize}
\end{prop}
\begin{proof}
	See Remark \ref{rema:differentiating.lambda1} for the existence of $(u_t)_{t \in [0,1]}$. Our ansatz for $\bm{h}$ already satisfies conclusions (1)-(2) by Proposition \ref{prop:m.k.plus.pairs}'s (1)-(2) and Proposition \ref{prop:m.k.plus.pairs.twist}'s (1)-(2). It remains to check conclusions (3)-(6). We will require that:
	\begin{equation} \label{eq:m.k.plus.pairs.psc.mean.convex.rho}
		\rho'(0) = 0 \text{ and } \rho'(t) > 0 \text{ for all } t \in (0,1];
	\end{equation}
	\begin{equation} \label{eq:m.k.plus.pairs.psc.mean.convex.tau}
		\tau'(t) > 0 \text{ for all } t \in (0,1).
	\end{equation}
	Lemma \ref{lemm:slicing.formula.curv} gives the relationship
	\begin{align} \label{eq:m.k.plus.pairs.psc.mean.convex.sff}
		\bar{\sff}_t 
			& = \frac{\eps \rho'(t) g_{\tau(t)} + (1 + \eps \rho(t)) [\tfrac{d}{dt} g_{t}]_{t=\tau(t)} \tau'(t)}{2 A u_{\tau(t)}} \nonumber \\
			& = \frac{\tfrac12 \eps \rho'(t) g_{\tau(t)} + (1 + \eps \rho(t)) \tau'(t) \sff_{\tau(t)}}{A u_{\tau(t)}}
	\end{align}
	between the second fundamental form $\bar{\sff}_t$ of $M \times \{t\} \subset (M \times [0,1], \bm{h})$ and $\sff_{\tau(t)}$ of $M \times \{\tau(t)\} \subset (M \times [0,1], g_t + dt^2)$. Likewise, we have the relationship
	\begin{equation} \label{eq:m.k.plus.pairs.psc.mean.convex.meancurv}
		\bar{H}_t = \frac{\tfrac12 \eps \rho'(t) \Tr_{g_{\tau(t)}} g_{\tau(t)} + (1+\eps\rho(t)) \tau'(t) H_{\tau(t)}}{(1+\eps \rho(t)) A u_{\tau(t)}} =  \frac{n \eps \rho'(t)}{2 (1+\eps \rho(t)) A u_{\tau(t)}} 
	\end{equation}
	between the mean curvature $\bar{H}_t$ of $M \times \{t\} \subset (M \times [0,1], \bm{h})$ and $H_{\tau(t)}$ of $M \times \{\tau(t)\} \subset (M \times [0,1], g_t + dt^2)$, where the last equality makes use of Proposition \ref{prop:m.k.plus.pairs.twist}'s (2), which forces $H_{\tau(t)} \equiv 0$. Conclusions (5)-(6) follow from \eqref{eq:m.k.plus.pairs.psc.mean.convex.rho},  \eqref{eq:m.k.plus.pairs.psc.mean.convex.meancurv}.

	It remains to control the scalar curvature, which by Lemma \ref{lemm:slicing.formula.curv} is:
	\begin{equation} \label{eq:m.k.plus.pairs.psc.mean.convex.scal}
		R_{\bm{h}}|_{M \times \{t\}} = \frac{2 \lambda_1(\tau(t))}{1 + \eps \rho(t)} - \frac{2 \tfrac{\partial}{\partial t} \bar{H}_t}{A u_{\tau(t)}} - \bar{H}_t^2 - |\bar{\sff}_t|^2,
	\end{equation}
	where $\lambda_1(t) := \lambda_1(-\Delta_{g_{t}} + \tfrac12 R_{g_{t}})$ and $|\bar{\sff}_t|^2$ is the squared norm of $\bar{\sff}_t$ with respect to $(1+\eps \rho(t)) g_{\tau(t)}$. Using $H_{\tau(t)} \equiv 0$ again, \eqref{eq:m.k.plus.pairs.psc.mean.convex.sff} implies:
	\begin{equation} \label{eq:m.k.plus.pairs.psc.mean.convex.sff.sq}
		|\bar{\sff}_t|^2 = \frac{\tfrac14 n \eps^2 \rho'(t)^2 + (1+\eps\rho(t))^2 \tau'(t)^2 |\sff_{\tau(t)}|^2}{(1+\eps \rho(t))^2 A^2 u_{\tau(t)}^2},
	\end{equation}
	where $|\sff_{\tau(t)}|^2$ is squared norm of $\sff_{\tau(t)}$ with respect to $g_{\tau(t)}$. Similarly, \eqref{eq:m.k.plus.pairs.psc.mean.convex.meancurv} gives:
	\begin{equation} \label{eq:m.k.plus.pairs.psc.mean.convex.meancurv.sq}
		\bar{H}_t^2 = \frac{n^2 \eps^2 \rho'(t)^2}{4(1+\eps \rho(t))^2 A^2 u_{\tau(t)}^2},
	\end{equation}
	and
	\begin{equation} \label{eq:m.k.plus.pairs.psc.mean.convex.meancurv.ddt}
		\frac{\tfrac{\partial}{\partial t} \bar{H}_t}{A u_{\tau(t)}} = \frac{n \eps \rho''(t)}{2(1+\eps\rho(t)) A^2 u_{\tau(t)}^2} - 2n^{-1} \bar{H}_t^2 - \frac{\bar{H}_t \tau'(t) [\tfrac{\partial}{\partial t} \log u_t]_{t=\tau(t)}}{A u_{\tau(t)}}.
	\end{equation}
	Combining \eqref{eq:m.k.plus.pairs.psc.mean.convex.sff.sq}, \eqref{eq:m.k.plus.pairs.psc.mean.convex.meancurv.sq}, \eqref{eq:m.k.plus.pairs.psc.mean.convex.meancurv.ddt} and some algebraic manipulation:
	\begin{align} \label{eq:m.k.plus.pairs.psc.mean.convex.scal.2}
		R_{\bm{h}}|_{M \times \{t\}} 
		& = \frac{2 \lambda_1(\tau(t))}{1+\eps \rho(t)} - \frac{n \eps \rho''(t)}{(1+\eps \rho(t)) A^2 u_{\tau(t)}^2} - \frac{(3-n)n \eps^2 \rho'(t)^2}{4 (1+\eps\rho(t))^2 A^2 u_{\tau(t)}^2} \nonumber \\
		& \qquad  - \frac{|\sff_{\tau(t)}|^2 \tau'(t)^2}{A^2 u_{\tau(t)}^2} + \frac{n \eps \rho'(t) \tau'(t) [\tfrac{\partial}{\partial t} \log u_t]_{t=\tau(t)}}{(1+\eps\rho(t)) A^2 u_{\tau(t)}^2}.
	\end{align} 

	Our next goal is to ensure conclusions (3)-(4). To that end, first note that, as in the proof of the previous proposition, Proposition \ref{prop:m.k.plus.pairs}'s (3) and the invariance of the sign of $\lambda_1(-\Delta + \tfrac12 R)$ under rescaling and diffeomorphisms gives the crude estimate
	\begin{equation} \label{eq:m.k.plus.psc.mean.convex.lambda1.positive}
		\lambda_1(t) > 0 \text{ for all } t \in (0, 1).
	\end{equation} 
	We improve on this as follows:\medskip
	
	\textbf{Step 1} (arranging conclusion (3) near $t=0$, and conclusion (4)).

			\begin{itemize}
				\item \textbf{Case A}: $g_L \in \sM^{>0}_{1/2}(M)$, i.e., $\lambda_1(0) > 0$. Take
					\[ \tau(t) = t, \; \rho(t) = t^2 \text{ near } t=0. \]
					For sufficiently large $A$ depending on $n$, $\eps$, $(u_t)_{t \in [0,1]}$, $(|\sff_t|)_{t \in [0,1]}$, the first term of the right hand side of \eqref{eq:m.k.plus.pairs.psc.mean.convex.scal.2}, which is uniformly positive near $t=0$ by $\lambda_1(0) > 0$, dominates all remaining terms near $t=0$. Therefore,
					\[ R_{\bm{h}}|_{M \times \{t\}} \geq \lambda_1(t) > 0 \text{ near } t=0. \]
					This ensures conclusion (3) near $t=0$, and (4) since $\lambda_1(0) > 0$.
					
				\item \textbf{Case B}: $g_L \not \in \sM^{>0}_{1/2}(M)$, i.e., $\lambda_1(0) = 0$. By Proposition \ref{prop:m.k.plus.pairs}'s (4.a),
					\[ \left[ \tfrac{d}{dt} \lambda_1(-\Delta_{g^o_t} + \tfrac12 R_{g^o_t}) \right]_{t=0} > 0. \]
					Recall that $g_t = \sigma_t \psi_t^* g^o_t$ for $\sigma_t \in (0, \infty)$, $\psi_t \in \operatorname{Diff}(M)$. Using 
					\[ \lambda_1(-\Delta_{\sigma_t \psi_t^* g_t} + \tfrac12 R_{\sigma_t \psi_t^* g_t}) = \sigma_t^{-1} \lambda_1(-\Delta_{g_t} + \tfrac12 R_{g_t}) \]
					together with $\lambda_1(-\Delta_{g_L} + \tfrac12 R_{g_L}) = 0$, we find that
					\[ \lambda_1'(0) = \left[ \tfrac{d}{dt} \lambda_1(-\Delta_{g_t} + \tfrac12 R_{g_t}) \right]_{t=0} > 0. \]
					Take $\alpha > 0$ sufficiently small depending on $\lambda_1'(0)$, so that $\tau(t) = \alpha t^2$ satisfies 
					\begin{equation} \label{eq:m.k.plus.psc.mean.convex.lambda1.nonnegative}
						(\tau'(t))^2 \leq \tfrac{1}{2} \lambda_1(\tau(t)) \text{ near } t=0.
					\end{equation}
					Take $A$ large enough depending on $n$, $\eps$, $(u_t)_{t \in [0,1]}$, $(|\sff_t|)_{t \in [0,1]}$, so that
					\[ \frac{n \eps(1 + |[\tfrac{\partial}{\partial t} \log u_t]_{t=\tau(t)}|)}{(1+\eps \rho(t)) A^2 u_{\tau(t)}^2} \leq 1, \; \frac{|\sff_{\tau(t)}|^2}{A^2 u_{\tau(t)}^2} \leq 1. \]
					We now take $\rho(t)$ to satisfy
					\[ \rho''(t) + \rho'(t) \leq \tfrac{1}{2} \lambda_1(\tau(t)) \text{ near } t=0. \]
					Then, \eqref{eq:m.k.plus.pairs.psc.mean.convex.scal.2} gives
					\[ R_{\bm{h}}|_{M \times \{t\}} \geq \tfrac12 \lambda_1(\tau(t)) \geq 0 \text{ near } t=0, \]
					with strict inequality except at $t=0$ by virtue of \eqref{eq:m.k.plus.psc.mean.convex.lambda1.nonnegative} since $\tau(t) = \alpha t^2$. This ensures (3) near $t=0$, while (4) is vacuous.
			\end{itemize}

	\textbf{Step 2} (arranging conclusion (3) near $t=1$). Step 1's strategy goes through verbatim except with $\tau(t) = 1 -\alpha t^2$ for small $\alpha$, and arranges for
	\[ (\tau'(t))^2 \leq \tfrac12 (1+\eps)^{-1} \lambda_1(\tau(t)). \]
	We also instead take $\rho(t)$ near $t=1$ to satisfy
	\[ -\rho''(t) + \tau'(t) \left[ \tfrac{\partial}{\partial t} \log u_t \right]_{t=\tau(t)} \rho'(t) > 0. \]
	The rest of the argument proceeds similarly to yield
	\[ R_{\bm{h}}|_{M \times \{t\}} > \tfrac12 (1+\eps)^{-1} \lambda_1(\tau(t)) \geq 0 \text{ near } t=1. \]
	We leave details to the reader.
	
	\textbf{Step 3} (arranging conclusion (3) away from $t=0$, $1$). At this point, $\rho$, $\tau$ are fixed near $t = 0$, $1$, and $A$ is bounded from below depending on $n$, $\eps$, $(u_t)_{t \in [0,1]}$, $(|\sff_t|)_{t \in [0,1]}$. If $J \subset [0,1]$ denotes the relatively open set near which Steps 1 and 2 were arranged, then note that by \eqref{eq:m.k.plus.psc.mean.convex.lambda1.positive} we have
	\[ \inf_{t \in [0,1] \setminus J} \lambda_1(\tau(t)) > 0. \]
	Now, extend $\rho$, $\tau$ arbitrarily to $[0,1]$ to be smooth bijections on $[0,1]$ satisfying \eqref{eq:m.k.plus.pairs.psc.mean.convex.rho}, \eqref{eq:m.k.plus.pairs.psc.mean.convex.tau}. We now take $A$ sufficiently large, depending only on $n$, $\eps$, $(u_t)_{t \in [0,1]}$, $(|\sff_t|)_{t \in [0,1]}$, and our choices for $\tau$, $\rho$, and arrange for
	\[ \bar{H}_t^2 + |\bar{\sff}_t|^2 + \frac{2 \tfrac{\partial}{\partial t} \bar{H}_t}{A u_{\tau(t)}} \leq \inf_{t \in [0,1] \setminus J} \frac{\lambda_1(\tau(t))}{1 + \eps \rho(t)}. \]
	Altogether, \eqref{eq:m.k.plus.pairs.psc.mean.convex.scal} gives
	\[ R_{\bm{h}}|_{M \times \{t\}} \geq c \lambda_1(\tau(t)) > 0 \text{ for } t \in [0,1] \setminus J \]
	too, and thus conclusion (3). This completes the proof of the proposition.
\end{proof}

\section{Proof of Theorem \ref{theo:m.k.plus.fillin}} \label{sec:proof.of.fillin}

We need the following lemma (which we will apply with $(\bm{M}_1, \bm{g}_1)$ having come from Proposition \ref{prop:m.k.plus.pairs.psc.mean.convex}) to prove Theorem \ref{theo:m.k.plus.fillin}.

\begin{lemm}[Smoothing a concatenation with no new minimal surfaces, I] \label{lemm:gluing.no.minimal.compact}
	Let $3 \leq n+1 \leq 7$. Consider two compact Riemannian $(n+1)$-manifolds-with-boundary $(\bm{M}_i, \bm{g}_i)$, $i=1$, $2$, with $\partial \bm{M}_i =: M_i^L \sqcup M_i^R$ for $i = 1, 2$, and
	\begin{itemize}
		\item $M_1^R \subset (\bm{M}_1, \bm{g}_1)$ is isometric to $M_2^L \subset (\bm{M}_2, \bm{g}_2)$; 
		\item $M_1^L \subset (\bm{M}_1, \bm{g}_1)$ and $M_2^L \subset (\bm{M}_2, \bm{g}_2)$ are stable minimal hypersurfaces, though $M_1^L$ is allowed to be empty;
		\item $M_1^R \subset (\bm{M}_1, \bm{g}_1)$ and $M_2^R \subset (\bm{M}_2, \bm{g}_2)$ have mean curvature vectors pointing strictly inward;
		\item $R_{\bm{g}_1} \geq 0$ on $\bm{M}_1$, strictly along $M_1^R$, and $R_{\bm{g}_2} \geq 0$ on $\bm{M}_2$;
		\item neither $(\bm{M}_i, \bm{g}_i)$ has closed interior minimal hypersurfaces.
	\end{itemize}
	Define 
	\[ (\hat{\bm{M}}, \hat{\bm{g}}) := (\bm{M}_1, \bm{g}_1) \sqcup_\sim (\bm{M}_2, \bm{g}_2), \]
	where $\sim$ identifies the isometric boundaries $M_1^R$, $M_2^L$ as a single hypersurface $\Sigma \subset \hat{\bm{M}}$. This is a compact Riemannian $(n+1)$-manifold-with-boundary with a  ``corner'' along $\Sigma$.\footnote{$\hat{\bm{M}}$ has a well-defined smooth structure; it is $\hat{\bm{g}}$ that has a ``corner'' in the sense of \cite{Miao:pmt.corners}.} Let $\bm{U}$ be a neighborhood of $\Sigma$ in $(\hat{\bm{M}}, \hat{\bm{g}})$. 
	
	Then, there exist smooth metrics $\hat{\bm{g}}_\delta$ on $\hat{\bm{M}}$, for $\delta > 0$ small, such that:
	\begin{enumerate}
		\item $\hat{\bm{g}}_\delta \to \hat{\bm{g}}$ in $C^0(\hat{\bm{M}})$ as $\delta \to 0$;
		\item $\hat{\bm{g}}_\delta \equiv \hat{\bm{g}}$ on $(\bm{M}_1 \setminus \bm{U}) \subset \hat{\bm{M}}$;
		\item $\hat{\bm{g}}_\delta \restr M_2^R \equiv (1+\delta) \hat{\bm{g}} \restr M_2^R$;
		\item $R_{\hat{\bm{g}}_\delta} \geq 0$ on $\hat{\bm{M}}$;
		\item $(\hat{\bm{M}}, \hat{\bm{g}}_\delta)$  has no closed interior minimal hypersurfaces.
	\end{enumerate}
\end{lemm}
\begin{proof}
	First, we need the following:
	
	\begin{claim} \label{clai:gluing.no.minimal.foliation}
		There exists a smooth family $[0, \sigma] \ni t \mapsto f^{(t)} \in C^\infty(M_2^L)$ with $f^{(0)} \equiv 0$, $f^{(t)} > 0$ for $t \in (0, \sigma]$, such that
		\begin{multline} \label{eq:gluing.no.minimal.foliation}
			\text{the mean curvature vectors of } M_2^{L,t} := \graph^{\bm{g}_2}_{M_2^L} f^{(t)} \\
			\text{point strictly away from } M_2^R \text{ for all } t \in (0, \sigma].
		\end{multline}
	\end{claim}
	\begin{proof}[Proof of claim]
		If $M_2^L$ is strictly stable, then flowing in the inward-pointing direction of a first eigenfunction of the stability operator of $M_2^L$, the rate of change of mean curvatures is strictly positive (by the second variation formula). The result follows with $f^{(t)}$ being small multiples of a fixed first eigenfunction of the stability operator of $M_2^L$.
		
		If $M_2^L$ is degenerate stable, the result is only slightly more delicate and follows from the implicit function theorem as in \cite[p. 225]{Galloway:outermost.psc} or \cite[Lemma 10]{Song:yau.conjecture}.
	\end{proof}
	
	Denote by $\phi \in C^\infty(\bm{M}_2)$ the unique solution to
	\begin{equation} \label{eq:gluing.no.minimal.phi}
		\Delta_{\bm{g}_2} \phi \equiv 0 \text{ on } \bm{M}_2, \; \phi \equiv 0 \text{ along } M_2^L, \; \phi \equiv 1 \text{ along } M_2^R.
	\end{equation}
	By the Hopf boundary point lemma,
	\begin{equation} \label{eq:gluing.no.minimal.hopf.1.L}
		\nabla_{\bm{g}_2} \phi \text{ is strictly inward pointing along } M_2^L,
	\end{equation}
	\begin{equation} \label{eq:gluing.no.minimal.hopf.1.R}
		\nabla_{\bm{g}_2} \phi \text{ is strictly outward pointing along } M_2^R.
	\end{equation}
	For $\delta > 0$, set $\bm{g}_{2,\delta}' := (1+\delta\phi)^{4/(n-1)} \bm{g}_2$. Then, \eqref{eq:scalar.curvature.conformal} and \eqref{eq:gluing.no.minimal.phi} imply 
	\begin{equation} \label{eq:gluing.no.minimal.scalar.1}
		R_{\bm{g}_{2,\delta}'} \geq 0 \text{ on } \bm{M}_2,
	\end{equation}
	while \eqref{eq:mean.curvature.conformal} and \eqref{eq:gluing.no.minimal.hopf.1.L}, \eqref{eq:gluing.no.minimal.hopf.1.R} imply that, as long as $\delta > 0$ is small,
	\begin{multline} \label{eq:gluing.no.minimal.foliation.1}
		\text{Claim } \ref{clai:gluing.no.minimal.foliation} \text{ holds in } (\bm{M}_2, \bm{g}_{2,\delta}') \text{ with the same hypersurfaces } M_2^{L,t}  \\
		\text{ and } \eqref{eq:gluing.no.minimal.foliation} \text{ holds for  all } t \in [0, \sigma],
	\end{multline}
	and
	\begin{multline} \label{eq:gluing.no.minimal.meancurv.1.L}
		\text{the mean curvature vectors of } M_2^L \subset (\bm{M}_2, \bm{g}_{2,\delta}') \\
		\text{are shorter than those of } M_1^R \subset (\bm{M}_1, \bm{g}_{1}),
	\end{multline}
	and
	\begin{equation} \label{eq:gluing.no.minimal.meancurv.1.R}
		\text{the mean curvature vectors of } M_2^R \subset (\bm{M}_2, \bm{g}_{2,\delta}') \text{ point strictly toward } M_2^L.
	\end{equation}
	Hold this $\delta$ fixed. Denote by $\zeta_\delta \in C^\infty(\bm{M}_2)$ the unique solution to
	\begin{equation} \label{eq:gluing.no.minimal.zeta}
		\Delta_{\bm{g}_{2,\delta}'} \zeta_\delta = -1 \text{ on } \bm{M}_2, \; \zeta_\eps \equiv 0 \text{ on } \partial \bm{M}_2.
	\end{equation}
	Set $\bm{g}_{2,\delta,\eps}'' := (1 + \eps \zeta_\delta)^{4/(n-1)}$. Then, \eqref{eq:scalar.curvature.conformal} and \eqref{eq:gluing.no.minimal.zeta} imply
	\begin{equation} \label{eq:gluing.no.minimal.scalar.2}
		R_{\bm{g}_{2,\delta,\eps}''} > 0 \text{ on } \bm{M}_2,
	\end{equation}
	and \eqref{eq:mean.curvature.conformal}, \eqref{eq:gluing.no.minimal.foliation.1}, \eqref{eq:gluing.no.minimal.meancurv.1.L}, \eqref{eq:gluing.no.minimal.meancurv.1.R} imply that, as long as $\eps > 0$ is sufficiently small, depending on $\delta$,
	\begin{multline} \label{eq:gluing.no.minimal.foliation.2}
		\text{Claim } \ref{clai:gluing.no.minimal.foliation} \text{ holds in } (\bm{M}_2, \bm{g}_{2,\delta,\eps}'') \text{ with the same hypersurfaces } M_2^{L,t}  \\
		\text{ and } \eqref{eq:gluing.no.minimal.foliation} \text{ holds for  all } t \in [0, \sigma],
	\end{multline}
	and
	\begin{multline} \label{eq:gluing.no.minimal.meancurv.2.L}
		\text{the mean curvature vectors of } M_2^L \subset (\bm{M}_2, \bm{g}_{2,\delta,\eps}'') \\
		\text{are shorter than those of } M_1^R \subset (\bm{M}_1, \bm{g}_{1}),
	\end{multline}
	and
	\begin{equation} \label{eq:gluing.no.minimal.meancurv.2.R}
		\text{the mean curvature vectors of } M_2^R \subset (\bm{M}_2, \bm{g}_{2,\delta,\eps}'') \text{ point strictly toward } M_2^L.
	\end{equation}
	We hold this $\eps$ fixed in denoting the resulting manifold $(\bm{M}_2, \bm{g}_{2,\delta}'')$. On this manifold, we apply Proposition \ref{prop.BarHankedefor.step1} with $\bm{U}$ inside the neighborhood of Claim \ref{clai:gluing.no.minimal.foliation}, and $\eta \to 0$ to be chosen, to deform it locally near $M_2^L$ to $(\bm{M}_2, \bm{g}_{2,\delta,\eta}''')$. Certainly, for small $\eta$ we have, by \eqref{eq:gluing.no.minimal.foliation.2} and Proposition \ref{prop.BarHankedefor.step1}'s property (4) that 
	\begin{equation} \label{eq:gluing.no.minimal.scalar.M2}
		R_{\bm{g}_{2,\delta,\eta}'''} > 0 \text{ on } \bm{M}_2.
	\end{equation}
	Moreover:
	\begin{claim} \label{clai:gluing.no.minimal.foliation.M2}
		For $\delta > 0$ sufficiently small, and then $\eta > 0$ sufficiently small depending on $\delta$, Claim \ref{clai:gluing.no.minimal.foliation} holds in $(\bm{M}_2, \bm{g}_{2,\delta,\eta}''')$ with the same hypersurfaces $M^{L,t}_2$ and \eqref{eq:gluing.no.minimal.foliation} holds for all $t \in [0,\sigma]$.
	\end{claim}
	\begin{proof}[Proof of claim]
		Since our conformal constructions trivialize in $C^\infty$ as $\delta \to 0$, we may write
		\[ M_2^{L,t} =: \graph_{M_2^L}^{\bm{g}_{2,\delta}''} f^{(\delta,t)} \]
		for $f^{(\delta,t)} : M_2^L \to \RR$ with uniform $C^2$ constants provided $\delta$ is small. 
		 
		On the other hand, note that the metrics $\bm{g}_{2,\delta}''$, $\bm{g}_{2,\delta,\eta}'''$ have no $dt$ factors in their difference in Fermi coordinates over $(x, t)$ with respect to either $\bm{g}_{2,\delta}''$ or $\bm{g}_{2,\delta,\eta}'''$ (see Remark \ref{rema:BarHanke.fermi}). Thus, by Lemmas \ref{lemm:slicing.formula.curv}, \ref{lemm:mean.curv.fermi}, the uniform $C^2$ bounds on $f^{(\delta,t)}$, and the fact that $\bm{g}_{2,\delta,\eta}''' \to \bm{g}_{2,\delta}''$ in $C^1$ by Proposition \ref{prop.BarHankedefor.step1}'s property (1), the mean curvatures of $M^{L,t}_2 := \graph_{M_2^L}^{\bm{g}_{2,\delta}''} f^{(\delta,t)}$ remain positive in $(\bm{M}_2, \bm{g}_{2,\delta,\eta}''')$ if $\eta$ is sufficiently small, since they were positive in $(\bm{M}_2, \bm{g}_{2,\delta}'')$. The claim follows.
	\end{proof}
	
	We hold this $\eta$ fixed in denoting the resulting manifold $(\bm{M}_2, \bm{g}_{2,\delta}''')$.
	
	Next, apply Theorem \ref{theo:refined.gluing.single} to $(\bm{M}_1, \bm{g}_1)$ with $\eta \to 0$ to be chosen anew to deform it locally near $M_1^R$ to $(\bm{M}_1, \bm{g}_{1,\eta}')$ so that the mean curvature of $M_1^R \subset (\bm{M}_1, \bm{g}_{1,\eta}')$ agrees with that of $M_2^L \subset (\bm{M}_2, \bm{g}_{2,\delta}''')$, which in turn agrees with that of $M_2^L \subset (\bm{M}_2, \bm{g}_{2,\delta}''')$. The reason we can invoke this theorem is \eqref{eq:gluing.no.minimal.meancurv.2.L} above and Propositon \ref{prop.BarHankedefor.step1}'s property (3). Note that, if $\eta > 0$ is sufficiently small, then a fixed neighborhood $\bm{U}_1$ of $M_1^R$ in $(\bm{M}_1, \bm{g}_{1,\eta}')$ will have the property that it is foliated by hypersurfaces whose mean curvature vectors all point strictly away from $M_1^R$. Fix any such $\eta > 0$, and write $(\bm{M}_1, \bm{g}_1')$ for the resulting manifold.
	
	Note that, by Remark \ref{rema:BarHanke.Cnormal.gluing},
	\[ (\hat{\bm{M}}, \hat{\bm{g}}_\delta) := (\bm{M}_1, \bm{g}_1''') \sqcup_\sim (\bm{M}_2, \bm{g}_{2,\delta}''') \]
	is smooth and satisfies conclusions (1), (2), (3),\footnote{For (3) we reparametrize so that $(1+\delta)^{4/(n-1)} \mapsto (1+\delta)$. This doesn't affect (1), (2), (4), (5).} (4) by construction. The lemma will follow once we prove that it also satisfies conclusion (5). We will use a smooth local foliation of mean-convex hypersurfaces near $\Sigma$ to achieve this. The portion the foliation to within $\bm{M}_1$ is covered by $\bm{U}_1$. On the other side, set
	\[ \bm{U}_2 := \bigcup_{t \in [0,\sigma)} M_2^{L,t}, \]
	and note that, by Claim \ref{clai:gluing.no.minimal.foliation.M2}, it is a neighborhood of $M_2^L \subset (\bm{M}_2, \bm{g}_{2,\delta}''')$ that is smoothly foliated by hypersurfaces whose mean curvature vectors all point strictly away from $M_2^R$, and $\bm{U}_2$ is independent of $\delta \to 0$.

	Suppose, for the sake of contradiction, that there did exist a closed minimal hypersurface $T^0_\delta \subset (\hat{\bm{M}} \setminus \partial \hat{\bm{M}}, \hat{\bm{g}}_\delta)$. 
	
	\begin{claim} \label{clai:gluing.no.minimal.compact.mcf.1}
		There exists a nonempty two-sided stable closed minimal hypersurface\footnote{$T_\delta$ may be a multiple cover of an underlying embedded hypersurface.} $T_\delta \subset (\hat{\bm{M}} \setminus \partial \hat{\bm{M}}, \hat{\bm{g}}_\delta)$ that satisfies $\vol_{\hat{\bm{g}}_\delta}(T_\delta) < \vol_{\hat{\bm{g}}_\delta}(M_2^R)$.
	\end{claim}
	\begin{proof}[Proof of claim]
		This claim follows directly by area minimization in the homology class of the mean-convex boundary component $M_2^R$ in $(\hat{\bm{M}} \setminus T^0_\delta, \hat{\bm{g}}_\delta)$, or more precisely in its metric completion where we have one or two additional minimal boundary components corresponding to the two-sided covers of $T^0_\delta$. We show yet another argument that will make the proof of the more complex Claim \ref{clai:min.surf.existence.mcf} below easier to follow. Our idea was inspired from the proof of \cite[Theorem 2.5]{White:isoperimetric}.
		
		We wish to consider the mean curvature flow starting at the mean-convex boundary component $M_2^R$. Since mean curvature flows may develop singularities in finite time, we need to work with a weaker notion: the mean-convex level set flow (see \cite{White:mean.convex.mcf.size.singular}). 
		
		For simplicity we first assume that $M_1^L = \emptyset$. Then $M_2^R$ bounds a compact mean-convex domain $\cK(0)$ in $(\hat{\bm{M}}, \hat{\bm{g}}_\delta)$. Let $t \mapsto \cK(t)$ be the mean-convex level set flow out of $\cK(0)$ in $(\hat{\bm{M}}, \hat{\bm{g}}_\delta)$. The constant flow $t \mapsto T^0_\delta$ is a weak set flow (\cite[Section 4]{White:topology}) by the minimality of $T^0_\delta$ and the maximum principle. Then, the weak set flow avoidance principle (\cite[Theorem 7.1]{White:topology}) implies $T^0_\delta \subset \cK(\infty) := \lim_{t \to \infty} \cK(t)$. By \cite[Theorem 11.1]{White:mean.convex.mcf.size.singular}, $T_\delta := \partial \cK(\infty)$ is a smooth stable minimal hypersurface (possibly the double cover of an embedded hypersurface). It is nonempty since $\cK(\infty) \supset T^0_\delta$, and $\vol_{\hat{\bm{g}}_\delta}(T_\delta) = \vol_{\hat{\bm{g}}}(\partial \cK(\infty)) < \vol_{\hat{\bm{g}}}(\partial \cK(0)) = \vol_{\hat{\bm{g}}_\delta}(M_2^R)$.
		
		Now consider the general case, $M_1^L \neq \emptyset$. Enlarge $\hat{\bm{M}}$ by gluing a $M_1^L \times (-\infty, 0]$ to $\partial \hat{\bm{M}}$ along $M_1^L$, and smoothly extend $\hat{\bm{g}}_\delta$ to a complete metric with Ricci curvature bounded below (this is necessary for mean curvature flow well-posedness), and so that $M_1^L \times \{-1\}$ has mean curvature vectors pointing toward $M_1^L \times \{0\}$. Now take $\bar \cK(0)$ to be the compact mean-convex domain between $M_2^R$ and $M_1^L \times \{-1\}$, and run the mean-convex level set flow $t \mapsto \bar \cK(t)$ out of $\bar \cK(0)$. We now proceed as above but also use that $t \mapsto M_1^L$ is also a constant weak set flow. By the avoidance principle, $M_1^L \cup T^0_\delta \subset \bar \cK(\infty) := \lim_{t \to \infty} \bar \cK(t)$ and $\partial \bar \cK(\infty)$ is a smooth stable minimal hypersurface, at least one of whose components is in the non-extended manifold.
	\end{proof}
	
	\begin{claim}
		$T_\delta \not \subset \bm{M}_1 \cup \bm{U}_2$, as long as $\delta$ is small enough.
	\end{claim}
	\begin{proof}[Proof of claim]
		The mean-convex foliation of $\bm{U}_1$, $\bm{U}_2$ with respect to $\hat{\bm{g}}_\delta$ and conclusion (2) would imply that $T_\delta$ would be a closed minimal hypersurface in $(\bm{M}_1 \setminus \bm{U}_1, \bm{g}_1)$. This is a contradiction.
	\end{proof}

	We now restrict our focus to $(\bm{M}_2 \setminus \partial \bm{M}_2, \bm{g}_{2,\delta}''')$ and send $\delta \to 0$. Note that $\bm{g}_{2,\delta}'''$ converges in $C^\infty_{\loc}$ to the original metric, $\bm{g}_2$. In view of the uniform volume bounds we have obtained, the portion of the stable minimal hypersurfaces $T_\delta$ that is contained in the interior of $\bm{M}_2$ subsequentially converges in $C^\infty_{\loc}$ to a smooth minimal hypersurface $T \subset (\bm{M}_2 \setminus \partial \bm{M}_2, \bm{g}_2)$ with $T \cap (\bm{M}_2 \setminus \bm{U}_2) \neq \emptyset$ by construction and with $\emptyset \neq \overline{T} \setminus T \subset M_2^L$ by \cite{SchoenSimon:stable}. Now:
	
	\begin{claim} \label{clai:min.surf.existence.mcf}
		There exists a nonempty two-sided stable closed minimal hypersurface $T^* \subset (\bm{M}_2 \setminus \partial \bm{M}_2, \bm{g}_2)$ that satisfies $\vol_{\bm{g}}(T^*) < \vol_{\bm{g}}(M_R)$.
	\end{claim}
	\begin{proof}
		This follows from the same argument as Claim \ref{clai:gluing.no.minimal.compact.mcf.1}, except now we need to work on an extension of $(\bm{M}_2, \bm{g}_2)$ instead of an extension of $(\hat{\bm{M}}, \hat{\bm{g}}_\delta)$ and use $M_2^L \cup T$ as an obstacle instead of $M_1^L \cup T^0_\delta$. To see that the constant flow on the closed set $M_2^L \cup T$ is a weak set flow one can argue by contradiction; a smooth mean curvature flow cannot have a finite first touching time with $M_2^L \cup T$ because it'd either have to happen with a smooth point of $M_2^L$ or, else, with a smooth point of $T$, either way contradicting the maximum principle. Thus, \cite[Theorem 7.1]{White:topology} applies just the same as before to give us avoidance from $M_2^L \cup T$, and the rest of the argument from Claim  \ref{clai:gluing.no.minimal.compact.mcf.1} now applies verbatim with $M_2^L \cup T$ and $(\bm{M}_2 \setminus \partial \bm{M}_2, \bm{g}_2)$ in place of $M_1^L \cup T^0_\delta$ and $(\hat{\bm{M}} \setminus \partial \hat{\bm{M}}, \hat{\bm{g}}_\delta)$.
	\end{proof}
	
	The existence of $T^*$ violates $(\bm{M}_2, \bm{g}_2) \in \cE_B(M_2, g_2, H=0)$. 
\end{proof}

Now we can prove Theorem \ref{theo:m.k.plus.fillin}.

\begin{proof}[Proof of Theorem \ref{theo:m.k.plus.fillin}]
	If $f : \bm{M} \to \RR$ satisfies
	\[ \Delta_{\bm{g}_0} f = \eps \text{ on } \bm{M}, \; f \equiv 1 \text{ on } \partial \bm{M}, \]
	then using \eqref{eq:scalar.curvature.conformal}, \eqref{eq:mean.curvature.conformal}, it follows that $(\bm{M}, f^2 \bm{g}_0)$ has everywhere positive scalar curvature and, when $\eps > 0$ is sufficiently small, its boundary is still mean-convex.
	
	Invoke Proposition \ref{prop:m.k.plus.pairs.psc.mean.convex} to construct a metric $\bm{h}$ on $\bm{C} := M \times [0,1]$ such that:
	\begin{itemize}
		\item $\bm{h} \restr (M \times \{0\}) = g_0$;
		\item $\bm{h} \restr (M \times \{1\}) = \psi^* g$ for some $\psi \in \operatorname{Diff}_0(M)$;
		\item $M \times \{0\} \subset (\bm{C}, \bm{h})$ is stable minimal;
		\item $M \times \{1\} \subset (\bm{C}, \bm{h})$ has mean curvature vectors pointing strictly toward $M \times \{0\}$;
		\item $R_{\bm{h}} \geq 0$ on $\bm{C}$.
	\end{itemize}
	Then, the result follows by using Lemma \ref{lemm:gluing.no.minimal.compact} to smooth the concatenation
	\[ (\bm{M}, f^2 \bm{g}_0) \sqcup_\sim (\bm{C}, (\Psi^{-1})^* \bm{h}), \]
	where $\Psi$ is the constant extension of $\psi \in \operatorname{Diff}_0(M)$ to $\operatorname{Diff}_0(\bm{C})$ and $\sim$ identifies the obvious boundary components under $\psi$.
\end{proof}

\section{Proof of Theorem \ref{theo:bartnik.3.dim}} \label{sec:applications.relativity.bartnik}

We need the following noncompact analog of Lemma \ref{lemm:gluing.no.minimal.compact}:

\begin{lemm}[Smoothing a concatenation with no new minimal surfaces, II] \label{lemm:gluing.no.minimal.af}
	Let $3 \leq n+1 \leq 7$. Consider a compact Riemannian $(n+1)$-manifold-with-boundary $(\bm{M}_1, \bm{g}_1)$ and a complete asymptotically flat Riemannian $n$-manifold-with-boundary $(\bm{M}_2, \bm{g}_2)$ so that $ \partial \bm{M}_1 =: M_1^L \sqcup M_1^R$, $\partial \bm{M}_2 =: M_2$, and:
	\begin{itemize}
		\item $M_1^R \subset (\bm{M}_1, \bm{g}_1)$ is isometric to $M_2 \subset (\bm{M}_2, \bm{g}_2)$;
		\item $M_1^L \subset (\bm{M}_1, \bm{g}_1)$, $M_2 \subset (\bm{M}_2, \bm{g}_2)$ are stable minimal hypersurfaces, though $M_1^L$ is allowed to be empty;
		\item $M_1^R \subset (\bm{M}_1, \bm{g}_1)$ has strictly inward pointing mean curvature;
		\item $R_{\bm{g}_1} \geq 0$ on $\bm{M}_1$, strictly near along $M_1^R$, and $R_{\bm{g}_2} \geq 0$ on $\bm{M}_2$;
		\item neither $(\bm{M}_i, \bm{g}_i)$ has no closed interior minimal hypersurfaces.
	\end{itemize}
	Define
	\[ (\hat{\bm{M}}, \hat{\bm{g}}) := (\bm{M}_1, \bm{g}_1) \sqcup_\sim (\bm{M}_2, \bm{g}_2), \]
	where $\sim$ identifies isometric boundaries $M_1^R$, $M_2^L$ as a single hypersurface $\Sigma \subset \hat{\bm{M}}$. This is a compact Riemannian $(n+1)$-manifold-with-boundary with a ``corner'' along $\Sigma$. Let $\bm{U}$ be a neighborhood of $\Sigma$ in $(\hat{\bm{M}}, \hat{\bm{g}})$. 
	
	Then, there exist smooth metrics $\hat{\bm{g}}_\eta$ on $\hat{\bm{M}}$, for $\eta > 0$ small, such that:
	\begin{enumerate}
		\item $\hat{\bm{g}}_\eta \to \hat{\bm{g}}$ in $C^0(\hat{\bm{M}}) \cap C^\infty_{\loc}(\hat{\bm{M}} \setminus \Sigma)$ as $\eta \to 0$;
		\item $\hat{\bm{g}}_\eta \equiv \hat{\bm{g}}$ on $(\bm{M}_1 \setminus \bm{U}) \subset \hat{\bm{M}}$;
		\item $\mathfrak{m}_{ADM}(\hat{\bm{M}}, \hat{\bm{g}}_\eta) = (1+\eta) \mathfrak{m}_{ADM}(\hat{\bm{M}}, \hat{\bm{g}})$;
		\item $R_{\hat{\bm{g}}_\eta} \geq 0$ on $\hat{\bm{M}}$;
		\item $(\hat{\bm{M}}, \hat{\bm{g}}_\eta)$ has no closed interior minimal hypersurfaces.
	\end{enumerate}
\end{lemm}
\begin{proof}
	A straightforward adaptation of the previous proof.
\end{proof}

\begin{proof}[Proof of Theorem \ref{theo:bartnik.3.dim}]
	It suffices to show that if $g, g' \in \sM^{\geq 0}_{1/2}(M)$, $\vol_g(M) = \vol_{g'}(M)$, and $(\bm{M}', \bm{g}') \in \cE_{B}(M, g', H=0)$, then
	\[ \mathfrak{m}_B(M, g, H=0) \leq \mathfrak{m}_{ADM}(\mathbf{M}', \mathbf{g}'). \]
	This readily implies the first two bullet points and \eqref{eq:bartnik.3.dim.general} by scaling. Moreover, \eqref{eq:bartnik.3.dim} follows from \eqref{eq:bartnik.3.dim.general} by evaluating the latter on a Schwarzschild manifold.
	
	So, assume the setup above. Let $\eps > 0$ be arbitrary. Invoke Proposition \ref{prop:m.k.plus.pairs.psc.mean.convex} to obtain a monotone PSC almost-cobordance $(\bm{N}, \bm{h})$ joining a minimal $(M, g)$ to a mean-convex $(M, (1+\eps) \psi^* g')$, $\psi \in \operatorname{Diff}_0(M)$. Then apply Lemma \ref{lemm:gluing.no.minimal.af} to concatenate $(\bm{N}, \bm{h})$ and $(\bm{M}', (1+\eps) \bm{g}')$. The manifolds $(\hat{\bm{M}}, \hat{\bm{g}}_\eta)$ satisfy $\hat{\bm{g}}_\eta \in \cE_{B}(M, g, H=0)$ by conclusions (2), (4), (5), and 
	\[ \mathfrak{m}_{ADM}(\hat{\bm{M}}, \hat{\bm{g}}_\eta) = (1+\eta) (1+\eps)^{n/2} \mathfrak{m}_{ADM}(\bm{M}', \bm{g}') \]
	by conclusion (3). Letting $\eta \to 0$ and then $\eps \to 0$, the result follows.
\end{proof}

\section{Proof of Theorem \ref{theo:bartnik.bray.n.dim}} \label{sec:applications.relativity.bartnik.bray}

We start this section by recalling a relative of Bartnik mass due to Bray \cite{Bray:penrose}. We call it the Bartnik--Bray mass. It is
\[ \mathfrak{m}_{BB}(M, g, H = 0) = \inf \{ \mathfrak{m}_{ADM}(\bm{M}, \bm{g}) : (\bm{M}, \bm{g}) \in \cE_{BB}(M, g, H=0) \}, \]
where $\cE_{BB}(M, g, H=0)$ denotes the set of complete, connected, asymptotically flat $(\bm{M}, \bm{g})$ with nonnegative scalar curvature, and minimal ($H=0$) boundary isometric to $(M, g)$ that strictly minimizes induced volume among all hypersurfaces that enclose it. It is easy to see, using elementary geometric measure theory, that
\[ \cE_{B}(M, g, H=0) \subset \cE_{BB}(M, g, H=0), \]
at least when $n \leq 6$. The same argument as in the introduction shows that, still,
\[ \cE_{BB}(M, g, H=0) \neq \emptyset \implies M \text{ is topologically PSC and } g \in \sM^{\geq 0}_{1/2}(M). \]
By Bray--Lee's Riemannian Penrose Inequality \cite{BrayLee:penrose}, and $2 \leq n \leq 6$,
\begin{equation} \label{eq:bartnik.vs.bartnik.bray}
	\mathfrak{m}_{B}(M, g, H = 0) \geq \mathfrak{m}_{BB}(M, g, H = 0) \geq \tfrac12 (\sigma_n^{-1} \vol_g(M))^{(n-1)/n},
\end{equation}
where $\sigma_n$ denotes the volume of the unit $n$-sphere in $\RR^{n+1}$.

Theorem \ref{theo:bartnik.bray.n.dim} will be proven below and will estimate $\mathfrak{m}_{BB}(M, g, H=0)$ from above by the lower bound in \eqref{eq:bartnik.vs.bartnik.bray}, and therefore compute $\mathfrak{m}_{BB}(M, g, H=0)$ precisely for all $2 \leq n \leq 6$. We will need the following technical lemma:

\begin{lemm}[Truncating Miao's smoothings] \label{lemm:bartnik.bray.truncated.miao}
	Let $(\bm{M}, \hat{\bm{g}})$ be a complete asymptotically flat Riemannian manifold-with-boundary with a ``corner'' along a closed hypersurface $\Sigma \subset \bm{M} \setminus \partial \bm{M}$ that satisfies Miao \cite{Miao:pmt.corners}'s mean curvature jump condition \cite[Theorem 1 (H)]{Miao:pmt.corners}. Assume, further, that
	\[ R_{\hat{\bm{g}}} \geq 0 \text{ on } \bm{M} \setminus \Sigma, \]
	\[ R_{\hat{\bm{g}}} > 0 \text{ on } \partial \bm{U}, \]
	where $\bm{U}$ is a neighborhood of $\Sigma$ with $\overline{\bm{U}} \subset \bm{M} \setminus \partial \bm{M}$ and $\partial \bm{U}$ compact.
	
	Then, there exists another neighborhood $\bm{W}$ of $\Sigma$ with $\overline{\bm{W}} \subset \bm{U}$, $\partial \bm{W}$ compact, and smooth metrics $\hat{\bm{g}}_\eta$ on $\bm{M}$ for $\eta > 0$ small such that:
	\begin{enumerate}
		\item $\hat{\bm{g}}_\eta \to \hat{\bm{g}}$ in $C^0(\bm{M}) \cap C^2_{\loc}(\bm{M} \setminus \Sigma)$ as $\eta \to 0$;
		\item $\mathfrak{m}_{ADM}(\bm{M}, \hat{\bm{g}}_\eta) \to \mathfrak{m}_{ADM}(\bm{M}, \hat{\bm{g}})$ as $\eta \to 0$;
		\item $R_{\hat{\bm{g}}_\eta} \geq 0$ on $\bm{M}$;
		\item $\hat{\bm{g}}_\eta \equiv \hat{\bm{g}}$ on $\bm{M} \setminus \bm{W}$.
	\end{enumerate}
\end{lemm}
\begin{proof}
	Denote Miao's mollifications from \cite[Proposition 3.1]{Miao:pmt.corners} by $\widetilde{\bm{g}}_\eta$, with $\eta > 0$ small. Denote by their conformal parameters from \cite[Section 4.1]{Miao:pmt.corners} by $u_\eta > 0$, with $\eta > 0$ small. 
	
	Fix a smooth cut-off $\zeta : \bm{M} \to [0, 1]$ with
	\[ \zeta \equiv 0 \text{ near } \Sigma, \; \zeta \equiv 1 \text{ outside } \bm{U}, \; \spt \nabla \zeta \subset \subset \{ R_{\hat{\bm{g}}} > 0 \}. \]
	Consider $\hat{\bm{g}}_\eta := (\zeta + (1-\zeta) u_\eta)^{4/(n-1)} \widetilde{\bm{g}}_\eta$. This family will satisfy:
	\begin{itemize}
		\item (1) by \cite[Proposition 4.1]{Miao:pmt.corners};
		\item (2) by \cite[Lemma 4.2]{Miao:pmt.corners} since $\zeta \equiv 0$ or $1$ at $\infty$;
		\item (3) by \cite[(45)]{Miao:pmt.corners} outside $\spt \zeta$, by $R_{\hat{\bm{g}}} \geq 0$ on $\{ \zeta \equiv 1 \}$, and by (1) and the $u_\eta \to 1$ convergence in $C^2_{\loc}$ away from $\Sigma$ by \cite[Proposition 4.1]{Miao:pmt.corners} on $\spt \nabla \zeta \subset \subset \{ R_{\hat{\bm{g}}} > 0 \}$;
		\item (4) by construction.
	\end{itemize}
	The result follows.
\end{proof}

\begin{proof}[Proof of Theorem \ref{theo:bartnik.bray.n.dim}]
	We show that if $M'$ is closed and orientable, $g' \in \sM^{\geq 0}_{1/2}(M')$, and $(\bm{M}', \bm{g}') \in \cE_{BB}(M', g', H=0)$, then
	\[ \mathfrak{m}_{BB}(M, g, H=0) \leq \mathfrak{m}_{ADM}(\mathbf{M}', \mathbf{g}') \]
	for all $g \in \operatorname{LinClos}[\sM^{>0}_{1/2}(M)]$ with $\vol_g(M) = \vol_{g'}(M')$. This readily implies \eqref{eq:bartnik.bray.n.dim.general} by scaling. Moreover, \eqref{eq:bartnik.bray.n.dim} follows from \eqref{eq:bartnik.bray.n.dim.general} by evaluating \eqref{eq:bartnik.bray.n.dim.general} on a Schwarzschild manifold.
	
	So, assume the setup above. Take any $(\bm{M}, \bm{g}) \in \cE_{BB}(M, g, H=0)$. 
	
	\textbf{Step 1} (truncating our extensions). Let $R$, $R'$ be large enough that the coordinate spheres $\{ r = \tfrac12 R \}$, $\{ r = \tfrac12 R' \}$ are in the asymptotically flat regimes of $(\bm{M}, \bm{g})$, $(\bm{M}', \bm{g}')$, respectively. We may additionally choose $R$, $R'$ so that the spheres $\{ r = R \}$, $\{ r = R' \}$ in the corresponding manifolds have the same volume radius $R_V$ and whose unit normalizations are close to each other and to the unit sphere, where the closeness is to be in the sense of Remark \ref{rema:m.k.plus.pairs.nearly.constant}, which we will invoke later.
	
	Then, first cut $(\bm{M}, \bm{g})$ along the large coordinate sphere $\Sigma = \{ r = R \}$ in the single asymptotically flat end and discard the unbounded component. Call $(\bm{N}_1, \bm{h}_1)$ the metric completion of what's left. It is a smooth manifold-with-boundary. Its boundary consists of:
	\begin{itemize}
		\item a minimal portion, $(M, g)$, and
		\item a mean-convex portion (mean curvature vector pointing to the interior) which is a very large, nearly round sphere of volume radius $R_V$ (note that $R_V \to \infty$ as $R \to \infty$).
	\end{itemize}
	Likewise, cut $(\bm{M}', \bm{g}')$ along $\Sigma' = \{ r = R' \}$ and discard the unbounded component, calling the metric completion of what's left $(\bm{N}_3, \bm{h}_3)$. Its boundary consists of:
	\begin{itemize}
		\item a minimal portion, $(M', g')$, and
		\item a mean-convex portion (mean curvature vector pointing to the interior) which is a very large, nearly round sphere of volume radius $R_V$ (note that $R_V \to \infty$ as $R \to \infty$).
	\end{itemize}

	\textbf{Step 2} (a bridge between $\Sigma$, $\Sigma'$). By Proposition \ref{prop:m.k.plus.pairs.psc.minimal} (see Remark \ref{rema:m.k.plus.pairs.nearly.constant}), there exists a PSC $(\bm{N}_2, \bm{h}_2)$, $\bm{N}_2 = \SS^n \times [0,1]$, such that:
	\begin{itemize}
		\item $\SS^n \times \{0\} \subset \partial (\bm{N}_2, \bm{h}_2)$ has induced metric $\bm{g} \restr \Sigma$,
		\item $\SS^n \times \{1\} \subset \partial (\bm{N}_2, \bm{h}_2)$ has induced metric $\bm{g}' \restr \Sigma'$, after a diffeomorphism $\psi \in \operatorname{Diff}_0(\mathbf{S}^n)$, and
		\item all $\SS^n \times \{t\}$, $t \in [0,1]$, are minimal.
	\end{itemize}
	
	\textbf{Step 3} (putting the pieces together). Glue together
	\[ (\hat{\bm{M}}, \hat{\bm{g}}) := (\bm{N}_1, \bm{h}_1) \sqcup_{\operatorname{Id}} (\bm{N}_2, \bm{h}_2) \sqcup_{\psi} (\bm{N}_3, \bm{h}_3) \sqcup_{\operatorname{Id}} (\bm{M}', \bm{g}') \]
	by identifying the corresponding isometric boundaries among consecutive pairs in the listing above. This results in an asymptotically flat manifold with ``corners'' along:
	\begin{itemize}
		\item the identification of the mean-convex component of $\partial \bm{N}_1$ with its isometric copy in $\partial \bm{N}_2$;
		\item the identification of the mean-convex component of $\partial \bm{N}_3$ with its isometric copy in $\partial \bm{N}_2$;
		\item the identification of $\partial \bm{M}'$ with its isometric copy in $\partial \bm{N}_3$.
	\end{itemize}
	The metric $\hat{\bm{g}}$ is not smooth across the corner hypersurfaces, only Lipschitz. Nonetheless, its boundary $\partial (\hat{\bm{M}}, \hat{\bm{g}}) = (M, g)$ satisfies:
	\begin{claim} \label{clai:bartnik.bray.2.dim.disconnected.M.minimizing}
		$\partial (\hat{\bm{M}}, \hat{\bm{g}}) \subset (\hat{\bm{M}}, \hat{\bm{g}})$ is homologically $\hat{\bm{g}}$-volume-minimizing.
	\end{claim}
	\begin{proof}[Proof of Claim]
		Consider a $\hat{\bm{g}}$-volume-minimizing $S \subset \hat{\bm{M}}$. First,
		\begin{equation} \label{eq:bartnik.bray.2.dim.disconnected.M.minimizing}
			\vol_{\hat{\bm{g}}}(S) \leq \vol_g(M),
		\end{equation} 
		since $(M, g) = \partial (\hat{\bm{M}}, \hat{\bm{g}}) $ is a viable competitor. The poorly understood regularity of $S$ across the corners will not be relevant.
		
		Consider $\cR := \{ r \leq \tfrac12 R \} \subset \bm{N}_1$. If $S \cap \cR \neq \emptyset$, we are done:
		\begin{itemize}
			\item If $S \subset \cR$, we are done by $(\bm{M}, \bm{g}) \in \cE_{BB}(M, g, H=0)$.
			\item If $S \not \subset \cR$, the monotonicity formula along the asymptotic regime forces $\vol_{\hat{\bm{g}}}(S \cap (\bm{N}_1 \setminus \cR)) \to \infty$ as $R \to \infty$, contradicting \eqref{eq:bartnik.bray.2.dim.disconnected.M.minimizing}.
		\end{itemize}
		
		Likewise, take $\cR' := \{ r \leq \tfrac12 R' \} \subset \bm{N}_3$. If $S \cap (\cR' \cup \bm{M}') \neq \emptyset$, we are done too:
		\begin{itemize}
			\item If $S \subset \cR' \cup \bm{M}'$, we are done by $(\bm{M}', \bm{g}') \in \cE_{BB}(M', g', H=0)$.
			\item If $S \not \subset \cR' \cup \bm{M}'$, the monotonicity formula along the asymptotic regime forces $\vol_{\hat{\bm{g}}}(S \cap (\bm{N}_3 \setminus \cR')) \to \infty$ as $R' \to \infty$, contradicting \eqref{eq:bartnik.bray.2.dim.disconnected.M.minimizing}.
		\end{itemize}
		
		So, $S \subset (\bm{N}_1 \setminus \cR) \cup \bm{N}_2 \cup (\bm{N}_3 \setminus \cR')$. This region admits distance-decreasing diffeomorphism into $(\SS^n \times \RR, g_{\SS^n(R_V/2)} + dt^2)$, where $g_{\SS^n(R_V/2)}$ is the metric of $\SS^n$ with radius $R_V/2$. Thus, $\vol_{\hat{\bm{g}}}(S) > \vol_g(M)$, contradicting \eqref{eq:bartnik.bray.2.dim.disconnected.M.minimizing}.
	\end{proof}

	\textbf{Step 4} (smoothing). Let $\eps > 0$. Take $(\bm{N}_0, \bm{h}_0)$ to be a monotone PSC almost cobordance from $(M, g)$ to $(M, (1+\eps) \varphi^* g)$ by Proposition \ref{prop:m.k.plus.pairs.psc.mean.convex} (see Remark \ref{rema:m.k.plus.pairs.nearly.constant}) performed one component at a time, where $\varphi \in \operatorname{Diff}_0(M)$. Then, consider the manifold with ``corners:''
	\[ (\breve{\bm{M}}, \breve{\bm{g}}_\eps) := (\bm{N}_0, \bm{g}_0) \sqcup_\varphi (\hat{\bm{M}}, (1+\eps) \hat{\bm{g}}). \]
	Its boundary $(M, g)$ is strictly homologically minimizing by combining Claim \ref{clai:bartnik.bray.2.dim.disconnected.M.minimizing} above and the mean-convex foliation of the interior of $(\bm{N}_0, \bm{h}_0)$. Consider the smoothings $(\breve{\bm{M}}, \breve{\bm{g}}_{\eps,\eta})$ obtained from Lemma \ref{lemm:bartnik.bray.truncated.miao}, where $\bm{U}$ is taken to be a small neighborhood of $\hat{\bm{M}}$ inside $\breve{\bm{M}}$. Then, the volume-minimization property is maintained by the smoothings of Lemma \ref{lemm:gluing.no.minimal.af} due to conclusions (1) and (4), and thus
	\[ (\breve{\bm{M}}, \breve{\bm{g}}_{\eps,\eta}) \in \cE_{BB}(M, g, H=0) \]
	for sufficiently small $\eta > 0$. By conclusion (2),
	\[ \lim_{\eta \to 0} \mathfrak{m}(\breve{\bm{M}}, \breve{\bm{g}}_{\eps,\eta}) = (1+\eps)^{n/2} \mathfrak{m}_{ADM}(\bm{M}', \bm{g}'). \]
	Letting $\eps \to 0$, the result follows.
\end{proof}

Here is what we know about Bartnik--Bray extendibility in $n$ dimensions:

\begin{lemm} \label{lemm:bartnik.bray.n.dim.extensions}
	We have:
	\begin{itemize}
		\item every $M$ that is the tubular neighborhood of a codimension $\geq 3$ submanifold of $\RR^{n+1}$ is Bartnik--Bray extendible;
		\item every disjoint finite union of $n$-dimensional Bartnik--Bray extendible manifolds is Bartnik--Bray extendible.
	\end{itemize}
\end{lemm}
\begin{proof}
	First, the extendibility of tubular neighborhoods of smooth codimension $\geq 3$ submanifolds of $\RR^{n+1}$ follows from \cite{DahlLarsson:unit.normal.bundle.horizons}. In fact the extensions constructed are Bartnik extensions, too.
	
	It remains to verify the extendibility of disjoint finite unions, which requires a bit of attention and is broken down into three steps. To that end, let $M_i$, $i = 1, \ldots, p$, denote our extendible manifolds, and let $g_i$ be an extendible metric on $M_i$.
	
	\textbf{Step 1} (concatenate suitable Schwarzschild manifolds). First, fix $i = 1, \ldots, p$. Consider a mass $m_i$ exterior Schwarzschild region whose horizon is a round sphere with volume $> \vol_{g_i}(M_i)$; e.g.,
	\begin{equation} \label{eq:bartnik.bray.n.dim.disconnected.extension.mass.individual}
		m_i := \tfrac12 (\sigma_n^{-1} (\vol_{g_i}(M_i)+\delta))^{(n-1)/n},
	\end{equation}
	though the value of this $m_i$ will not be relevant. Call this exterior manifold $(\bm{M}_i, \bm{g}_i)$. We now invoke the Carlotto--Schoen \cite{CarlottoSchoen:gluing} gluing construction to concatenate portions of $(\bm{M}_i, \bm{g}_i)$, $i = 1, \ldots, p$, into a single $(\bm{M}, \bm{g})$, which is complete, asymptotically flat, with nonnegative scalar curvature, and minimal boundary isometric to $(M, g)$.\footnote{The gluing construction can be localized away from the boundary, or be applied to the double of our manifolds and then cut along the Schwarzschild horizons.} Even though the gluing theorem offers extremely refined conclusions, all we need is that:
	\begin{enumerate}
		\item $(\bm{M}, \bm{g})$ is $\delta$-close in $C^2$ to being flat on the set $\cV$ outside the $\delta^{-1}$-neighborhood of its boundary;
		\item the $4\delta^{-1}$-neighborhood of the boundary of $(\bm{M}, \bm{g})$ consists of disjoint neighborhoods $\bm{U}_i$, $i = 1, \ldots, p$, that are isometric to neighborhoods of the boundaries of $(\bm{M}_i, \bm{g}_i)$.
	\end{enumerate}
	Using this, we have:
	
	\begin{claim} \label{clai:bartnik.bray.n.dim.disconnected.union.minimizing.1}
		$\partial (\bm{M}, \bm{g}) \subset (\bm{M}, \bm{g})$ is strictly homologically minimizing for small $\delta$.
	\end{claim}
	\begin{proof}
		Consider a $\bm{g}$-volume-minimizing $S \subset \bm{M}$. It satisfies 
		\begin{equation} \label{eq:bartnik.bray.n.dim.disconnected.M.minimizing}
			\vol_{\bm{g}}(S) \leq \vol_{g}(M) + O(\delta),
		\end{equation} 
		since $(M, g) = \partial (\bm{M}, \bm{g})$ is a viable competitor and \eqref{eq:bartnik.bray.n.dim.disconnected.extension.mass.individual} holds. Note that $S \subset \bigcup_{i=1}^p \bm{U}_i$. Indeed, if not, then the monotonicity formula and (1) above would yield $\vol_{\bm{g}}(S) \to \infty$ as $\delta \to 0$, violating  \eqref{eq:bartnik.bray.n.dim.disconnected.M.minimizing} for small $\delta$. Now that $S \subset \bigcup_{i=1}^p \bm{U}_i$, the result follows from the isometry in (2) and the volume-minimizing nature of the boundaries of $(\bm{M}_i, \bm{g}_i)$.
	\end{proof} 	
	The gluing theorem also gives $\mathfrak{m}_{ADM}(\bm{M}, \bm{g}) \leq \sum_{i=1}^p m_i + \delta$, but we will disregard this bound here.\footnote{It is tempting to conjecture $\mathfrak{m}_B(M, g, H=0) \leq \tfrac12 \sum_{i=1}^p (\sigma_n^{-1} \vol_{g_i}(M_i))^{(n-1)/n}$ based on this bound and \eqref{eq:bartnik.bray.n.dim.disconnected.extension.mass.individual}. We do not pursue this here.}
	
	\textbf{Step 2} (re-introducing $(M, g)$ as the boundary). Note that $(\bm{M}, \bm{g})$ is precisely Schwarzschild near its boundary, so it's manifestly not a contender for $\cE_{BB}(M, g, H=0)$. We fix this issue in this step. 
	
	By Theorem \ref{theo:bartnik.bray.n.dim}, each of $(M_i, g_i)$ has a valid extension that is exactly Schwarzschild outside a compact set. Using Corollary \ref{coro:psc.round.foliations.bent.schwarzschild}, we can ensure that this extension has positive scalar curvature along a hypersurface that bounds $(M_i, g_i)$. After rescaling each of $(\bm{M}_i, \bm{g}_i)$ (thus giving up mass control, by a lot) and using Corollary \ref{coro:psc.round.foliations.bent.schwarzschild} again, we can ensure that a large mean-convex coordinate sphere of our extension of $(M_i, g_i)$ will be isometric to the boundary of $(\bm{M}_i, \bm{g}_i)$. Call $(\bm{N}_i, \bm{h}_i)$ the region of the extension of $(M_i, g_i)$ up until the aforementioned large mean-convex coordinate sphere. Consider the gluing
	\[ (\hat{\bm{M}}, \hat{\bm{g}}) := (\bm{M}, \bm{g}) \sqcup_{\psi_1} (\bm{N}_1, \bm{h}_1) \sqcup _{\psi_2} \cdots \sqcup_{\psi_p} (\bm{N}_p, \bm{h}_p). \]
	The Riemannian metric has corners along the gluing hypersurfaces, as usual, but the mean-curvature jumps are of Miao-type. Note that the $(\bm{M}, \bm{g})$ is the suitably rescaled copy of the Carlotto--Schoen gluing, since each of $(\bm{M}_i, \bm{g}_i)$ was scaled.
	
	\begin{claim} \label{clai:bartnik.bray.2.dim.disconnected.union.minimizing.2}
		$\partial (\hat{\bm{M}}, \hat{\bm{g}}) \subset (\hat{\bm{M}}, \hat{\bm{g}})$ is strictly homologically minimizing
	\end{claim}
	\begin{proof}[Proof of claim]
		Any $S \subset \hat{\bm{M}}$ homologous to $\partial \hat{\bm{M}}$ can be improved to a $S' \subset \cup_{i=1}^p \bm{N}_i$ using Claim \ref{clai:bartnik.bray.n.dim.disconnected.union.minimizing.1}, and $S'$ can be improved to $\partial \hat{\bm{M}}$ using the strictly minimizing property in each $(\bm{N}_i, \bm{h}_i)$.
	\end{proof}
	
	\textbf{Step 3} (smoothing). Turn $(\hat{\bm{M}}, \hat{\bm{g}})$ into an element of $\cE_{BB}(M, g, H=0)$. We have done this many times so far, so we omit the details.
\end{proof}

\section{Acknowledgments}

This project began during ``Geometry of Scalar Curvature,'' a conference and summer school held in July 2019 in Italy and funded by the Deutsche Forschungsgemeinschaft (SPP 2026 ``Geometry at Infinity''). We are grateful to the co-organizers. We also acknowledge Bernd Ammann, Renato Bettiol, Alessandro Carlotto, Otis Chodosh, Mattias Dahl, Demetre Kazaras, and Or Hershkovits for helpful discussions, and the anonymous referees for valuable comments. C.L. and C.M. were both supported by the NSF (DMS-2005287/2202343 and DMS-1905165/2050120/2147521). The authors have no financial or proprietary interests in any material discussed in this article.

\appendix

\section{Some curvature formulas} \label{app:curvature.formulas}

For the reader's convenience, we collect some well-known curvature formulas from Riemannian geometry that play a crucial role for us.

First, suppose that $(M^n, g)$ is a Riemannian manifold with $n = \dim M \geq 3$. If $u > 0$ is smooth on $M$, then the conformal metric $\overline{g} = u^{4/(n-2)} g$ has (see \cite[(1.2)]{Escobar})
\begin{equation} \label{eq:scalar.curvature.conformal}
	R_{\overline{g}} = u^{-\frac{n+2}{n-2}} \left( - \tfrac{4(n-1)}{n-2} \Delta_g u + R_g u \right).
\end{equation}
Additionally, if $\Sigma \subset M$ is a two-sided hypersurface with unit normal $\nu$, then the mean curvature scalar of $H_\Sigma$ with $\nu$ as an outward pointing normal (so, $H_\Sigma = -\Div_\Sigma \nu$) transforms as (see \cite[(1.4)]{Escobar}, where mean curvatures are normalized with an $(n-1)^{-1}$ factor):
\begin{equation} \label{eq:mean.curvature.conformal}
	\overline{H}_\Sigma = u^{-\frac{n+2}{n-2}} \left( \tfrac{2(n-1)}{n-2} \langle \nabla_g u, \nu \rangle_g + H_\Sigma u \right).
\end{equation}

Next, the following slicing formulas for the curvature along normal foliations form the foundation for the deep relationship between $\sM^{>0}_{1/2}(M)$, scalar curvature, and stable minimal hypersurfaces. They are well-known to experts and follow from elementary computations.

\begin{lemm}[Slicing formulas for curvature] \label{lemm:slicing.formula.curv}
	Suppose that $\bm{M}$ is a smooth manifold, $I \subset \RR$ is an interval, $(g_t)_{t \in I}$ is a smooth path of metrics on $M$, and $(u_t)$ is a smooth path of positive $C^\infty(M)$ functions. On $\bm{N} := M \times I$, consider the metric:
	\[ \bm{h}(x,t) := g_t(x) + u_t(x)^2 dt^2. \]
	Then, for every $x \in M$, $t \in I$,
	\begin{align*}
		\sff_t(x) & = (2u_t(x))^{-1} \tfrac{d}{dt} g_t(x), \\
		H_t(x) & = (2u_t(x))^{-1} \operatorname{Tr}_{g_t(x)} \tfrac{d}{dt} g_t(x), \\
		R_{\bm{h}}(x,t) & = 2 u_t(x)^{-1} ( -\Delta_{g_t} u_t(x) + \tfrac12 R_{g_t}(x) u_t(x)) \\
			& \qquad - 2 u_t(x)^{-1} \tfrac{\partial}{\partial t} H_t(x) - H_t(x)^2 - |\sff_t(x)|^2,
	\end{align*}
	where $\sff_t(\cdot)$ and $H_t(\cdot)$ are the second fundamental form and mean curvature scalar of $M \times \{t\} \subset (\bm{N}, \bm{h})$ with $\partial_t$ taken to be the outward pointing direction.
\end{lemm}
\begin{proof}
	The formulas for $\sff_t$, $H_t$ follow from the first variation formula, i.e., the interpretation of the second fundamental form and the mean curvature scalar as the rate of change of the metric and the volume form in the outward unit normal direction, $u_t^{-1} \partial_t$. The formula for the scalar curvature follows from the second variation formula
	\[ \tfrac{\partial}{\partial t} H_t = -\Delta_{g_t} u_t -(|\sff_t|^2 + \Ric_{\bm{h}}(u_t^{-1} \partial_t, u_t^{-1} \partial_t)|_{M \times \{t\}}) u_t \]
	and the twice-traced Gauss equation for $M \times \{t\} \subset (\bm{N}, \bm{h})$
	\[ R_{g_t} = R_{\bm{h}}|_{M \times \{t\}} - 2 \Ric_{\bm{h}}(u_t^{-1} \partial_t, u_t^{-1} \partial_t)|_{M \times \{t\}} + H_t^2 - |\sff_t|^2; \]
	multiply the former by $2u_t^{-1}$ and subtract the latter.
\end{proof}

Finally, we need the following formula for the mean curvature of graphs in case $u_t \equiv 1$ for all $t$, i.e., in case we are working in Fermi coordinates:

\begin{lemm}[Mean curvature in Fermi coordinates] \label{lemm:mean.curv.fermi}
	Assume the setting of Lemma \ref{lemm:slicing.formula.curv}, with $u_t \equiv 1$ for all $t$. Suppose that $\Sigma := \graph_M f \subset \bm{N}$ for a smooth $f : M \to I$. Then, the mean curvature scalar $H_\Sigma$ of $\Sigma$, with $\partial_t$ as the outward pointing unit normal, equals:
	\begin{align*}
		H_\Sigma(x) 
			& = - \Div_{g_{f(x)}} \left( \frac{\nabla_{g_{f(x)}} f(x)}{( 1 + g^{ij}_{f(x)} f_i(x) f_j(x) )^{1/2}} \right) \\
			& \qquad - \frac{\sff^{ij}_{f(x)}(x) f_i(x) f_j(x)}{( 1 + g^{ij}_{f(x)} f_i(x) f_j(x) )^{1/2}} + ( 1 + g^{ij}_{f(x)} f_i(x) f_j(x) )^{1/2} H_{f(x)}(x).
	\end{align*}
\end{lemm}
\begin{proof}
	See \cite[(A.13)]{ChodoshMantoulidis:generic}.
\end{proof}

\section{Round normal foliations} \label{app:psc.round.foliations}

Seeing as to how we are concerned with the flexibility of PSC cobordances with mean-convex foliations throughout this paper, we dedicate this appendix to such cobordances of the simplest kind, namely:
\begin{equation} \label{eq:psc.round.foliations}
(\SS^n \times I, f(t)^2 g_{\SS^n} + dt^2),
\end{equation}
where $I \subset \RR$ is an interval, $f : I \to (0, \infty)$ is a smooth function, and $g_{\SS^n}$ indicates a round metric on a unit $\SS^n$.  We call such metrics round normal foliations with unit speed. 

By Lemma \ref{lemm:slicing.formula.curv}, the second fundamental form and mean curvature of $\SS^n \times \{t\}$ (with the unit normal pointing to the right as the outward pointing normal) are:
\begin{equation} \label{eq:psc.round.foliations.sff.meancurv}
	\sff_t = \frac{f'(t)}{f(t)} g_{\SS^n} \text{ and } H_t = n \frac{f'(t)}{f(t)},
\end{equation}
and the ambient scalar curvature is
\begin{equation} \label{eq:psc.round.foliations.scal}
	R|_{\SS^n \times \{t\}} = \frac{n(n-1)(1 - (f'(t))^2) - 2n f''(t) f(t)}{f(t)^2}.
\end{equation}
Thus, our cobordance is:
\begin{equation} \label{eq:psc.round.foliations.psc}
	\text{PSC} \iff \frac{1-(f')^2}{f} > \frac{2}{n-1} f'',
\end{equation}
\begin{equation} \label{eq:psc.round.foliations.mean.convex}
	\text{mean-convex} \iff f' > 0.
\end{equation}
For the sake of the reader's intuition, we point out that exact solutions of the equality case of \eqref{eq:psc.round.foliations.psc} correspond to rotationally symmetric scalar-flat manifolds (i.e., Schwarzschild manifolds).

\begin{lemm}[Refined gluing lemma] \label{lemm:psc.round.foliations.gluing}
	Suppose that we have smooth functions $f_i : [a_i, b_i] \to (0, \infty)$, $i = 1$, $2$, satisfying \eqref{eq:psc.round.foliations.psc}, \eqref{eq:psc.round.foliations.mean.convex}. Assume, also, that $f_1(b_1) < f_2(a_2)$. 
	
	Then, the following two statements are equivalent:
	\begin{enumerate}
		\item There exists a $T > b_1 - a_2$ and a smooth function $F : [a_1, T+b_2] \to (0, \infty)$ satisfying \eqref{eq:psc.round.foliations.psc}, \eqref{eq:psc.round.foliations.mean.convex}, and
			\[ F = f_1 \text{ on } [a_1, b_1] \text{ and } F(T + \cdot) = f_2 \text{ on } [a_2, b_2]. \]
		\item The mean curvature of $\SS^n \times \{a_2\}$ in $(\SS^n \times [a_2, b_2], f_2^2 g_{\SS^n} + dt^2)$ is \emph{strictly less} than the mean curvature of the sphere of radius $f_2(a_2)$ in the unique Schwarzschild manifold containing $\SS^n \times \{b_1\}$ with its induced metric from $(\SS^n \times [a_2, b_2], f_1^2 g_{\SS^n} + dt^2)$.
	\end{enumerate}
\end{lemm}
\begin{proof}
	Since we're only interested in mean-convex foliations, i.e., solutions satisfying \eqref{eq:psc.round.foliations.mean.convex}, we are allowed to change variables and treat $f(t)$ as an independent variable $x \in (0, \infty)$, and $f'(t)$ as a dependent variable $y(x) \in (0, \infty)$. This change of variables eliminates the need to determine $T$. Moreover, the chain rule gives:
	\begin{equation} \label{eq:psc.round.foliations.psc.xy}
		\text{PSC} \iff \frac{1-y(x)^2}{xy(x)} > \frac{2}{n-1} y'(x).
	\end{equation}
	This is a first order nonlinear differential inequality. Elementary ODE analysis shows that, for all $(x_0, y_0) \in (0, \infty) \times (0, \infty)$, solutions $\mathbf{y}(x)$ of the corresponding \emph{equation} with $\mathbf{y}(x_0) = y_0$ exist for all $x \in [x_0, \infty)$ forward in time, and as long as $x \not \to 0$, $\infty$ backward in time, and are unique. We refer to such $\mathbf{y}(x)$ as
	\[ \textit{Schwarzschild solutions through } (x_0, y_0). \]
	Let us rewrite (1), (2) above in this formulation. Let $y_i : [f_i(a_i), f_i(b_i)] \to (0, \infty)$, $i=1$, $2$, denote the reparametrizations of our $f_i$. Then, (1) turns into:
	\begin{enumerate}
		\item [(1')] There exists a smooth extension $Y : [f_1(a_1), f_2(b_2)] \to (0, \infty)$ of $y_1$ and $y_2$ that satisfies \eqref{eq:psc.round.foliations.psc.xy}.
	\end{enumerate}
	For (2), we note:
	\begin{itemize}
		\item The mean curvature of $\SS^n \times \{a_2\}$ in $(\SS^n \times [a_2, b_2], f_2^2 g_{\SS^n} + dt^2)$ is $n y_2(f_2(a_2)) / f_2(a_2)$ by \eqref{eq:psc.round.foliations.sff.meancurv}.
		\item If $\mathbf{y}_1 : [f_1(b_1), \infty) \to (0, \infty)$ denotes the unique Schwarzschild solution through $(f_1(b_1), f_1'(b_1))$, then the mean curvature of the sphere of radius $f_2(a_2)$ in this Schwarzschild solution is $n \mathbf{y}_1(f_2(a_2)) / f_2(a_2)$ by \eqref{eq:psc.round.foliations.sff.meancurv}.
	\end{itemize}
	Therefore, (2) turns into:
	\begin{enumerate}
		\item [(2')] $y_2(f_2(a_2)) < \mathbf{y}_1(f_2(a_2))$.
	\end{enumerate}
	The equivalence of (1') and (2') is now trivial. Indeed, the existence of $Y$ in (1') comes down to getting a smooth extension that always flows monotonically across the Schwarzschild solution flowlines. 
\end{proof}

\begin{coro}[Bending and gluing Schwarzschild] \label{coro:psc.round.foliations.bent.schwarzschild}
	Suppose that $m_1 < m_2$ and $(\bm{M}_{i}, \bm{g}_{i})$, $i = 1$, $2$, are mass $m_i$ exterior Schwarzschild manifolds. Suppose that $\Sigma_i$, $i = 1$, $2$ are two round spheres (along the standard foliation) of radius $\rho_i > 0$ and mean curvature $h_i > 0$ with respect to the unit normal pointing to infinity. If $\rho_1 < \rho_2$ and $h_2$ is strictly less than the mean curvature of the mean-convex sphere of radius $\rho_2$ that can be found inside $(\bm{M}_{1}, \bm{g}_{1})$, then there exists a smooth Riemannian manifold $(\widetilde{\mathbf{M}}, \widetilde{\mathbf{g}})$ with the following properties:
	\begin{enumerate}
		\item $(\widetilde{\mathbf{M}}, \widetilde{\mathbf{g}})$ contains an isometric copy of the mean-convex region bounded by $\Sigma_1$ and the mean-concave region bounded by $\Sigma_2$;
		\item the interior region spanning the two regions above is of the form \eqref{eq:psc.round.foliations} and also satisfies \eqref{eq:psc.round.foliations.psc}, \eqref{eq:psc.round.foliations.mean.convex}.
	\end{enumerate}
	Additionally, recalling the inherent dependence $h_i = h_i(m_i, \rho_i)$, we can also have:
	\begin{enumerate}
		\item [(3)] If $\rho_1$, $\rho_2$ are held fixed and $m_1$, $m_2 \to m_\star$, then $(\widetilde{\mathbf{M}}, \widetilde{\mathbf{g}})$ converges in $C^\infty$ to the mass $m_\star$ exterior Schwarzschild manifold.
	\end{enumerate}
\end{coro}
\begin{proof}
	One can apply a small initial PSC bend to the region just past $\Sigma_1 \subset (\bm{M}_1, \bm{g}_1)$ and the region just prior to $\Sigma_2 \subset (\bm{M}_2, \bm{g}_2)$ (cf. \cite[Lemma 2.3]{MantoulidisSchoen:bartnik} for the 2-dimensional argument, which generalizes trivially now using  \eqref{eq:psc.round.foliations.psc} instead of the 2-dimensional formula). Since our inequalities on $\rho_1$, $\rho_2$,$h_1$, $h_2$ were strict, they are preserved for the boundaries of the slightly bent regions, which are then glued together using Lemma \ref{lemm:psc.round.foliations.gluing}. Conclusions (1) and (2) follow immediately from this construction. To see conclusion (3), it is perhaps fastest to revisit the proof of Lemma \ref{lemm:psc.round.foliations.gluing}. When $m_1, m_2 \to m_\star$, our reduction shows
	\[ |y_2(\rho_1) - \mathbf{y}_\star(\rho_1)| + |\mathbf{y}_1(\rho_1) - \mathbf{y}_\star(\rho_1)| \to 0 \]
	where $y_2$, $\mathbf{y}_1$ is as above, $\rho_1 =: f_2(a_2)$, and $\mathbf{y}_\star$ denotes the orbit line of the mass $m_\star$ Schwarzschild solution through $(\rho_2, h_\star)$, where $h_\star > 0$ is the mean curvature of our radius-$\rho_2$ sphere inside $(\bm{M}_\star, \bm{g}_\star)$. Thus, we have convergence of the initial and terminal data of the Schwarzschild solution, and the result follows by the smooth dependence of ODE solutions on their data.
\end{proof}

In the next lemma we deal with metrics of the form \eqref{eq:psc.round.foliations} that have strictly positive lower bounds on the scalar curvature. 

\begin{lemm}[cf. {\cite[Proposition 3.3]{Marques:deforming.psc}}, {\cite[Lemma 6.2]{Marques:deforming.psc}}] \label{lemm:psc.round.foliations.isotopy}
	Suppose that $\bm{g}_0 := f_0(t)^2 g_{\SS^n} + dt^2 \in \operatorname{Met}(\SS^n \times [a,b])$ has $R_{\bm{g}_0} \geq \theta n(n-1)$, where the metric $g$ is a unit round metric on $\SS^n$, $n \geq 2$, and $\theta \in [0,1]$. There exists a continuous extension of $\bm{g}_0$ to $[0,1] \ni \mu \mapsto \bm{g}_\mu \in \operatorname{Met}(\SS^n \times [a,b])$, where each $\bm{g}_\mu$ is of the form $f_\mu(t)^2 g_{\SS^n} + h_\mu(t)^2 dt^2$, and:
	\begin{enumerate}
		\item $\bm{g}_1 = g_{\SS^n} + dt^2$;
		\item $R_{\bm{g}_\mu} > \theta n(n-1)$ on $\SS^n \times [a,b]$ for all $\mu \in (0,1)$;
		\item if $f_0(t_0) = 1$, then $f_\mu(t_0) = h_\mu(t_0) = 1$ for all $\mu \in [0,1]$.
	\end{enumerate}
\end{lemm}
\begin{proof}
	A key observation is that the conformal metric
	\begin{equation} \label{eq:psc.round.foliation.isotopy.obs}
		f_0(t)^{-2} \bm{g}_0 := g_{\SS^n} + f_0(t)^{-2} dt^2 \text{ is a standard product metric}
	\end{equation}
	modulo a reparametrization of the $t$ coordinate.
	
	\textbf{Step 1} ($\mu \in [0,\tfrac12]$). Define
	\[ \bm{g}_\mu := \big[ (1-2\mu) + 2\mu f_0(t)^{(1-n)/2} \big]^{4/(n-1)} \bm{g}_0, \; \mu \in [0,\tfrac12]. \]
	Conclusion (3) is trivial. By \eqref{eq:scalar.curvature.conformal}, $R_{\bm{g}_\mu} \geq \theta n(n-1)$ is equivalent to
	\begin{multline*}
		-\tfrac{4n}{n-1} \Delta_{\bm{g}_0} \big[ 2\mu f_0^{(1-n)/2} \big] + R_{\bm{g}_0} \big[ (1-2\mu) + 2\mu f_0^{(1-n)/2} \big] \\
		- \theta n(n-1) \big[ (1-2\mu) + 2\mu f_0^{(1-n)/2} \big]^{(n+3)/(n-1)} \geq 0.
	\end{multline*}
	This inequality is an identity at $\mu = 0$ by assumption, and is also true at $\mu = \tfrac12$ by \eqref{eq:psc.round.foliation.isotopy.obs} and $\theta \in [0,1]$. For $\mu \in (0, \tfrac12)$, we observe that the expression above is a concave function of $\mu \in [0,\tfrac12]$ since $f_0 \leq 1$. This implies conclusion (2) by calculus.
	
	\textbf{Step 2} ($\mu \in (\tfrac12, 1]$). Define
	\[ \bm{g}_\mu := g_{\SS^n} + \big[ (2-2\mu) f_0(t)^{-2} + 2 \mu - 1 \big] dt^2, \; \mu \in (\tfrac12, 1]. \]
	By construction, $\bm{g}_\mu$ is continuous at $\mu = \tfrac12$, and as in \eqref{eq:psc.round.foliation.isotopy.obs}, $\bm{g}_\mu$ is a  standard product metric modulo a reparametrizatoin of the $t$ coordinate. This readily implies conclusion (2), and conclusion (3) is trivial.
\end{proof}

\section{More facts and formulas regarding $\sM^{>0}_{k}$, $\sM^{\geq 0}_{k}$} \label{app:m.k.plus.more}

The following lemma is well-known to experts in the Yamabe problem, and relates $\cM^{>0}_k(M)$ ($\cM^{\geq 0}_k(M)$) to the question of existence of a metric of positive (nonnegative) scalar curvature on $M$.

\begin{lemm}[cf. {\cite[Lemma 1.2]{Schoen:variational.theory}}] \label{lemm:m.k.plus.equivalences}
	Properties (1) and (2) below are equivalent for a closed connected Riemannian $n$-manifold $(M, g)$, if $n \geq 3$.
	\begin{enumerate}
		\item $g \in \sM^{>0}_{(n-2)/4(n-1)}(M)$ ($\sM^{\geq 0}_{(n-2)/4(n-1)}(M)$);
		\item there exists a metric $\overline{g}$ conformal to $g$ with $R_{\overline{g}} > 0$ ($R_{\overline{g}} \geq 0$).
	\end{enumerate}
	If $n = 2$, then property (1') below implies (2'):
	\begin{enumerate}
		\item[(1')] $g \in \sM^{>0}_{k}(M)$ ($\sM^{\geq 0}_{k}(M)$) for any $k \in (0, \infty)$;
		\item[(2')] $M$ is diffeomorphic to $\SS^2$ or $\RR \PP^2$ (or, additionally, $\TT^2$, $\KK^2$).
	\end{enumerate}
\end{lemm}
\begin{proof}
	We only treat $\sM^{>0}_{k}(M)$; the case of $\sM^{\geq 0}_{k}(M)$ is essentially similar. First consider $n \geq 3$. For brevity, we write $k=\tfrac{n-2}{4(n-1)}$.  
	
	(1) $\implies$ (2). Let $u$ be a positive first eigenfunction of $-\Delta_g + k R_g$. Then $\overline{g} = u^{4/(n-2)} g$ has $R_{\overline{g}} > 0$ by \eqref{eq:scalar.curvature.conformal}. 

	(2) $\implies$ (1). Suppose that $\overline{g} \in [g]$ has $R_{\overline{g}} > 0$. Then, \[ \int_M (|\nabla_{\overline{g}} \tilde f|^2 + k R_{\overline{g}} \tilde f^2) \, d\mu_{\overline{g}} > 0 \text{ for all } \tilde f \in C^\infty(M) \setminus \{0\}  \]
	\[ \implies \int_M R_{\tilde g} \, d\mu_{\tilde g} > 0 \text{ for all } \tilde g \in [\overline{g}] \]
	by \eqref{eq:scalar.curvature.conformal}. On the other hand, $[g] = [\overline{g}]$, so by \eqref{eq:scalar.curvature.conformal} again,
	\[ \implies \int_M (|\nabla_g f|^2 + k R_g f^2) \, d\mu_g > 0 \text{ for all } f \in C^\infty(M) \setminus \{0\}. \]
	Thus, $g \in \sM^{>0}_{k}(M)$ by the variational characterization of the first eigenvalue.
	
	(1') $\implies$ (2'). By the variational characterization of the first eigenvalue,
	\[ \int_M (|\nabla_g f|^2 + k R_g f^2) \, d\mu_g > 0 \text{ for all } f \in C^\infty(M) \setminus \{0\}. \]
	Plugging in $f \equiv 1$, it follows from Gauss--Bonnet that $M$ is an $\SS^2$ or $\RR \PP^2$. 
\end{proof}

\begin{coro} \label{coro:m.k.plus.nonempty}
	Let $M$ be a closed $n$-manifold. The following are equivalent:
	\begin{enumerate}
		\item $\sM^{>0}_{k}(M)$ ($\sM^{\geq 0}_{k}(M)$) is nonempty for all $k \in [\tfrac{n-2}{4(n-1)}, \infty) \cap (0, \infty)$;
		\item $\sM^{>0}_{k}(M)$ ($\sM^{\geq 0}_{k}(M)$) is nonempty for some $k \in [\tfrac{n-2}{4(n-1)}, \infty) \cap (0, \infty)$;
		\item $\sM^{>0}_{\infty}(M)$ ($\sM^{\geq 0}_{\infty}(M)$) is nonempty, i.e., $M$ is topologically PSC (NNSC).
	\end{enumerate}
\end{coro}
\begin{proof}
	(1) $\implies$ (2) is trivial.

	(2) $\implies$ (3) follows by applying \eqref{eq:m.k.plus.inclusion} and Lemma \ref{lemm:m.k.plus.equivalences} to each component of $M$.
		
	(3) $\implies$ (1) also follows from \eqref{eq:m.k.plus.inclusion}.
\end{proof}

The following problem seems interesting:
\begin{prob}
	Suppose that $M$ is a closed $n$-manifold, and $k \in (0, \infty]$. Are $\sM^{>0}_{k}(M)$, $\sM^{>0}_{k} \big/ \operatorname{Diff}(M)$ connected when nonempty?
\end{prob}

Here is what is known about this problem:

\begin{itemize}
	\item When $n=2$, the answer is known to be yes for all $k \in (0, \infty]$. This follows from a crucial property of $\lambda_1(-\Delta + kR)$ that was observed in  \cite[Proposition 1]{MantoulidisSchoen:bartnik}. Since this reference only proves the result for $k=\tfrac12$, we show the general argument in Proposition \ref{prop:m.k.2.dim.connected} below.
	\item When $n=3$, Theorems \ref{theo:m.k.plus.path}, \ref{theo:m.k.plus.path.yamabe} shows that the answer is yes for all $k \in \{ \tfrac18 \} \cup [\tfrac14, \infty]$; when $k=\infty$, this is due respectively to  Bamler--Kleiner \cite{BamlerKleiner:contractibility} (who proved a stronger result) and Cod\'a Marques \cite{Marques:deforming.psc}. We do not know what happens when $k \in (0, \tfrac14) \setminus \{ \tfrac18 \}$.
\end{itemize}

\begin{prop} \label{prop:m.k.2.dim.connected}
	The following spaces are all smoothly path-connected:
	\begin{itemize}
		\item $\sM^{>0}_{k}(\SS^2)$;
		\item $\sM^{\geq 0}_{k}(\SS^2)$, and all path interiors can be taken in $\sM^{>0}_{k}(\SS^2)$;
	\end{itemize}
\end{prop}
\begin{proof}
	The proof goes in three steps.
	
	\textbf{Step 1} (reduce to round metrics). Let $g \in \sM^{>0}_{k}(\SS^2)$ ($\sM^{\geq 0}_{k}(\SS^2)$). By uniformization, there exists a round metric $g_0$ on $\SS^2$ and a smooth $u : \SS^2 \to \RR$ such that $g = e^{2u} g_0$. Define $[0,1] \ni t \mapsto g_t := e^{2(tu + (1-t))} g_0 \in \operatorname{Met}(\SS^2)$, where $g_1 = g$. We claim that $g_t \in \sM^{>0}_{k}(\SS^2)$ for all $t \in (0,1)$. Indeed, for any smooth test function $f : \SS^2 \to \RR$, the 2-dimensional conformal invariance of Dirichlet energy, and 
	\[ R_{g_t} \, d\mu_{g_t} = (R_{g_0} - 2 \Delta_{g_0} (tu + (1-t))) \, d\mu_{g_0} \]
	imply that
	\[ \int_{\SS^2} (|\nabla_{g_t} f|^2 + k R_{g_t} f^2) \, d\mu_{g_t} \]
	is \emph{linear} in $t$. Since it is positive at $t=0$ ($g_0$ is round, so $R_{g_0} > 0$) and positive (nonnegative) at $t=1$ (by assumption), it follows that the quantity is positive for all $t \in (0,1)$, and all $f$. This completes the proof of the claim.
	
	\textbf{Step 2} (connectedness of moduli space). We have shown that all metrics in $\sM^{>0}_{k}(\SS^2)$ ($\sM^{\geq 0}_{k}(\SS^2)$) can be continuously deformed to a round metric within $\sM^{>0}_{k}(\SS^2)$. On the other hand, by the classification of space forms, every round metric is of the form $\varphi^* A^2 g_\star$, where $\varphi \in \operatorname{Diff}(\SS^2)$, $A > 0$, and $g_\star$ is the model round metric on the unit sphere in $\RR^3$. The second step follows by deforming $A$ to $A=1$ within the space of round metrics (a subset of $\sM^{>0}_{k}(\SS^2$)) by scaling. 
	
	\textbf{Step 3} (connectedness of full space). It remains to deform $\varphi^* g_\star$ to $g_\star$. By the path-connectedness of the space of orientation-preserving diffeomorphisms of $\SS^2$, we can deform $\varphi$ to $\pm \operatorname{Id}$ within the space of round metrics (a subset of $\sM^{>0}_{k}(\SS^2$)). The result follows from the fact that $(-\operatorname{Id})^* g_\star = g_\star$.
\end{proof}

\begin{rema} \label{rema:m.k.plus.contractibility}
	The argument can be improved to showing that $\sM^{>0}_{k}(\SS^2)$, $\sM^{\geq 0}_{k}(\SS^2)$ are contractible. Compare to the proof of Theorem \ref{theo:m.k.plus.path.yamabe} which treats $n=3$ and $k=\tfrac18$ as a consequence of Bamler--Kleiner's breakthrough for $k=\infty$  \cite{BamlerKleiner:contractibility}. Their technique may extend to $k \in [\tfrac14, \infty)$. We do not pursue this.
\end{rema}

Finally, a relationship to stable minimal hypersurfaces. If $(\bm{N}, \bm{h})$ is a Riemannian $(n+1)$-manifold with $R_{\bm{h}} > 0$ and $M \subset \bm{N}$ is a two-sided closed stable minimal hypersurface in $\bm{N}$, then Schoen--Yau \cite{SchoenYau:incompressible} first observed that the second variation formula and the Gauss equation imply that $\lambda_1(-\Delta_g + \tfrac12 R_g) > 0$ for the induced metric $g$ on $M$, i.e., $g \in \sM^{>0}_{1/2}(M)$.  Likewise, $g \in  \sM^{\geq 0}_{1/2}(M)$ if $R_{\bm{h}} \geq 0$. While they did not explicitly formulate this observation in this manner at the time, they did explicitly consider this exact operator in \cite{SchoenYau:condensation}, as did Gromov--Lawson in \cite{GromovLawson:summary}. We summarize in the following:

\begin{lemm} \label{lemm:m.12.plus.equivalences}
	The following are equivalent for closed Riemannian manifolds $(M, g)$:
	\begin{enumerate}
		\item $g \in \sM^{>0}_{1/2}(M)$ ($\sM^{\geq 0}_{1/2}(M)$);
		\item $g$ is the metric induced on $M$ when $M$ occurs as a two-sided stable minimal hypersurface in a manifold $(\bm{N}, \bm{h})$ with $R_{\bm{h}} > 0$ ($R_{\bm{h}} \geq 0$).
	\end{enumerate}
\end{lemm}
\begin{proof}
	(2) $\implies$ (1) follows from the second variation formula. 
	
	(1) $\implies$ (2) follows by letting $f > 0$ be a first eigenfunction of the positive definite operator $-\Delta_g + \tfrac12 R_g$ on $M$, and setting $(\bm{N}, \bm{h}) := (M \times \SS^1, g + f^2 dt^2)$. It is not hard to see that the slices $\{ t = \text{const}\}$ are two-sided stable minimal hypersurfaces.
\end{proof}

The following more subtle result is essentially contained in the proof of \cite[Theorem 3.1]{Galloway:outermost.psc} but is not spelled out:

\begin{lemm}[{cf \cite[Theorem 3.1]{Galloway:outermost.psc}}] \label{lemm:m.12.plus.galloway}
	Suppose that $M$ is a closed two-sided minimal hypersurface in a manifold $(\bm{N}, \bm{h})$, with $R_{\bm{h}} \geq 0$, and that $M$ is strictly area-minimizing on one side (either one) among small graphical perturbations. Then, the metric $g$ induced on $M$ satisfies $g \in \overline{\sM^{>0}_{1/2}(M)}$; the closure is in the usual $C^\infty$ topology.
\end{lemm}
\begin{proof}
	By Lemma \ref{lemm:m.12.plus.equivalences}, it follows that $g \in \sM^{\geq 0}_{1/2}(M)$. Without loss of generality, 
	\[ g \in \sM^{\geq 0}_{1/2}(M) \setminus \sM^{> 0}_{1/2}(M), \]
	otherwise there is nothing to prove. By \cite[Lemma 2.3]{Galloway:outermost.psc} (whose operator coincides with $-\Delta_g + \tfrac12 R_g$ provided we take $K \equiv 0$), and restricting to the side of $M$ on which it's strictly minimizing, there exists a local foliation $M \times [0, t_0)$ of $\bm{N}$ so that $\bm{h} = g_t + u_t^2 \, dt^2$ in these coordinates, and the mean curvature scalars $H_t$ of $M \times \{t\}$ with respect to the outward unit normal are all constant. Rearranging the slicing formula for scalar curvature, and using that $R_{\bm{h}}, H_t^2, |\sff_t|^2, u \geq 0$, yields:
	\begin{align*}
		\tfrac{d}{dt} H_t 
			& = - \Delta_{g_t} u_t + \tfrac12 R_{g_t} u_t - \tfrac12 ( R_{\bm{h}}+ H_t^2 + |\sff_t|^2) u_t \\
			& \leq - \Delta_{g_t} u_t + \tfrac12 R_{g_t} u_t.
	\end{align*}
	Thus, $\lambda_1(-\Delta_{g_t} + \tfrac12 R_{g_t}) \geq 0$ by the maximum principle whenever $\tfrac{d}{dt} H_t \geq 0$, in which case the inequality is even strict unless $\tfrac{d}{dt} H_t = 0$. 
	
	It remains to prove that there exist $t_i \to 0$ with $[\tfrac{d}{dt} H_t]_{t=t_i} > 0$. Indeed, if this were false, then $\tfrac{d}{dt} H_t \leq 0$ for all $t \in [0, t_1)$ for some $t_1 \in (0, t_0)$. But $H_0 = 0$ by assumption, so $H_t \leq 0$ for all $t \in [0, t_1)$, contradicting that $M \times \{0\}$ strictly minimizes area. This completes the proof.
\end{proof}

\section{Codimension $\geq 3$ surgery in $\sM^{>0}_{k}(M)$ by B\"ar--Dahl} \label{sec:m.k.plus.connected.sum}

Our proof relies on showing the Gromov--Lawson \cite[Theorem A]{GromovLawson:classification} codimension $\geq 3$ surgery (cf. Schoen--Yau's slightly different \cite[Theorem 4]{SchoenYau:structure.psc}) maintains $\sM^{>0}_{k}$, $k \in (0, \infty)$, as it is known to maintain $\sM^{>0}_{\infty}$. The following result follows from the proof of a rather general spectral approximation result by B\"ar--Dahl \cite{BarDahl:surgery}. (We also refer the reader to \cite{AmmannDahlHumbert} for other interesting applications of such results.)

\begin{prop}[cf. {\cite[Theorem 3.1]{BarDahl:surgery}}]  \label{prop:m.k.plus.surgery}
	Let $(M_1, g_1)$ and $(M_2, g_2)$ be closed $n$-dimensional Riemannian manifolds, $n-3 \geq d \geq 0$, and $\Sigma^d$ be either a single point ($d=0$), or a $d$-dimensional sphere ($1 \leq d \leq n-3$). Fix $k$, $\delta \in (0, \infty)$. For each $i=1$, $2$, assume that $\Phi_i : \Sigma \to U_i$ is an embedding of $\Sigma$ into a smooth open $U_i \subset M_i$, and $N_i : N(\Phi_i(\Sigma)) : \Sigma \times \RR^{n-d}$ is a trivializing section of its normal bundle. Then, there exists a choice of connected-sum parameters for
	\[ g := g_1 \#_\Sigma g_2 \in \operatorname{Met}(M), \; M := M_1 \#_\Sigma M_2, \]
	where the connected sum is performed with the given trivializations $N_1$, $N_2$, as well as a diffeomorphism 
	\[ F : (M_1 \setminus U_1) \sqcup (M_2 \setminus U_2) \to M \setminus U \]
	of compact manifolds-with-boundary, for a smooth open $U \subset M$ that is diffeomorphic to $U_1 \#_\Sigma U_2$, such that the following hold:
	\begin{enumerate}
		\item $F^* g \restr (M \setminus U) = (g_1 \restr (M_1 \setminus U_1)) \sqcup (g_2 \restr (M_2 \setminus U_2))$;
		\item $\min_U R_g \geq \min_{i=1,2} \min_{U_i} R_{g_i} - \delta$;
		\item $\lambda_1(-\Delta_g + kR_g) \geq \min_{i=1,2} \lambda_1(-\Delta_{g_i} + k R_{g_i}) - \delta$.
	\end{enumerate}
\end{prop}

The parameters of the construction depend continuously on the underlying data, so it applies to continuous 1-parameter families of metrics. We summarize the result, which we only state for $d=0$:

\begin{coro} [cf. {\cite[Proposition 6.1]{Marques:deforming.psc}}] \label{coro:m.k.plus.surgery.families}
	Let $M_1$ and $M_2$ be closed $n$-manifolds, $n \geq 3$, $k$, $\delta \in (0, \infty)$. For each $i=1$, $2$, suppose that we have the following continuous paths:
	\begin{itemize}
		\item $[0, 1] \ni \mu \mapsto g_{i,\mu} \in \operatorname{Met}(M_i)$;
		\item $[0,1] \ni \mu \mapsto p_{i,\mu} \in U_i$ for a smooth open $U_i \subset M_i$;
		\item $[0,1] \ni \mu \mapsto \{ e^{(j)}_{i,\mu} \}_{j=1,\ldots,n}$ for a $g_i$-orthonormal basis of $T_{p_{i,\mu}} M_i$.
	\end{itemize}
	There exists a uniform (in $\mu$) choice of connected-sum parameters for
	\[ g_\mu := g_{1,\mu} \# g_{2,\mu} \in \operatorname{Met}(M), \; M := M_1 \# M_2, \]
	where the connected sum is performed with the given trivializations of $T_{p_{i,\mu}} M_i$, so that $\mu \mapsto g_\mu$ is continuous, and a diffeomorphism 
	\[ F : (M_1 \setminus U_1) \sqcup (M_2 \setminus U_2) \to M \setminus U \]
	of compact manifolds-with-boundary, for a smooth open $U \subset M$ that is diffeomorphic to $U_1 \# U_2$, such that the following hold for all $\mu \in [0,1]$:
	\begin{enumerate}
		\item $F^* g_\mu \restr (M \setminus U) = (g_{1,\mu} \restr (M_1 \setminus U_1)) \sqcup (g_{2,\mu} \restr (M_2 \setminus U_2))$;
		\item $\min_{U} R_{g_\mu} \geq \min_{i=1,2} \min_{U_i} R_{g_{i,\mu}} - \delta$;
		\item $\lambda_1(-\Delta_{g_\mu} + k R_{g_\mu}) \geq \min_{i=1,2} \lambda_1(-\Delta_{g_{i,\mu}} + k R_{g_{i,\mu}}) - \delta$.
	\end{enumerate}
\end{coro}

\section{Some results from Kleiner--Lott's notes} \label{app:kleiner.lott}

In this appendix we collect two results from \cite[Section 93.12]{KleinerLott:notes}, where operators of the form $-4 \Delta_g + V$ were studied.

\begin{lemm} [cf. {\cite[Lemma 93.16]{KleinerLott:notes}}] \label{lemm:kleiner.lott.lambda}
	Let $(M, g)$ be a closed $n$-dimensional Riemannian manifold. Suppose that $X \subset M$ be a compact submanifold-with-boundary of the same dimension as $M$. 	If $\psi$ denotes a first eigenfunction of $-\Delta_g + kR_g$ on $M$, then:
	\begin{align*}
		\lambda_1(-\Delta_g + kR_g) 
			& \leq \lambda_1((-\Delta_g + kR_g) \restr X) \\
			& \leq \lambda_1(-\Delta_g + kR_g) + \frac{\int_M |\nabla_g \eta|^2 \psi^2 \, d\mu_g}{\int_M \eta^2 \psi^2 \, d\mu_g}
	\end{align*}
	for all $\eta \in C^\infty_c(X \setminus \partial X)$. Here, $\lambda_1((-\Delta_g + kR_g) \restr X)$ denotes the first Dirichlet eigenvalue of $-\Delta_g + kR_g$ restricted to $X$.
\end{lemm}
\begin{proof}
	This follows by applying \cite[Lemma 93.16]{KleinerLott:notes} with $4k R_g$ in place of $V$ and rescaling by $2$.
\end{proof}

We will seek to control the right hand side of conclusion (2) above using the following coercivity (``Agmon-type'') estimate:

\begin{lemm} [cf. {\cite[Lemma 93.21]{KleinerLott:notes}}] \label{lemm:kleiner.lott.agmon}
	With the notation of Lemma \ref{lemm:kleiner.lott.lambda}, given a nonnegative smooth $\phi : M \to \RR$, suppose that a smooth $f : M \to \RR$ satisfies:
	\begin{equation} \label{eq:kleiner.lott.agmon.assumption}
		|\nabla_g f|^2 \leq kR_g - \lambda - c \text{ on } \spt \phi,
	\end{equation}
	for some $c > 0$, and $\lambda := \lambda(-\Delta_g + kR_g)$. Then,
	\begin{align} \label{eq:kleiner.lott.agmon.conclusion}
		& \Vert e^f \phi \psi \Vert_{L^2(M)} \\
		& \leq 64c^{-1} \left( \Vert e^f \Delta_g \phi \Vert_{L^\infty(M)} + \Vert e^f \nabla_g \phi \Vert_{L^\infty(M)} (\lambda - \min k R_g)^{1/2} \right) \Vert \psi \Vert_{L^2(M)}, \nonumber 
	\end{align}
	where the constant $C$ depends only on $c$. The same holds true if $f$ is only Lipschitz.
\end{lemm}
\begin{proof}
	This follows by applying \cite[Lemma 93.21]{KleinerLott:notes} with $4k R_g$ in place of $V$ and $4c$ in place of $c$ and rescaling by $2$.
\end{proof}

\section{Mean convex foliation refinement of B\"ar--Hanke gluing} \label{app:b.h.gluing}

There are several important scalar curvature smoothing results in the literature; see, e.g., Gromov--Lawson \cite[Theorem 5.7]{GromovLawson:fundamental.group}, Miao \cite[Theorem 1]{Miao:pmt.corners}, Brendle--Marques--Neves \cite[Theorem 5]{BrendleMarquesNeves:min.oo}, Gromov \cite[p. 705]{Gromov:metric.inequalities}). For our Bartnik mass computation, we will need a more refined smoothing theorem that respects the sign of both the scalar curvature and the mean curvature along foliations near the boundary.

We derive this from a recent beautiful construction of B\"ar--Hanke \cite{BarHanke2021boundary,BarHanke2021flexibility}, whose work applies to other type of open partial relations as well, and even allows for the simultaneous treatment of families of metrics. For readers' convenience and for the completeness of this paper, we include the results relevant to our specific purpose, as well as the modifications needed to ensure the mean curvature inequality in the interior, which was not handled in the original construction but nevertheless follows from it.

We introduce some necessary notation and definitions.

Given a Riemannian manifold-with-boundary $(\bm{M},\bm{g})$ with compact nonempty boundary $M := \partial \bm{M}$, we write the metric in a tubular neighborhood $\bm{U}_M$ of $M$ in Fermi coordinates $(x, t)$ as follows:
\begin{equation} \label{eq.metric.BarHanke}
	\bm{g}(x,t)= dt^2 + g_t(x), \;  x \in M, \; t\in [0,\eps_M).
\end{equation}
Here $g_t(x)$ is a smooth family of Riemannian metrics on $M$. Note that the vector field $-\partial_t$ is the outward pointing unit normal to $M_t$, the level surface of the distance function $t$. Throughout this appendix, we use $\dot g_t$, $\ddot{g}_t$, etc. to denote $t$-derivatives of $g_t$. The second fundamental form $\sff_t$ and mean curvature scalar $H_t$ of the Fermi image of $M \times \{t\}$ in $(\bm{M}, \bm{g})$, computed with $-\partial_t$ as the outward pointing normal, equal
\begin{equation}\label{eq.H.slicing.BarHanke}
	\sff_t=-\tfrac12 \dot g_t, \; H_t =-\tr_{g_t} \dot g_t.
\end{equation}
This follows from Lemma \ref{lemm:slicing.formula.curv} (see also \cite[(14)]{BarHanke2021boundary}).

\begin{defi}[$C$-normal metrics, {\cite[Definition 21]{BarHanke2021boundary}}] 
	Let $C\in \RR$. A Riemannian metric $\bm{g}$ is said to be $C$-normal if
	\[ g_t(x) = g_0(x) + t g_1(x) - Ct^2 g_0(x). \]
	in the notation of \eqref{eq.metric.BarHanke}. In the notation of \eqref{eq.H.slicing.BarHanke}, $g_1 = -2 \sff_0$.
\end{defi}

\begin{rema} \label{rema:BarHanke.Cnormal.gluing}
	Two $C$-normal metrics glue together smoothly if and only if their $g_0$, $C$ coincide and their $\sff_0$ are additive inverses.
\end{rema}

\begin{theo}[cf. {\cite[Theorem 27]{BarHanke2021boundary}}] \label{theo:refined.gluing.single}
	Let $(\bm{M}, \bm{g})$ be a Riemannian manifold-with-compact-boundary $M := \partial \bm{M}$, and that $k$ is a symmetric 2-tensor on $M$ with:
	\begin{equation} \label{eq:refined.gluing.single.k}
		\tr_{g_0} k \leq H_0.
	\end{equation}
	Fix $\eta > 0$. There exists $C_0 = C_0(\bm{g}, \eta)$ such that, for every $C \geq C_0$ and neighborhood $\bm{U}$ of $M$ there is a tubular neighborhood $\tilde{\bm{U}} \subset \subset \bm{U} \cap \bm{U}_M$ of $M$ and a metric $\tilde{\bm{g}}$ on $\bm{M}$ satisfying:
	\begin{enumerate}
		\item $\tilde{\bm{g}} \equiv \bm{g}$ on $\bm{M} \setminus \tilde{\bm{U}}$ and $\Vert \tilde{\bm{g}} - \bm{g} \Vert_{C^0(\tilde{\bm{U}}, \bm{g})} \leq \eta$;
		\item $\tilde{\bm{g}} \restr M \equiv \bm{g} \restr M$;
		\item $\tilde{\sff}_0 \equiv - 2 k$ on $M$;
		\item $R_{\tilde{\bm{g}}} \geq R_{\bm{g}} - \eta$ on $\tilde{\bm{U}}$;
		\item $\tilde{H}_t \geq \tr_{g_0} k - \eta$ in Fermi coordinates on $\tilde{\bm{U}}$;
		\item $\tilde{\bm{g}}$ is $C$-normal and $\tilde{\bm{g}} - \bm{g}$ has no $dt$ factors in Fermi coordinates on $\tilde{\bm{U}}$;
	\end{enumerate}
	Above, $\tilde{\sff}_t$, $\tilde{H}_t$ are the second fundamental form and mean curvature scalar of the Fermi image of $M \times \{t\}$ in $(\bm{M}, \tilde{\bm{g}})$ with respect to $-\partial_t$ as the outward pointing normal.
\end{theo}

\begin{rema} \label{rema:BarHanke.fermi}
	When two metrics near $M$ differ by a tensor that has no $dt$ factors, they have identical Fermi coordinates $(x, t)$ relative the hypersurface $M$. Therefore, we will never need to specify whether we are computing Fermi coordinates, tubular neighborhoods, or distance-$t$ level surfaces with respect to $\bm{g}$ or $\tilde{\bm{g}}$.
\end{rema}

To prove Theorem \ref{theo:refined.gluing.single} we proceed in two steps:
\begin{itemize}
	\item Step 1: we initially deform $\bm{g}$ locally in a neighborhood of $M$ to be $C$-normal without changing the induced second fundamental form (or metric) on $M$.
	\item Step 2: we then deform this new $C$-normal metric to one whose second fundamental form on $M$ equals $k$.
\end{itemize}

\begin{prop}[Making a metric $C$-normal] \label{prop.BarHankedefor.step1}
	Assume the same setup as above. There exists $C_0 = C_0(\bm{g}, \eta)$ such that, for every $C \geq C_0$ and neighborhood $\bm{U}$ of $M$ there is a tubular neighborhood $\hat{\bm{U}} \subset \subset \bm{U} \cap \bm{U}_M$ of $M$ and a metric $\hat{\bm{g}}$ on $\bm{M}$ so that
	\begin{enumerate}
		\item $\hat{\bm{g}} \equiv \bm{g}$ on $\bm{M} \setminus \hat{\bm{U}}$ and $\Vert \hat{\bm{g}} - \bm{g} \Vert_{C^1(\hat{\bm{U}}, \bm{g})} \leq \eta$;
		\item $\hat{\bm{g}} \restr M \equiv \bm{g} \restr M$;
		\item $\hat{\sff}_0 \equiv \sff_0$ on $M$;
		\item $R_{\hat{\bm{g}}} - R_{\bm{g}} \geq - \eta$ on $\hat{\bm{U}}$;
		\item $\hat{\bm{g}}$ is $C$-normal and $\hat{\bm{g}} - \bm{g}$ has no $dt$ factors in Fermi coordinates on $\hat{\bm{U}}$.
	\end{enumerate}
	Above, $\hat{\sff}_0$ is the second fundamental form of the Fermi image of $M \times \{t\}$ in $(\bm{M}, \hat{\bm{g}})$ with respect to $-\partial_t$ as the outward pointing unit normal.
\end{prop}

The deformation in this step is achieved by interpolation using a special cutoff function.

\begin{lemm}\label{lemma.logcutoff}
	For any $\delta \in (0,\tfrac14)$, $\eps\in (0,1)$, there exists a $C^\infty$ function $\tau_{\delta,\eps}:\RR\to \RR$ such that:
	\begin{enumerate}
		\item $\tau_{\delta,\eps}=1$ when $t\le \delta\eps$.
		\item $\tau_{\delta,\eps}=0$ when $t\ge \eps$.
		\item $0\le\tau_{\delta,\eps}\le 1$ for $t\in \RR$.
		\item for every positive integer $l$, there is a constant $C_l>0$ such that for all $t>0$:
		\[\left| \tau_{\delta,\eps}^{(l)}(t) \right|\le C_l \cdot t^{-l} \cdot |\log \delta|^{-1}.\]
	\end{enumerate}
\end{lemm}

For the proof, see \cite[Appendix A]{BarHanke2021flexibility}. 

\begin{proof}[Proof of Proposition \ref{prop.BarHankedefor.step1}]
	Write $\bm{g}(t,x)=dt^2+g_t(x)$ as in \eqref{eq.metric.BarHanke}. Consider the Taylor expansion of the tensors $g_t(x)$ in terms of $t$:
	\[g_t(x) = g_0(x) +  t \dot g_0(x) + \tfrac12 t^2 \ddot g_0(x) + R_t(x).\]
	By Taylor's theorem,
	\begin{equation} \label{eq.BarHankedefor.step1.Rt}
		\Vert R_t \Vert_{C^2(M, g_0)} + t \Vert \dot R_t \Vert_{C^1(M, g_0)} + t^2 \Vert \ddot R_t \Vert_{C^0(M, g_0)} = o(t^2).
	\end{equation}
	For $C>0$ to be chosen later and $s\in [0,1]$ consider the auxiliary metric
	\[ {}^{(s)}\bm{g}(x,t) := \bm{g}(x,t)  - s\left(\tfrac12 t^2(\ddot g_0(x) + 2Cg_0(x)) + R_t(x)\right). \]
	This metric is \textit{not} the $\hat{\bm{g}}$ we will ultimately take, but will serve as a convenient comparison metric when we eventually define $\hat{\bm{g}}$. Note that, by construction, ${}^{(s)}\bm{g}$ and $\bm{g}$ has the same first order terms in the Taylor expansion in $t$. Thus, by Lemma \ref{lemm:slicing.formula.curv},  we have:
	\[R_{{}^{(s)}\bm{g}}(x,0) = R_{\bm{g}} (x,0) + s (\tr_{g_0(x)} \ddot g_0(x) + 2(n-1)C) \ge R_{\bm{g}}(x,0) ,\]
	provided
	\[ C \geq C_0 := \tfrac{1}{2(n-1)} \max_M (- \tr_{g_0} \ddot g_0). \]
	Fix such a $C$. By compactness, there exists $\eps_0 \in (0, \eps_M)$ such that
	\begin{equation} \label{eq.BarHankedefor.step1.Rs}
		R_{{}^{(s)}\bm{g}}(x,t) > R_{\bm{g}}(x,0) - \tfrac12 \eta \text{ for all } s \in [0,1], \; t \in [0,\eps_0]
	\end{equation}
	For $\eps\in [0,\eps_0]$ small enough for the Fermi image $\bm{U}_\eps$ of $M \times [0,\eps)$ to be contained in $\bm{U}$, and $\delta>0$ to be chosen later, consider the metric 
	\[ \hat{\bm{g}}(x,t) := \bm{g}(x,t) - \tau_{\delta,\eps}(t)(\tfrac12 t^2 (\ddot g_0(x)+ 2 C g_0(x)) + R_t(x)),\]
	where $\tau_{\delta,\eps}(t)$ is the function constructed in Lemma \ref{lemma.logcutoff}. We claim $\hat{\bm{g}}$ has all the desired properties as long as we choose $\delta>0$ small enough. 
	
	Properties (1.a), (2), (3), (6) are true by construction. It remains to check properties (1.b) and (4). Lemma \ref{lemma.logcutoff} and \eqref{eq.BarHankedefor.step1.Rt} imply:
	\begin{align*}
		& \Vert \hat{\bm{g}} - \bm{g} \Vert_{C^1(\bm{U}, \bm{g})} \\
		& \qquad = \Vert \tau_{\delta,\eps}(t)(\tfrac12 t^2 (\ddot g_0(x)+ 2 C g_0(x)) + R_t(x)) \Vert_{C^1(\bm{U}_\eps, \bm{g})} \\
		& \qquad \leq C_1 |\log \delta|^{-1} \Vert \tfrac12 t (\ddot g_0(x)+ 2 C g_0(x)) + t^{-1} R_t(x) \Vert_{C^0(\bm{U}_\eps, \bm{g})} \\
		& \qquad \qquad + \Vert \tfrac12 t^2 (\ddot g_0(x)+ 2 C g_0(x)) + R_t(x) \Vert_{C^1(\bm{U}_\eps, \bm{g})} \leq \eta,
	\end{align*}
	as long as $\eps > 0$ is sufficiently small depending on $\bm{g}$, $C$, $C_1$, $\eta$. Thus (1.b) holds.
	
	We finally check property (4). This is where we rely on the auxiliary family of metrics, ${}^{(s)}\bm{g}$. Freeze a point in $\bm{U}_\eps$, with Fermi coordinates $(x, t)$, and set $s :=\tau_{\delta,\eps}(t)$. We estimate the $C^2$ norm of $\hat{\bm{g}} - {}^{(s)} \bm{g}$ at $(x, t)$ via Lemma \ref{lemma.logcutoff} and \eqref{eq.BarHankedefor.step1.Rt}:
	\begin{align*}
		& \Vert \hat{\bm{g}} - {}^{(s)} \bm{g} \Vert_{C^2((x,t),\bm{g})} \\
		& \qquad \leq \Vert (s - \tau_{\delta,\eps}(t))(\tfrac12 t^2 (\ddot g_0(x)+ 2 C g_0(x)) + R_t(x)) \Vert_{C^2((x,t),\bm{g})} \\
		& \qquad \leq C_2 |\log \delta|^{-1} \Vert \tfrac12 (\ddot g_0(x)+ 2 C g_0(x)) + t^{-2} R_t(x) \Vert_{C^0((x,t),\bm{g})} \\
		& \qquad \qquad + C_1 |\log \delta|^{-1} \Vert \tfrac12 t (\ddot g_0(x)+ 2 C g_0(x)) + t^{-1} R_t(x) \Vert_{C^1((x,t),\bm{g})}.
	\end{align*}
	Note that the right hand side $\to 0$ as $\delta \to 0$, depending on $\bm{g}$, $C$, $C_1$, $C_2$. In particular, we may choose $\delta > 0$ sufficiently small so that
	\[ |R_{\hat{\bm{g}}}(x,t) - R_{{}^{(s)}\bm{g}}(x,t)| \leq \tfrac12 \eta. \]
	Property (4) then follows from \eqref{eq.BarHankedefor.step1.Rs}.
\end{proof}

Let $(\bm{M}, \hat{\bm{g}})$ be the metric obtained in Proposition \ref{prop.BarHankedefor.step1} applied with $\eps$ and $\eta$ small enough for property (5) to imply, for the mean curvature scalar $\hat{H}_t$ of the Fermi image of $M \times \{t\}$ in $(\bm{M}, \hat{\bm{g}})$ with respect to $-\partial_t$ as the outward pointing unit normal:
\begin{equation} \label{eq.BarHankedefor.step2.H}
	\hat{H}_t(x) \geq H_0(x) - \eta.
\end{equation}
We may shrink $\hat{\bm{U}}$ to be the Fermi image of $M \times [0,\eps_1)$ in $\bm{M}$, where $\eps_1$ is an arbitrary constant in $(0,\delta \eps)$ that will be fixed later. Note that:
\[ \hat{\bm{g}} (x,t)= dt^2 + g_0(x) - 2t \sff_0(x) - Ct^2 g_0(x),\]
where $\sff_0$ was the second fundamental form of $M = \partial \bm{M} \subset (\bm{M}, \bm{g})$ with respect to $-\partial_t$ as an outward pointing unit normal. 

\begin{prop}[Interpolation in $C$-normal metrics] \label{prop.BarHankedeform.step2}
	Assume the setup above. There exists $C_0 = C_0(g_0, \sff_0, k)$ such that, for every $C \geq C_0$ and   neighborhood $\bm{U}' \subset \hat{\bm{U}}$ of $M$, there is a tubular neighborhood $\tilde{\bm{U}} \subset \subset \bm{U}'$ of $M$ and a metric $\tilde{\bm{g}}$ on $\bm{M}$ satisfying:
	\begin{enumerate}
		\item $\tilde{\bm{g}} \equiv \hat{\bm{g}}$ on $\bm{M} \setminus \tilde{\bm{U}}$ and $\Vert \tilde{\bm{g}} - \hat{\bm{g}} \Vert_{C^0(\tilde{\bm{U}}, \bm{g})} \leq \eta$;
		\item $\tilde{\bm{g}} \restr M \equiv \hat{\bm{g}} \restr M$;
		\item $\tilde{\sff}_0 \equiv k$ on $M$;
		\item $R_{\tilde{\bm{g}}} \geq R_{\hat{\bm{g}}} -\eta$ on $\tilde{\bm{U}}$;
		\item $\tilde{H}_t \geq \tr_{g_0} k - \eta$ on $\tilde{\bm{U}}$;
		\item $\tilde{\bm{g}}$ is $C$-normal and $\tilde{\bm{g}} - \hat{\bm{g}}$ has no $dt$ factors on $\tilde{\bm{U}}$;
	\end{enumerate}
	Above, $\tilde{\sff}_t$, $\tilde{H}_t$ are the second fundamental form and mean curvature scalar of the Fermi image of $M \times \{t\}$ in $(\bm{M}, \tilde{\bm{g}})$ with respect to $-\partial_t$ as the outward pointing normal.
\end{prop}

Note that these propositions imply our theorems:

\begin{proof}[Proof of Theorem \ref{theo:refined.gluing.single}]
	Follows from Propositions \ref{prop.BarHankedefor.step1}, \ref{prop.BarHankedeform.step2} as explained.
\end{proof}

Proposition \ref{prop.BarHankedeform.step2} essentially follows from the same proof of \cite[Proposition 26]{BarHanke2021boundary}. We first take the same cutoff function as in \cite[Lemma 25]{BarHanke2021boundary}. 

\begin{lemm} \label{lemm.chi.cutoff}
	There exists a constant $c_0>0$ such that for each $\eps_1\in (0,\tfrac12)$, there exists a smooth function $\chi: [0,\infty)\to \RR$ such that:
	\begin{enumerate}
		\item $\chi(t)=t$ for $t$ near $0$, $\chi(t)=0$ for $t\ge \sqrt{\eps_1}$, and $0\le \chi(t)\le \tfrac12 \eps_1$ for all $t$.
		\item $|\chi'(t)|\le c_0$ and $\chi'(t)\le 1$ for all $t$.
		\item $-\tfrac{2}{\eps_1}\le \chi''(t)\le 0$ for all $t\in [0,\eps_1]$ and $|\chi''(t)|\le c_0$ for all $t\in [\sqrt{\eps_1},\eps_1]$.
	\end{enumerate}
\end{lemm}
\begin{proof}
	Everything except the property $\chi'\le 1$ was already proven in \cite[Lemma 25]{BarHanke2021boundary}. The fact that $\chi' \leq 1$ can be arranged follows from the construction of $\chi$ as $\varphi_\delta + \psi_\delta$ for smooth functions satisfying $\varphi_\delta' \leq 1$ and $\psi_\delta$ being able to be taken non-decreasing, so $\psi_\delta' \leq 0$.
\end{proof}

\begin{proof}[Proof of Proposition \ref{prop.BarHankedeform.step2}]
	Consider the function $\chi(t)$ constructed in Lemma \ref{lemm.chi.cutoff}. We define the metric $\tilde{\bm{g}}$ so that
	\[\tilde{\bm{g}}(x,t) := dt^2 + g_0(x) - 2t\sff_0(x) + 2\chi(t) (\sff_0(x)-k(x)) - Ct^2 g_0(x) \]
	when $t\le \sqrt{\eps_1}$, and
	\[ \tilde{\bm{g}}(x,t) := \hat{\bm{g}}(x,t) \]
	when $t\ge \sqrt{\eps_1}$. The proof in \cite[Proposition 26]{BarHanke2021boundary} verifies properties (1), (2), (3), (6) provided $C>C_0(g_0, \sff_0, k)$ and $\eps_1 < \min\{\tfrac12, \delta\eps, C^{-2}\}$. Property (4) is verified as well in the course of the proof, except the authors' statement doesn't reflect it. It remains to verify property, (5). Write $\tilde{\bm{g}}(x,t) = dt^2 + \tilde{g}_t$ as in \eqref{eq.metric.BarHanke}. Then
	\begin{equation} \label{eq.BarHankedeform.step2.H}
		\tilde{g}_0=g_0, \; \tilde{H}_t=-\tfrac12 \tr_{\tilde{g}_t} \dot{\tilde{g}}_t
	\end{equation}
	by \eqref{eq.H.slicing.BarHanke}. By \cite[Lemma 24]{BarHanke2021boundary}, as long as $\eps_1 < C^{-2}$, we have
	\begin{equation} \label{eq.BarHankedeform.step2.trg}
		\Vert \dot{\tilde{g}}_t \Vert_{g_0} \lesssim 1 \text{ and thus } \left|\tr_{\tilde{g}_t}(\dot{\tilde{g}}_t)- \tr_{g_0}(\dot{\tilde{g}}_t)\right|\le 2\sqrt{\eps_1}\|\dot{\tilde{g}}_t\|_{g_0} \lesssim \sqrt{\eps_1},
	\end{equation}
	where $\lesssim$ denotes an inequality where the LHS is bounded by the RHS times a positive constant that only depends on $g_0$, $\sff_0$, $k$, but no other data. We compute
	\begin{align*}
		-\tfrac12 \tr_{g_0} \dot{\tilde{g}}_t 
			& = \tr_{g_0} \sff_0 - \chi'(t) \tr_{g_0} (\sff_0-k) + (n-1)Ct \nonumber \\
			& = H_0 - \chi'(t)(H_0 - \tr_{g_0} k) + (n-1) Ct.
	\end{align*}
	When $\chi'(t) \geq 0$, $\tr_{g_0} k \leq H_0$ and $\chi'(t) \leq 1$ from Lemma \ref{lemm.chi.cutoff} imply that
	\begin{align} \label{eq.mean.curvature.foliation}
		-\tfrac12 \tr_{g_0} \dot{\tilde{g}}_t 
			& = (1-\chi'(t))H_0 + \chi'(t) \tr_{g_0} k + (n-1)Ct \nonumber \\
			& \geq \tr_{g_0} k + (n-1) Ct.
	\end{align}
	Otherwise, if $\chi'(t) < 0$, we use $\tr_{g_0} k \leq H_0$ twice to obtain:
	\[ -\tfrac12 \tr_{g_0} \dot{\tilde{g}}_t \geq H_0 + (n-1) Ct \geq \tr_{g_0} k + (n-1)Ct, \]
	i.e., \eqref{eq.mean.curvature.foliation} holds once again. Property (5) then follows from \eqref{eq.BarHankedeform.step2.H}, \eqref{eq.BarHankedeform.step2.trg}, \eqref{eq.mean.curvature.foliation} by choosing $\eps$ sufficiently small.
\end{proof}

\bibliography{main} 
\bibliographystyle{amsalpha}

\end{document}